\documentclass[nospthms, 11pt, final]{svjour3_arXiv}%
\smartqed
\usepackage[a4paper, headheight = 2.0cm, tmargin=1.5cm, bmargin = 1.5cm, hmargin={2.0cm,2.0cm} ]{geometry}
\usepackage{amsmath}
\usepackage{amsfonts}
\usepackage{amssymb}
\usepackage{graphicx}
\usepackage[unicode]{hyperref}
\usepackage{color}%
\providecommand{\U}[1]{\protect\rule{.1in}{.1in}}

\newtheorem{theorem}{Theorem}[section]
\newtheorem{algorithm}{Algorithm}[section]

\newtheorem{definition}{Definition}[section]
\newtheorem{example}{Example}[section]

\newtheorem{lemma}{Lemma}[section]

\newtheorem{remark}{Remark}[section]

\newenvironment{proof}[1][Proof]{\noindent\textbf{#1.} }{\ \rule{0.5em}{0.5em}}

\usepackage{wallpaper}

\journalname{}

\hypersetup{
colorlinks = true,citecolor = blue,
filecolor=black,linkcolor = red,anchorcolor = red,
pagecolor = red,
urlcolor= red
}
\allowdisplaybreaks

\begin{document}
\LRCornerWallPaper{0.225}{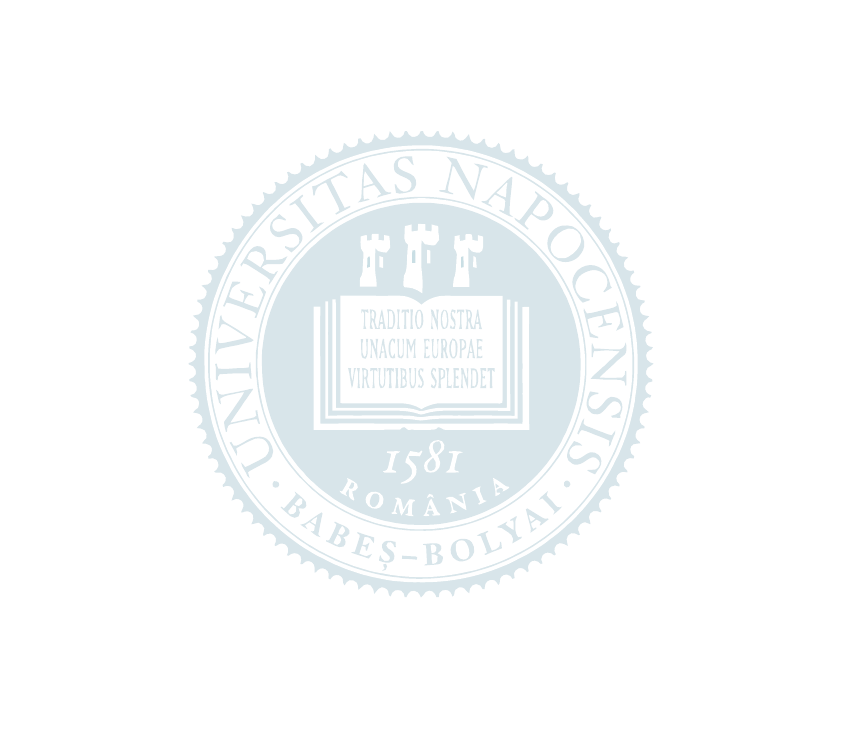}

\setcounter{tocdepth}{3}

\titlerunning{Control point based exact description}
\title{Control point based exact description of higher dimensional\\trigonometric and hyperbolic curves
and multivariate surfaces}

\author{\'{A}goston R\'{o}th}

\institute{
\'A. R\'oth \at Department of Mathematics and Computer Science, Babe\c{s}--Bolyai University, RO--400084 Cluj-Napoca, Romania \\
Tel.: +40-264-405300\\
Fax:  +40-264-591906\\
\email{agoston\_roth@yahoo.com}
}

\date{Finished on January 28, 2014 $\cdot$ Submitted to arXiv on \today}

\maketitle

\begin{abstract}
Using the normalized B-bases of vector spaces of trigonometric and hyperbolic polynomials of finite order, we specify control point configurations for the exact description of higher dimensional (rational) curves and (hybrid) multivariate surfaces determined by coordinate functions that are exclusively given either by traditional trigonometric or hyperbolic polynomials in each of their variables. The usefulness and applicability of theoretical results and proposed algorithms are illustrated by many examples that also comprise the control point based exact description of several famous curves (like epi- and hypocycloids, foliums, torus knots, Bernoulli's lemniscate, hyperbolas), surfaces (such as pure trigonometric or hybrid surfaces of revolution like tori and hyperboloids, respectively) and 3-dimensional volumes. The core of the proposed modeling methods relies on basis transformation matrices with entries that can be efficiently obtained by order elevation. Providing subdivision formulae for curves described by convex combinations of these normalized B-basis functions and control points, we also ensure the possible incorporation of all proposed techniques into today's CAD systems.

\end{abstract}

\section{Introduction}

Normalized B-bases (a comprehensive study of which can be found in
\cite{Pena1999} and references therein) are normalized totally positive bases
that imply optimal shape preserving properties for the representation of
curves described as linear combinations of control points and basis functions.
Similarly to the classical Bernstein polynomials
\[
\mathcal{B}_{n}=\left\{  \dbinom{n}{i}u^{i}\left(  1-u\right)  ^{n-i}%
:u\in\left[  0,1\right]  \right\}  _{i=0}^{n}%
\]
of degree $n\in%
\mathbb{N}
$ -- that in fact form the normalized B-basis of the vector space of
polynomials
\[
\mathbb{P}_{n}=\left\{  1,u,\ldots,u^{n}:u\in\left[  0,1\right]  \right\}
\]
of degree at most $n$ on the compact interval $\left[  0,1\right]  $, cf.
\cite{Carnicer1993} -- normalized B-bases provide shape preserving properties
like closure for the affine transformations of the control polygon, convex
hull, variation diminishing (which also implies convexity preserving of plane
control polygons), endpoint interpolation, monotonicity preserving, hodograph
and length diminishing, and a recursive corner cutting algorithm (also called
B-algorithm) that is the analogue of the de Casteljau algorithm of B\'{e}zier
curves. Among all normalized totally positive bases of a given vector space of
functions a normalized B-basis is the least variation diminishing and the
shape of the generated curve more mimics its control polygon. Important curve
design algorithms like evaluation, subdivision, degree elevation or knot
insertion are in fact corner cutting algorithms that can be treated in a
unified way by means of B-algorithms induced by B-bases.

These advantageous properties make normalized B-bases ideal blending function system
candidates for curve (and surface) modeling. Using B-basis functions, our
objective is to provide control point based exact description for higher order
derivatives of trigonometric and hyperbolic curves specified with coordinate
functions given in traditional parametric form, i.e., in vector spaces%
\begin{equation}
\mathbb{T}_{2n}^{\alpha}=\operatorname{span}\mathcal{T}_{2n}^{\alpha
}=\operatorname{span}\left\{  \cos\left(  ku\right)  ,\sin\left(  ku\right)
:u\in\left[  0,\alpha\right]  \right\}  _{k=0}^{n}
\label{truncated_Fourier_vector_space}%
\end{equation}
or%
\begin{equation}
\mathbb{H}_{2n}^{\alpha}=\operatorname{span}\mathcal{H}_{2n}^{\alpha
}=\operatorname{span}\left\{  \cosh\left(  ku\right)  ,\sinh\left(  ku\right)
:u\in\left[  0,\alpha\right]  \right\}  _{k=0}^{n},
\label{hyperbolic_vector_space}%
\end{equation}
where $\alpha$ is a fixed strictly positive shape (or design)
parameter which is either strictly less than $\pi$, or it is unbounded from above in
the trigonometric and hyperbolic cases, respectively. The obtained results will also be
extended for the control point based exact description of the rational
counterpart of these curves and of higher dimensional multivariate (rational)
surfaces that are also specified by coordinate functions given in traditional
trigonometric or hyperbolic form along of each of their variables.

\begin{remark}
From the point of view of control point based exact description of smooth
(rational) trigonometric closed curves and surfaces (i.e., when $\alpha=2\pi
$), articles \cite{RothJuhasz2010}\ and \cite{JuhaszRoth2010} already provided
control point configurations by using the so-called cyclic basis functions
introduced in \cite{RothEtAl2009}. Although this special cyclic basis of
$\mathbb{T}_{2n}^{2\pi}$ fulfills some important properties (like positivity,
normalization, cyclic variation diminishing, cyclic symmetry, singularity free
parametrization, efficient explicit formula for arbitrary order elevation), it
is not totally positive, hence it is not a B-basis, since, as it was shown in
\cite{Pena1997}, the vector space (\ref{truncated_Fourier_vector_space}) has
no normalized totally positive bases when $\alpha\geq\pi$. Therefore, by using
the B-basis of (\ref{truncated_Fourier_vector_space}), the control point based
exact description of arcs, patches or volume entities of higher dimensional
(rational) trigonometric curves and multivariate surfaces given in traditional
parametric form remained, at least for us, an interesting and challenging question.
\end{remark}

The rest of the paper is organized as follows. Section
\ref{sec:special_parametrizations} briefly recalls some basic properties of
rational B\'{e}zier curves and points out that curves described as linear
combinations of control points and B-basis functions of vector spaces
(\ref{truncated_Fourier_vector_space}) or (\ref{hyperbolic_vector_space}) are
in fact special reparametrizations of specific classes of rational B\'{e}zier
curves. This section also defines control point based (rational) trigonometric
and hyperbolic curves of finite order, briefly reviews some of their
(geometric) properties like order elevation and asymptotic behavior and at the
same time also describes their subdivision algorithm which, to the best of our
knowledge, were either totally not detailed or not described with full
generality for these type of curves in the literature. Based on multivariate
tensor products of trigonometric and hyperbolic curves, Section
\ref{sec:multivariate_surfaces} defines higher dimensional multivariate
(rational) trigonometric and hyperbolic surfaces. Section
\ref{sec:basis_transformations} provides efficient and parallely implementable
recursive formulae for those base changes that transform the normalized B-bases of vector
spaces (\ref{truncated_Fourier_vector_space}) and
(\ref{hyperbolic_vector_space}) to their corresponding canonical (traditional)
bases, respectively. Using these transformations, theorems and algorithms of
Section \ref{sec:exact_description} provide control point configurations for
the exact description of large classes of higher dimensional (rational)
trigonometric or hyperbolic curves and multivariate (hybrid) surfaces. All
examples included in this section emphasize the applicability and usefulness
of the proposed curve and surface modeling tools. Finally, Section
\ref{sec:final_remarks} closes the paper with our final remarks.

\section{Special parametrizations of a class of rational B\'{e}zier
curves\label{sec:special_parametrizations}}

Using Bernstein polynomials, a rational B\'{e}zier curve of even degree $2n$
can be described as%
\begin{equation}
\mathbf{r}_{2n}\left(  v\right)  =\frac{%
{\displaystyle\sum\limits_{i=0}^{2n}}
w_{i}\mathbf{d}_{i}B_{i}^{2n}\left(  v\right)  }{%
{\displaystyle\sum\limits_{j=0}^{2n}}
w_{j}B_{j}^{2n}\left(  v\right)  },~v\in\left[  0,1\right]
,\label{rational_Bezier_curve}%
\end{equation}
where $\left[  \mathbf{d}_{i}\right]  _{i=0}^{2n}\in\mathcal{M}_{1,2n+1}%
\left(
\mathbb{R}
^{\delta}\right)  $ is a user defined control polygon ($\delta\geq2$), while
$\left[  w_{i}\right]  _{i=0}^{2n}\in\mathcal{M}_{1,2n+1}\left(
\mathbb{R}
_{+}\right)  $ is also a user specified non-negative weight vector of rank $1$
(i.e., $\sum_{i=0}^{2n}w_{i}\neq0$).

For any fixed ratio $v\in\left[  0,1\right]  $, the recursive relations%
\begin{equation}
\left\{
\begin{array}
[c]{l}%
w_{i}^{r}\left(  v\right)  =\left(  1-v\right)  w_{i}^{r-1}\left(  v\right)
+vw_{i+1}^{r-1}\left(  v\right)  ,\\
\\
\mathbf{d}_{i}^{r}\left(  v\right)  =\left(  1-v\right)  \dfrac{w_{i}%
^{r-1}\left(  v\right)  }{w_{i}^{r}\left(  v\right)  }\mathbf{d}_{i}%
^{r-1}\left(  v\right)  +v\dfrac{w_{i+1}^{r-1}\left(  v\right)  }{w_{i}%
^{r}\left(  v\right)  }\mathbf{d}_{i+1}^{r-1}\left(  v\right)  ,~r=1,2,\ldots
,2n,~i=0,1,\ldots,2n-r
\end{array}
\right.  \label{rational_de_Casteljau}%
\end{equation}
with initial conditions%
\[
w_{i}^{0}\left(  v\right)  \equiv w_{i},~\mathbf{d}_{i}^{0}\left(  v\right)
\equiv\mathbf{d}_{i},~i=0,1,\ldots,2n
\]
define the B-algorithm (or rational de Casteljau algorithm) of the curve
(\ref{rational_Bezier_curve}) (cf. \cite{Farin2002}).

We will produce control point exact based description of trigonometric and
hyperbolic curves, therefore we need proper bases for vector spaces
(\ref{truncated_Fourier_vector_space}) and (\ref{hyperbolic_vector_space}) of
trigonometric and hyperbolic polynomials of order at most $n$ (or of degree at
most $2n$), respectively. In what follows, $\overline{\mathcal{T}}%
_{2n}^{\alpha}$ and $\overline{\mathcal{H}}_{2n}^{\alpha}$ denote the normalized B-bases
of vector spaces $\mathbb{T}_{2n}^{\alpha}$ and $\mathbb{H}_{2n}^{\alpha}$, respectively.

\subsection{Trigonometric curves and their rational
counterpart\label{sec:trigonometric_curves}}

Let $\alpha\in\left(  0,\pi\right)  $ be an arbitrarily fixed parameter and
consider the linearly reparametrized version of the B-basis
\begin{equation}
\overline{\mathcal{T}}_{2n}^{\alpha}=\left\{  T_{2n,i}^{\alpha}\left(
u\right)  :u\in\left[  0,\alpha\right]  \right\}  _{i=0}^{2n}=\left\{
t_{2n,i}^{\alpha}\sin^{2n-i}\left(  \frac{\alpha-u}{2}\right)  \sin^{i}\left(
\frac{u}{2}\right)  :u\in\left[  0,\alpha\right]  \right\}  _{i=0}^{2n}
\label{Sanchez_basis}%
\end{equation}
of order $n$ (degree $2n$) specified in \cite{Sanchez1998}, where the
non-negative normalizing coefficients%
\[
t_{2n,i}^{\alpha}=\frac{1}{\sin^{2n}\left(  \frac{\alpha}{2}\right)  }%
\sum_{r=0}^{\left\lfloor \frac{i}{2}\right\rfloor }\binom{n}{i-r}\binom
{i-r}{r}\left(  2\cos\left(  \frac{\alpha}{2}\right)  \right)  ^{i-2r}%
,~i=0,1,\ldots,2n
\]
fulfill the symmetry property%
\begin{equation}
t_{2n,i}^{\alpha}=t_{2n,2n-i}^{\alpha},~i=0,1,\ldots,n\text{.}
\label{symmetry_of_Sanchez_Reyes_constants}%
\end{equation}

\begin{definition}
[Trigonometric curves]A trigonometric curve of order $n$ (degree $2n$) can be
described as the convex combination%
\begin{equation}
\mathbf{t}_{n}^{\alpha}\left(  u\right)  =\sum_{i=0}^{2n}\mathbf{d}%
_{i}T_{2n,i}^{\alpha}\left(  u\right)  ,~u\in\left[  0,\alpha\right]  ,
\label{trigonometric_curve}%
\end{equation}
where $\left[  \mathbf{d}_{i}\right]  _{i=0}^{2n}\in\mathcal{M}_{1,2n+1}%
\left(
\mathbb{R}
^{\delta}\right)  $ defines its control polygon.
\end{definition}

As stated in Remark \ref{rem:trigonometric_reparametrization} curves of type
(\ref{trigonometric_curve}) can also be obtained as a special trigonometric
reparametrization of a class of rational B\'{e}zier curves of even degree.

\begin{remark}
[Trigonometric reparametrization]\label{rem:trigonometric_reparametrization}%
Using the function%
\begin{equation}
\left\{
\begin{array}
[c]{l}%
v:\left[  0,\alpha\right]  \rightarrow\left[  0,1\right]  ,\\
\\
v\left(  u\right)  =\dfrac{1}{2}+\dfrac{\tan\left(  \frac{u}{2}-\frac{\alpha
}{4}\right)  }{2\tan\left(  \frac{\alpha}{4}\right)  }=\dfrac{\sin\left(
\frac{u}{2}\right)  }{2\cos\left(  \frac{\alpha}{4}-\frac{u}{2}\right)
\sin\left(  \frac{\alpha}{4}\right)  }%
\end{array}
\right.  \label{trigonometric_reparametrization}%
\end{equation}
and weights%
\begin{equation}
w_{i}=\frac{t_{2n,i}^{\alpha}}{\binom{2n}{i}},~i=0,1,\ldots,2n,
\label{trigonometric_weights}%
\end{equation}
one can reparametrize the rational B\'{e}zier curve
(\ref{rational_Bezier_curve}) into the trigonometric form
(\ref{trigonometric_curve}). Indeed, one has that%
\begin{align*}
w_{i}B_{i}^{2n}\left(  v\left(  u\right)  \right)   &  =\frac{t_{2n,i}%
^{\alpha}}{\binom{2n}{i}}\binom{2n}{i}v^{i}\left(  u\right)  \left(
1-v\left(  u\right)  \right)  ^{2n-i}\\
&  =t_{2n,i}^{\alpha}\cdot\dfrac{\sin^{i}\left(  \frac{u}{2}\right)  }%
{2^{i}\cos^{i}\left(  \frac{\alpha}{4}-\frac{u}{2}\right)  \sin^{i}\left(
\frac{\alpha}{4}\right)  }\cdot\frac{\sin^{2n-i}\left(  \frac{\alpha-u}%
{2}\right)  }{2^{2n-i}\cos^{2n-i}\left(  \frac{\alpha}{4}-\frac{u}{2}\right)
\sin^{2n-i}\left(  \frac{\alpha}{4}\right)  }\\
&  =\frac{1}{2^{2n}\cos^{2n}\left(  \frac{\alpha}{4}-\frac{u}{2}\right)
\sin^{2n}\left(  \frac{\alpha}{4}\right)  }\cdot T_{2n,i}^{\alpha}\left(
u\right)
\end{align*}
for all $i=0,1,\ldots,2n$ and $u\in\left[  0,\alpha\right]  $, therefore%
\[
\mathbf{r}_{2n}\left(  v\left(  u\right)  \right)  =\frac{%
{\displaystyle\sum\limits_{i=0}^{2n}}
w_{i}\mathbf{d}_{i}B_{i}^{2n}\left(  v\left(  u\right)  \right)  }{%
{\displaystyle\sum\limits_{j=0}^{2n}}
w_{j}B_{j}^{2n}\left(  v\left(  u\right)  \right)  }=\frac{%
{\displaystyle\sum\limits_{i=0}^{2n}}
\mathbf{d}_{i}T_{2n,i}^{\alpha}\left(  u\right)  }{%
{\displaystyle\sum\limits_{j=0}^{2n}}
T_{2n,j}^{\alpha}\left(  u\right)  }=%
{\displaystyle\sum\limits_{i=0}^{2n}}
\mathbf{d}_{i}T_{2n,i}^{\alpha}\left(  u\right)  =\mathbf{t}_{n}^{\alpha
}\left(  u\right)  ,~\forall u\in\left[  0,\alpha\right]  ,
\]
since the function system (\ref{Sanchez_basis}) is normalized, i.e.,
$\sum_{j=0}^{2n}T_{2n,j}^{\alpha}\left(  u\right)  \equiv1$, $\forall
u\in\left[  0,\alpha\right]  $. Basis functions (\ref{Sanchez_basis}) and the
reparametrization function (\ref{trigonometric_reparametrization}) were
repeatedly applied in articles \cite{Sanchez1990}, \cite{Sanchez1997} and
\cite{Sanchez1998}, however the (inverse) transformation between bases
$\overline{\mathcal{T}}_{2n}^{\alpha}$ and $\mathcal{T}_{2n}^{\alpha}$ were
calculated only up to second order in \cite[p. 916]{Sanchez1998} with the aid
of a computer algebra system, moreover subdivision of such curves was
detailed only for very special control point configurations in
\cite{Sanchez1990}.
\end{remark}

\begin{remark}
[B-algorithm of trigonometric curves]\label{rem:trigonometric_subdivision}Due
to Remark \ref{rem:trigonometric_reparametrization}, the subdivision algorithm
of trigonometric curves of type (\ref{trigonometric_curve}) is a simple
corollary of the rational de Casteljau algorithm (\ref{rational_de_Casteljau}%
). One has to apply the parameter transformation
(\ref{trigonometric_reparametrization}) and initial weights
(\ref{trigonometric_weights}) in recursive formulae
(\ref{rational_de_Casteljau}). Fig. \ref{fig:subdivsion}(a) shows the steps of
this special variant of the classical rational corner cutting algorithm in
case of a third order trigonometric curve.
\end{remark}

%

\begin{figure}
[!h]
\begin{center}
\includegraphics[
height=3.243in,
width=6.2388in
]%
{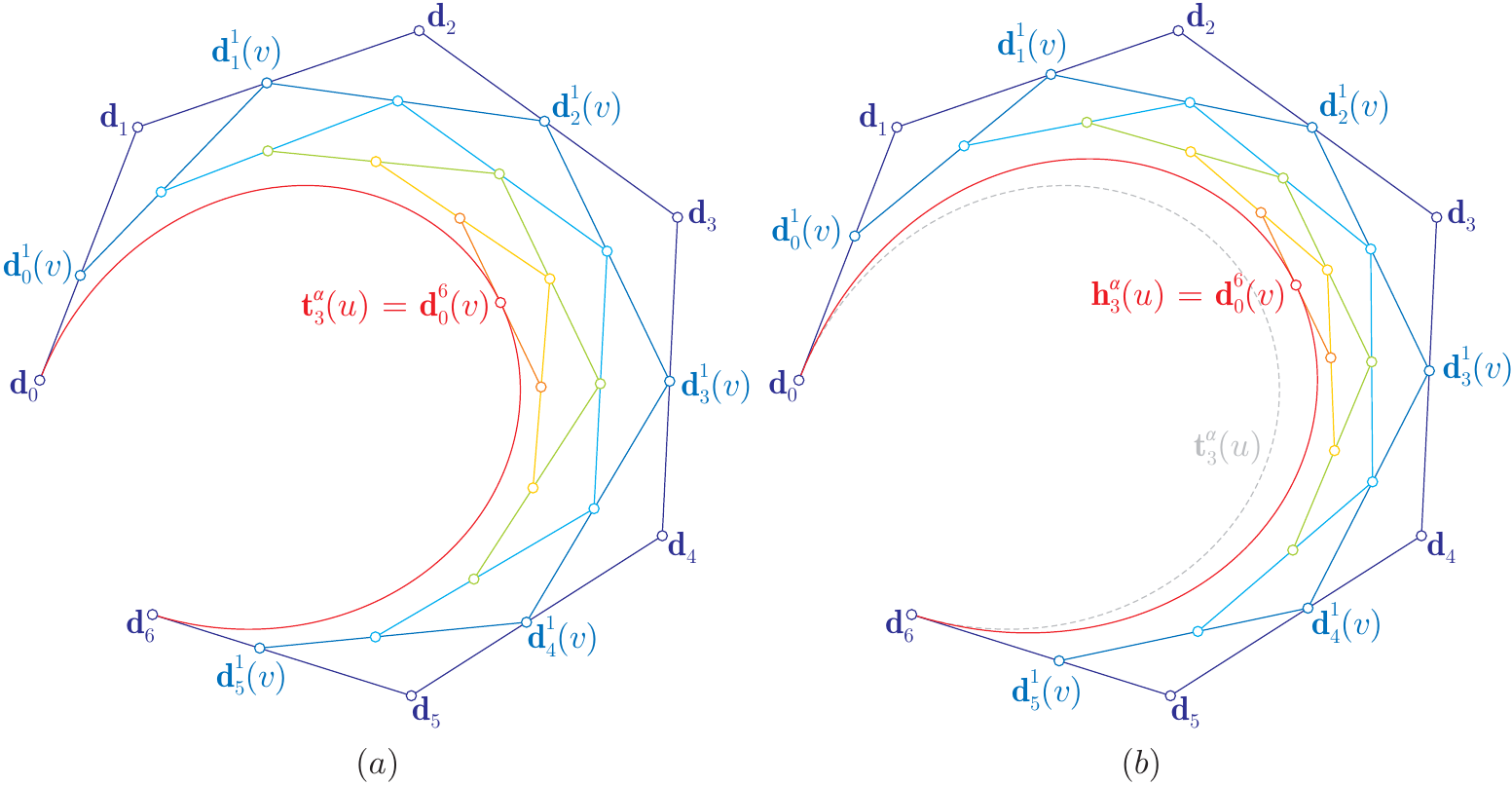}%
\caption{Consider the fixed control polygon $\left[  \mathbf{d}_{i}\right]
_{i=0}^{6}$ and the shape parameter $\alpha=\frac{\pi}{2}$ that generate third
order trigonometric and hyperbolic curves of the type
(\ref{trigonometric_curve}) and (\ref{hyperbolic_curve}), respectively. Cases
(\emph{a}) and (\emph{b}) illustrate the subdivision of these curve at the
common parameter value $u=\frac{\pi}{4}$. (The parameter value $\left.
v\left(  u\right)  \right\vert _{u=\frac{\pi}{4}}=\frac{1}{2}$ is generated by
reparametrization functions (\ref{trigonometric_reparametrization}) and
(\ref{hyperbolic_reparametrization}), respectively.)}%
\label{fig:subdivsion}%
\end{center}
\end{figure}
\qquad

The order elevation of trigonometric curves of type
(\ref{trigonometric_weights}) was also considered in \cite[Section
4.2]{Sanchez1998}. This method will be one of the main auxiliary tools used by
the present paper, therefore we briefly recall this process, by using our notations.

\begin{remark}
[Order elevation of trigonometric curves]%
\label{rem:trigonometric_order_elevation}Multiplying the curve
(\ref{trigonometric_curve}) with the first order constant function%
\[
1\equiv T_{2,0}^{\alpha}\left(  u\right)  +T_{2,1}^{\alpha}\left(  u\right)
+T_{2,2}^{\alpha}\left(  u\right)  ,~\forall u\in\left[  0,\alpha\right]
\]
and applying the product rule%
\[
T_{2n,i}^{\alpha}\left(  u\right)  T_{2m,j}^{\alpha}\left(  u\right)
=\frac{t_{2n,i}^{\alpha}t_{2m,j}^{\alpha}}{t_{2\left(  n+m\right)
,i+j}^{\alpha}}T_{2\left(  n+m\right)  ,i+j}^{\alpha}\left(  u\right)
,~\forall u\in\left[  0,\alpha\right]  ,
\]
one obtains the trigonometric curve%
\[
\mathbf{t}_{n+1}^{\alpha}\left(  u\right)  =\sum_{r=0}^{2\left(  n+1\right)
}\mathbf{e}_{r}T_{2\left(  n+1\right)  ,r}^{\alpha}\left(  u\right)
,~u\in\left[  0,\alpha\right]
\]
of order $n+1$ such that%
\[
\mathbf{t}_{n+1}^{\alpha}\left(  u\right)  =\mathbf{t}_{n}^{\alpha}\left(
u\right)  ,~\forall u\in\left[  0,\alpha\right]  ,
\]
where%
\begin{equation}
\left\{
\begin{array}
[c]{lcl}%
\mathbf{e}_{0} & = & \mathbf{d}_{0}\dfrac{t_{2n,0}^{\alpha}t_{2,0}^{\alpha}%
}{t_{2\left(  n+1\right)  ,0}^{\alpha}}=\mathbf{d}_{0}\\
&  & \\
\mathbf{e}_{1} & = & \mathbf{d}_{0}\dfrac{t_{2n,0}^{\alpha}t_{2,1}^{\alpha}%
}{t_{2\left(  n+1\right)  ,1}^{\alpha}}+\mathbf{d}_{1}\dfrac{t_{2n,1}^{\alpha
}t_{2,0}^{\alpha}}{t_{2\left(  n+1\right)  ,1}^{\alpha}},\\
&  & \\
\mathbf{e}_{r} & = & \mathbf{d}_{r-2}\dfrac{t_{2n,r-2}^{\alpha}t_{2,2}%
^{\alpha}}{t_{2\left(  n+1\right)  ,r}^{\alpha}}+\mathbf{d}_{r-1}%
\dfrac{t_{2n,r-1}^{\alpha}t_{2,1}^{\alpha}}{t_{2\left(  n+1\right)
,r}^{\alpha}}+\mathbf{d}_{r}\dfrac{t_{2n,r}^{\alpha}t_{2,0}^{\alpha}%
}{t_{2\left(  n+1\right)  ,r}^{\alpha}},~r=2,3,\ldots,2n,\\
&  & \\
\mathbf{e}_{2n+1} & = & \mathbf{d}_{2n-1}\dfrac{t_{2n,2n-1}^{\alpha}%
t_{2,2}^{\alpha}}{t_{2\left(  n+1\right)  ,2n+1}^{\alpha}}+\mathbf{d}%
_{2n}\dfrac{t_{2n,2n}^{\alpha}t_{2,1}^{\alpha}}{t_{2\left(  n+1\right)
,2n+1}^{\alpha}},\\
&  & \\
\mathbf{e}_{2\left(  n+1\right)  } & = & \mathbf{d}_{2n}\dfrac{t_{2n,2n}%
^{\alpha}t_{2,2}^{\alpha}}{t_{2\left(  n+1\right)  ,2\left(  n+1\right)
}^{\alpha}}=\mathbf{d}_{2n}.
\end{array}
\right.  \label{trigonometric_order_elevation}%
\end{equation}

\end{remark}

Due to normality of functions systems $\overline{\mathcal{T}}_{2}^{\alpha}$,
$\overline{\mathcal{T}}_{2n}^{\alpha}$ and $\overline{\mathcal{T}}_{2\left(
n+1\right)  }^{\alpha}$, one has the simple equality%
\[
1^{n+1}=\sum_{r=0}^{2\left(  n+1\right)  }T_{2\left(  n+1\right)  ,r}^{\alpha
}\left(  u\right)  =\left(  \sum_{i=0}^{2}T_{2,i}^{\alpha}\left(  u\right)
\right)  \left(  \sum_{j=0}^{2n}T_{2n,j}^{\alpha}\left(  u\right)  \right)
=1\cdot1^{n},~\forall u\in\left[  0,\alpha\right]
\]
from which follows that%
\begin{align*}
1  &  =\frac{t_{2n,0}^{\alpha}t_{2,0}^{\alpha}}{t_{2\left(  n+1\right)
,0}^{\alpha}}=\frac{t_{2n,2n}^{\alpha}t_{2,2}^{\alpha}}{t_{2\left(
n+1\right)  ,2\left(  n+1\right)  }^{\alpha}},\\
1  &  =\frac{t_{2n,0}^{\alpha}t_{2,1}^{\alpha}}{t_{2\left(  n+1\right)
,1}^{\alpha}}+\frac{t_{2n,1}^{\alpha}t_{2,0}^{\alpha}}{t_{2\left(  n+1\right)
,1}^{\alpha}}=\frac{t_{2n,2n-1}^{\alpha}t_{2,2}^{\alpha}}{t_{2\left(
n+1\right)  ,2n+1}^{\alpha}}+\frac{t_{2n,2n}^{\alpha}t_{2,1}^{\alpha}%
}{t_{2\left(  n+1\right)  ,2n+1}^{\alpha}},\\
1  &  =\frac{t_{2n,r-2}^{\alpha}t_{2,2}^{\alpha}}{t_{2\left(  n+1\right)
,r}^{\alpha}}+\frac{t_{2n,r-1}^{\alpha}t_{2,1}^{\alpha}}{t_{2\left(
n+1\right)  ,r}^{\alpha}}+\frac{t_{2n,r}^{\alpha}t_{2,0}^{\alpha}}{t_{2\left(
n+1\right)  ,r}^{\alpha}},~r=2,3,\ldots,2n,
\end{align*}
i.e., all combinations that appear in the order elevation process
(\ref{trigonometric_order_elevation}) are convex. This observation implies
that the order elevated control polygon is closer to the shape of the curve
than its original one. Therefore, repeatedly increasing the order of the
trigonometric curve (\ref{trigonometric_curve}) from $n$ to $n+z$ ($z\geq1$),
we obtain a sequence of control polygons that converges to the curve generated
by the starting control polygon. This geometric property is illustrated in
Fig. \ref{fig:trigonometric_order_elevation} and it will be essential in case
of control point based exact description of higher dimensional rational
trigonometric curves and multivariate surfaces given in traditional parametric form.%

\begin{figure}
[!h]
\begin{center}
\includegraphics[
height=3.0268in,
width=2.9551in
]%
{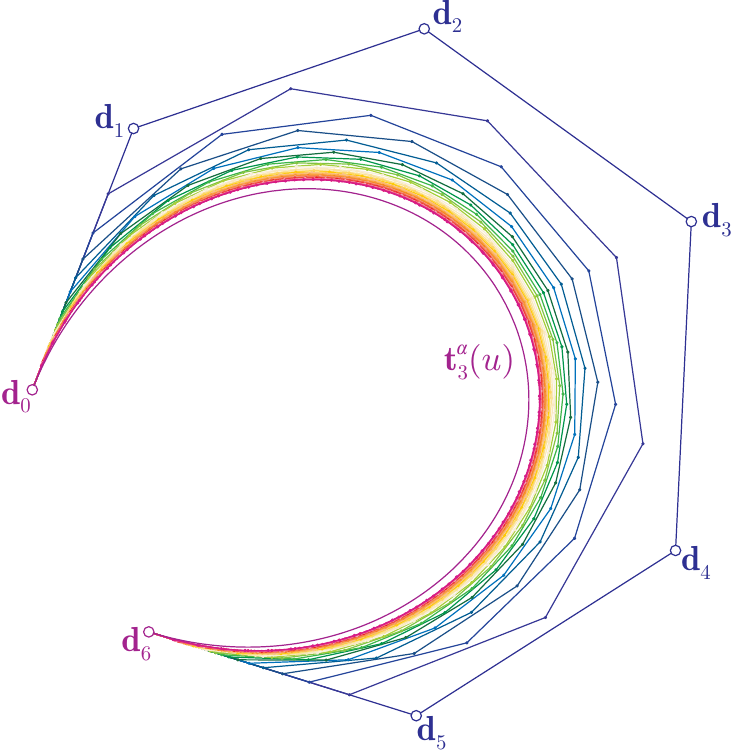}%
\caption{A plane third order trigonometric curve ($\alpha=\frac{\pi}{2}$) with
its original and order elevated control polygons which form a sequence that
converges to the curve generated by the original control polygon.}%
\label{fig:trigonometric_order_elevation}%
\end{center}
\end{figure}

\begin{remark}
[Asymptotic behavior]\label{rem:trigonometric_asymptotic_behavior}As proved in
\cite[Proposition 2.1, p. 249]{JuhaszRoth2014}, the basis $\overline
{\mathcal{T}}_{2n}^{\alpha}$ degenerates to the classical Bernstein polynomial
basis $\mathcal{B}_{2n}$ defined over the unit compact interval as the shape
parameter $\alpha$ tends to $0$ from above. In this case the trigonometric
curve (\ref{trigonometric_curve}) becomes a classical B\'{e}zier curve of
degree $2n$, while the subdivision algorithm presented in Remark
\ref{rem:trigonometric_subdivision} degenerates to the classical non-rational
de Casteljau algorithm. Fig. \ref{fig:effect_of_trigonometric_shape_parameter}
illustrates the effect of the shape parameter $\alpha\in\left(  0,\pi\right)
$ on the image of a third order trigonometric curve.
\end{remark}

%

\begin{figure}
[!h]
\begin{center}
\includegraphics[
height=3.0398in,
width=3.0952in
]%
{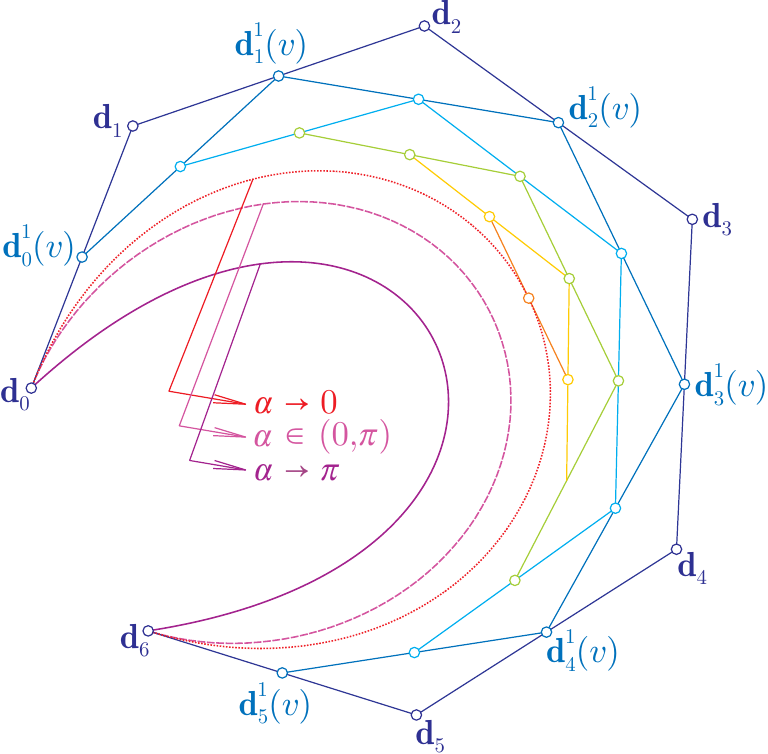}%
\caption{Effect of the design parameter $\alpha\in\left(  0,\pi\right)  $ on
the shape of a third order trigonometric curve. In the limiting case
$\alpha\rightarrow0$ the curve becomes a classical B\'{e}zier curve of degree
$6$ (as it is expected, in this special case, the B-algorithm of the
trigonometric curve degenerates to the classical corner cutting de Casteljau
algorithm, i.e., each subdivision point is determined by the same ratio along
the edges of the control polygon.)}%
\label{fig:effect_of_trigonometric_shape_parameter}%
\end{center}
\end{figure}

\begin{definition}
[Rational trigonometric curves]The non-negative weight vector $\mathbf{%
\boldsymbol{\omega}%
}=\left[  \omega_{i}\right]  _{i=0}^{2n}$ of rank $1$ associated with the
control polygon $\left[  \mathbf{d}_{i}\right]  _{i=0}^{2n}\in\mathcal{M}%
_{1,2n+1}\left(
\mathbb{R}
^{\delta}\right)  $ and the normalized linearly independent rational (or
quotient) functions%
\[
R_{2n,i}^{\alpha,\mathbf{%
\boldsymbol{\omega}%
}}\left(  u\right)  =\frac{\omega_{i}T_{2n,i}^{\alpha}\left(  u\right)  }{%
{\displaystyle\sum\limits_{j=0}^{2n}}
\omega_{j}T_{2n,j}^{\alpha}\left(  u\right)  },~u\in\left[  0,\alpha\right]
,~i=0,1,\ldots,2n
\]
define the rational counterpart%
\begin{equation}
\mathbf{t}_{n}^{\alpha,\mathbf{%
\boldsymbol{\omega}%
}}\left(  u\right)  =\sum_{i=0}^{2n}\omega_{i}\mathbf{d}_{i}R_{2n,i}%
^{\alpha,\mathbf{%
\boldsymbol{\omega}%
}}\left(  u\right)  =\frac{%
{\displaystyle\sum\limits_{i=0}^{2n}}
\omega_{i}\mathbf{d}_{i}T_{2n,i}^{\alpha}\left(  u\right)  }{%
{\displaystyle\sum\limits_{j=0}^{2n}}
\omega_{j}T_{2n,j}^{\alpha}\left(  u\right)  },~u\in\left[  0,\alpha\right]
\label{rational_trigonometric_curve}%
\end{equation}
of the trigonometric curve (\ref{trigonometric_curve}).
\end{definition}

\begin{remark}
[Pre-image of rational trigonometric curves]The rational trigonometric curve
(\ref{rational_trigonometric_curve}) can also be considered as the central
projection of the higher dimensional curve%
\begin{equation}
\mathbf{t}_{n,\mathcal{\wp}}^{\alpha,\mathbf{%
\boldsymbol{\omega}%
}}\left(  u\right)  =\sum_{i=0}^{2n}\left[
\begin{array}
[c]{c}%
\omega_{i}\mathbf{d}_{i}\\
\omega_{i}%
\end{array}
\right]  T_{2n,i}^{\alpha}\left(  u\right)  ,~u\in\left[  0,\alpha\right]
\label{trigonometric_pre_image}%
\end{equation}
in the $\delta+1$ dimensional space from the origin onto the $\delta$
dimensional hyperplane $x^{\delta+1}=1$ (assuming that the coordinates of $%
\mathbb{R}
^{\delta+1}$ are denoted by $x^{1},x^{2},\ldots,x^{\delta+1}$). The curve
(\ref{trigonometric_pre_image}) is called the pre-image of the rational curve
(\ref{rational_trigonometric_curve}), while the vector space $%
\mathbb{R}
^{\delta+1}$ is called its pre-image space. This concept will be useful in
case of control point based exact description of smooth rational trigonometric
curves given in traditional parametric form expressed in the canonical basis
$\mathcal{T}_{2n}^{\alpha}$.
\end{remark}

\subsection{Hyperbolic curves and their rational counterpart}

In this case, let $\alpha>0$ be an arbitrarily fixed parameter and consider
the B-basis%
\begin{equation}
\overline{\mathcal{H}}_{2n}^{\alpha}=\left\{  H_{2n,i}^{\alpha}\left(
u\right)  :u\in\left[  0,\alpha\right]  \right\}  _{i=0}^{2n}=\left\{
h_{2n,i}^{\alpha}\sinh^{2n-i}\left(  \frac{\alpha-u}{2}\right)  \sinh
^{i}\left(  \frac{u}{2}\right)  :u\in\left[  0,\alpha\right]  \right\}
_{i=0}^{2n} \label{Wang_basis}%
\end{equation}
of order $n$ (degree $2n$) of the vector space (\ref{hyperbolic_vector_space})
introduced in \cite{ShenWang2005}, where the non-negative normalizing coefficients%
\[
h_{2n,i}^{\alpha}=\frac{1}{\sinh^{2n}\left(  \frac{\alpha}{2}\right)  }%
\sum_{r=0}^{\left\lfloor \frac{i}{2}\right\rfloor }\binom{n}{i-r}\binom
{i-r}{r}\left(  2\cosh\left(  \frac{\alpha}{2}\right)  \right)  ^{i-2r}%
,~i=0,1,\ldots,2n
\]
fulfill the symmetry property%
\begin{equation}
h_{2n,i}^{\alpha}=h_{2n,2n-i}^{\alpha},~i=0,1,\ldots,n\text{.}%
\end{equation}

\begin{definition}
[Hyperbolic curves]The convex combination%
\begin{equation}
\mathbf{h}_{n}^{\alpha}\left(  u\right)  =\sum_{i=0}^{2n}\mathbf{d}%
_{i}H_{2n,i}^{\alpha}\left(  u\right)  ,~u\in\left[  0,\alpha\right]  ,
\label{hyperbolic_curve}%
\end{equation}
defines a hyperbolic curve of order $n$ (degree $2n$), where $\left[
\mathbf{d}_{i}\right]  _{i=0}^{2n}\in\mathcal{M}_{1,2n+1}\left(
\mathbb{R}
^{\delta}\right)  $ forms a control polygon.
\end{definition}

Similarly to Subsection \ref{sec:trigonometric_curves} it is easy to observe
that curves of type (\ref{hyperbolic_curve}) are in fact special
reparametrizations of a class of rational B\'{e}zier curves of even degree
$2n$. Instead of trigonometric sine, cosine, and tangent functions one has to
apply the hyperbolic variant of these functions, i.e., instead of parameter
transformation (\ref{trigonometric_reparametrization}) and weights
(\ref{trigonometric_weights}) one has to substitute the reparametrization
function%
\begin{equation}
\left\{
\begin{array}
[c]{l}%
v:\left[  0,\alpha\right]  \rightarrow\left[  0,1\right]  ,\\
\\
v\left(  u\right)  =\dfrac{1}{2}+\dfrac{\tanh\left(  \frac{u}{2}-\frac{\alpha
}{4}\right)  }{2\tanh\left(  \frac{\alpha}{4}\right)  }=\dfrac{\sinh\left(
\frac{u}{2}\right)  }{2\cosh\left(  \frac{\alpha}{4}-\frac{u}{2}\right)
\sinh\left(  \frac{\alpha}{4}\right)  }%
\end{array}
\right.  \label{hyperbolic_reparametrization}%
\end{equation}
and weights%
\begin{equation}
w_{i}=\frac{h_{2n,i}^{\alpha}}{\binom{2n}{i}},~i=0,1,\ldots,2n
\label{hyperbolic_weights}%
\end{equation}
into the rational B\'{e}zier curve (\ref{rational_Bezier_curve}), respectively.

Using observations similar to Remarks \ref{rem:trigonometric_subdivision},
\ref{rem:trigonometric_order_elevation} and
\ref{rem:trigonometric_asymptotic_behavior}, the subdivision, order elevation
and asymptotic behavior of hyperbolic curves of type (\ref{hyperbolic_curve})
can also be formulated. With the exception of the subdivision algorithm and
without the observation of the parameter transformation
(\ref{hyperbolic_reparametrization}) and special weight settings
(\ref{hyperbolic_weights}), the asymptotic behavior and the order elevation of
hyperbolic curves were first studied in \cite{ShenWang2005}. The steps of the
subdivision of a third order hyperbolic curve is presented in Fig.
\ref{fig:subdivsion}(\emph{b}).

The rational variant of the hyperbolic curve (\ref{hyperbolic_curve}) and its
pre-image can also be easily described.

\begin{definition}
[Rational hyperbolic curves]Consider the non-negative weight vector
$\mathbf{%
\boldsymbol{\omega}%
}=\left[  \omega_{i}\right]  _{i=0}^{2n}$ of rank $1$ associated with the
control polygon $\left[  \mathbf{d}_{i}\right]  _{i=0}^{2n}\in\mathcal{M}%
_{1,2n+1}\left(
\mathbb{R}
^{\delta}\right)  $. Normalized quotient basis functions%
\[
S_{2n,i}^{\alpha,\mathbf{%
\boldsymbol{\omega}%
}}\left(  u\right)  =\frac{\omega_{i}H_{2n,i}^{\alpha}\left(  u\right)  }{%
{\displaystyle\sum\limits_{j=0}^{2n}}
\omega_{j}H_{2n,j}^{\alpha}\left(  u\right)  },~u\in\left[  0,\alpha\right]
,~i=0,1,\ldots,2n
\]
generate the rational counterpart%
\begin{equation}
\mathbf{h}_{n}^{\alpha,\mathbf{%
\boldsymbol{\omega}%
}}\left(  u\right)  =\sum_{i=0}^{2n}\omega_{i}\mathbf{d}_{i}S_{2n,i}%
^{\alpha,\mathbf{%
\boldsymbol{\omega}%
}}\left(  u\right)  =\frac{%
{\displaystyle\sum\limits_{i=0}^{2n}}
\omega_{i}\mathbf{d}_{i}H_{2n,i}^{\alpha}\left(  u\right)  }{%
{\displaystyle\sum\limits_{j=0}^{2n}}
\omega_{j}H_{2n,j}^{\alpha}\left(  u\right)  },~u\in\left[  0,\alpha\right]
\label{rational_hyperbolic_curve}%
\end{equation}
of the hyperbolic curve (\ref{hyperbolic_curve}), the pre-image of which is%
\begin{equation}
\mathbf{h}_{n,\mathcal{\wp}}^{\alpha,\mathbf{%
\boldsymbol{\omega}%
}}\left(  u\right)  =\sum_{i=0}^{2n}\left[
\begin{array}
[c]{c}%
\omega_{i}\mathbf{d}_{i}\\
\omega_{i}%
\end{array}
\right]  H_{2n,i}^{\alpha}\left(  u\right)  ,~u\in\left[  0,\alpha\right]  .
\label{hyperbolic_pre_image}%
\end{equation}

\end{definition}

\section{(Hybrid) (rational) trigonometric and hyperbolic multivariate
surfaces\label{sec:multivariate_surfaces}}

By means of tensor products of curves of type (\ref{trigonometric_curve}) and
(\ref{hyperbolic_curve}) one can introduce the following multivariate higher
dimensional surface modeling tools. Let $\delta\geq2$ and $\kappa\geq0$
arbitrarily fixed natural numbers and consider the also fixed vector
$\mathbf{n}=\left[  n_{j}\right]  _{j=1}^{\delta}$ of orders, where $n_{j}%
\geq1$ for all $j=1,2,\ldots,\delta$.

\begin{definition}
[Trigonometric surfaces and their rational counterpart]%
\label{trigonometric_surface_definitions}Let
\[
\mathbf{%
\boldsymbol{\alpha}%
=}\left[  \alpha_{j}\right]  _{j=1}^{\delta}\in\times_{j=1}^{\delta}\left(
0,\pi\right)
\]
be a fixed vector of shape parameters and consider the multidimensional
control grid
\begin{equation}
\left[  \mathbf{d}_{i_{1},i_{2},\ldots,i_{\delta}}\right]  _{i_{1}%
=0,i_{2}=0,\ldots,i_{\delta}=0}^{2n_{1},2n_{2},\ldots,2n_{\delta}}%
\in\mathcal{M}_{2n_{1}+1,2n_{2}+1,\ldots2n_{\delta}+1}\left(
\mathbb{R}
^{\delta+\kappa}\right)  . \label{trigonometric_control_grid}%
\end{equation}
The multivariate surface%
\begin{align}
\mathbf{t}_{\mathbf{n}}^{\mathbf{%
\boldsymbol{\alpha}%
}}\left(  \mathbf{u}\right)   &  =\mathbf{t}_{n_{1},n_{2},\ldots,n_{\delta}%
}^{\alpha_{1},\alpha_{2},\ldots,\alpha_{\delta}}\left(  u_{1},u_{2}%
,\ldots,u_{\delta}\right) \label{trigonometric_surface}\\
&  =\sum_{i_{1}=0}^{2n_{1}}\sum_{i_{2}=0}^{2n_{2}}\cdots\sum_{i_{\delta}%
=0}^{2n_{\delta}}\mathbf{d}_{i_{1},i_{2},\ldots,i_{\delta}}T_{2n_{1},i_{1}%
}^{\alpha_{1}}\left(  u_{1}\right)  T_{2n_{2},i_{2}}^{\alpha_{2}}\left(
u_{2}\right)  \cdot\ldots\cdot T_{2n_{\delta},i_{\delta}}^{\alpha_{\delta}%
}\left(  u_{\delta}\right)  ,~\mathbf{u}=\left[  u_{j}\right]  _{j=1}^{\delta
}\in\times_{j=1}^{\delta}\left[  0,\alpha_{j}\right] \nonumber
\end{align}
is called $\delta$-variate trigonometric surface of order $\mathbf{n}$ taking
values in $%
\mathbb{R}
^{\delta+\kappa}$. Assigning the non-negative multidimensional weight matrix
\begin{equation}
\boldsymbol{\Omega}=\left[  \omega_{i_{1},i_{2},\ldots,i_{\delta}}\right]  _{i_{1}%
=0,i_{2}=0,\ldots,i_{\delta}=0}^{2n_{1},2n_{2},\ldots,2n_{\delta}}%
\in\mathcal{M}_{2n_{1}+1,2n_{2}+1,\ldots2n_{\delta}+1}\left(
\mathbb{R}
_{+}\right)  \label{trigonometric_weight_grid}%
\end{equation}
of rank at least $1$ to the control grid (\ref{trigonometric_control_grid}),
one obtains the $\delta$-variate rational trigonometric surface%
\begin{align}
\mathbf{t}_{\mathbf{n}}^{\mathbf{%
\boldsymbol{\alpha}%
},\boldsymbol{\Omega}}\left(  \mathbf{u}\right)   &  =\mathbf{t}_{n_{1},n_{2}%
,\ldots,n_{\delta}}^{\alpha_{1},\alpha_{2},\ldots,\alpha_{\delta},\boldsymbol{\Omega}
}\left(  u_{1},u_{2},\ldots,u_{\delta}\right)
\label{trigonometric_rational_surface}\\
&  =\frac{%
{\displaystyle\sum\limits_{i_{1}=0}^{2n_{1}}}
{\displaystyle\sum\limits_{i_{2}=0}^{2n_{2}}}
\cdots%
{\displaystyle\sum\limits_{i_{\delta}=0}^{2n_{\delta}}}
\omega_{i_{1},i_{2},\ldots,i_{\delta}}\mathbf{d}_{i_{1},i_{2},\ldots
,i_{\delta}}T_{2n_{1},i_{1}}^{\alpha_{1}}\left(  u_{1}\right)  T_{2n_{2}%
,i_{2}}^{\alpha_{2}}\left(  u_{2}\right)  \cdot\ldots\cdot T_{2n_{\delta
},i_{\delta}}^{\alpha_{\delta}}\left(  u_{\delta}\right)  }{%
{\displaystyle\sum\limits_{j_{1}=0}^{2n_{1}}}
{\displaystyle\sum\limits_{j_{2}=0}^{2n_{2}}}
\cdots%
{\displaystyle\sum\limits_{j_{\delta}=0}^{2n_{\delta}}}
\omega_{j_{1},j_{2},\ldots,j_{\delta}}T_{2n_{1},j_{1}}^{\alpha_{1}}\left(
u_{1}\right)  T_{2n_{2},j_{2}}^{\alpha_{2}}\left(  u_{2}\right)  \cdot
\ldots\cdot T_{2n_{\delta},j_{\delta}}^{\alpha_{\delta}}\left(  u_{\delta
}\right)  }\nonumber
\end{align}
of the same order, which is the central projection of the pre-image%
\begin{align}
\mathbf{t}_{\mathbf{n},\wp}^{\mathbf{%
\boldsymbol{\alpha}%
},\boldsymbol{\Omega}}\left(  \mathbf{u}\right)   &  =\mathbf{t}_{n_{1},n_{2}%
,\ldots,n_{\delta},\wp}^{\alpha_{1},\alpha_{2},\ldots,\alpha_{\delta},\boldsymbol{\Omega}
}\left(  u_{1},u_{2},\ldots,u_{\delta}\right)  \label{trigonometric_preimage}%
\\
&  =\sum_{i_{1}=0}^{2n_{1}}\sum_{i_{2}=0}^{2n_{2}}\cdots\sum_{i_{\delta}%
=0}^{2n_{\delta}}\left[
\begin{array}
[c]{c}%
\omega_{i_{1},i_{2},\ldots,i_{\delta}}\mathbf{d}_{i_{1},i_{2},\ldots
,i_{\delta}}\\
\omega_{i_{1},i_{2},\ldots,i_{\delta}}%
\end{array}
\right]  T_{2n_{1},i_{1}}^{\alpha_{1}}\left(  u_{1}\right)  T_{2n_{2},i_{2}%
}^{\alpha_{2}}\left(  u_{2}\right)  \cdot\ldots\cdot T_{2n_{\delta},i_{\delta
}}^{\alpha_{\delta}}\left(  u_{\delta}\right) \nonumber
\end{align}
in $%
\mathbb{R}
^{\delta+\kappa+1}$ from its origin onto the $\delta+\kappa$ dimensional
hyperplane $x^{\delta+k+1}=1$ (provided that the coordinates of $%
\mathbb{R}
^{\delta+\kappa+1}$ are labeled by $x^{1},x^{2},\ldots,x^{\delta+k+1}$).
\end{definition}

\begin{remark}
[$3$-dimensional $2$-variate trigonometric surfaces]The simplest variant of
multivariate surfaces introduced in Definition
\ref{trigonometric_surface_definitions} corresponds to $\delta=2$ and
$\kappa=1$, when the $2$-variate trigonometric surface
(\ref{trigonometric_surface}) is a $3$-dimensional traditional tensor product
surface of curves of the type (\ref{trigonometric_curve}). In this special
case, the grid (\ref{trigonometric_control_grid}) of control points\ and the
multidimensional weight matrix (\ref{trigonometric_weight_grid}) degenerate to
a traditional control net and rectangular weight matrix, respectively.
\end{remark}

\begin{remark}
[$3$-dimensional trigonometric volumes]Using settings $\delta=3$ and
$\kappa=0$, Definition \ref{trigonometric_surface_definitions} describes
$3$-dimensional volumes (solids) by means of $3$-variate tensor product of
curves of the type (\ref{trigonometric_curve}).
\end{remark}

\begin{definition}
[Hyperbolic surfaces and their rational counterpart]%
\label{trigonometric_surface_definitions copy(1)}Let
\[
\mathbf{%
\boldsymbol{\alpha}%
=}\left[  \alpha_{j}\right]  _{j=1}^{\delta}\in\times_{j=1}^{\delta}\left(
0,+\infty\right)
\]
be a fixed vector of shape parameters and consider the non-negative
multidimensional weight matrix
\[
\boldsymbol{\Omega}=\left[  \omega_{i_{1},i_{2},\ldots,i_{\delta}}\right]  _{i_{1}%
=0,i_{2}=0,\ldots,i_{\delta}=0}^{2n_{1},2n_{2},\ldots,2n_{\delta}}%
\in\mathcal{M}_{2n_{1}+1,2n_{2}+1,\ldots2n_{\delta}+1}\left(
\mathbb{R}
_{+}\right)
\]
(of rank at least $1$) associated with the control grid
\[
\left[  \mathbf{d}_{i_{1},i_{2},\ldots,i_{\delta}}\right]  _{i_{1}%
=0,i_{2}=0,\ldots,i_{\delta}=0}^{2n_{1},2n_{2},\ldots,2n_{\delta}}%
\in\mathcal{M}_{2n_{1}+1,2n_{2}+1,\ldots2n_{\delta}+1}\left(
\mathbb{R}
^{\delta+\kappa}\right)  .
\]
The multivariate hyperbolic surface of order $\mathbf{n}$, its rational
counterpart, and the pre-image of the rational variant are%
\begin{align}
\mathbf{h}_{\mathbf{n}}^{\mathbf{%
\boldsymbol{\alpha}%
}}\left(  \mathbf{u}\right)   &  =\mathbf{h}_{n_{1},n_{2},\ldots,n_{\delta}%
}^{\alpha_{1},\alpha_{2},\ldots,\alpha_{\delta}}\left(  u_{1},u_{2}%
,\ldots,u_{\delta}\right) \label{hyperbolic_surface}\\
&  =\sum_{i_{1}=0}^{2n_{1}}\sum_{i_{2}=0}^{2n_{2}}\cdots\sum_{i_{\delta}%
=0}^{2n_{\delta}}\mathbf{d}_{i_{1},i_{2},\ldots,i_{\delta}}H_{2n_{1},i_{1}%
}^{\alpha_{1}}\left(  u_{1}\right)  H_{2n_{2},i_{2}}^{\alpha_{2}}\left(
u_{2}\right)  \cdot\ldots\cdot H_{2n_{\delta},i_{\delta}}^{\alpha_{\delta}%
}\left(  u_{\delta}\right)  ,\nonumber\\
\mathbf{u}  &  =\left[  u_{j}\right]  _{j=1}^{\delta}\in\times_{j=1}^{\delta
}\left[  0,\alpha_{j}\right]  ,\nonumber\\
& \nonumber\\
\mathbf{h}_{\mathbf{n}}^{\mathbf{%
\boldsymbol{\alpha}%
},\boldsymbol{\Omega}}\left(  \mathbf{u}\right)   &  =\mathbf{h}_{n_{1},n_{2}%
,\ldots,n_{\delta}}^{\alpha_{1},\alpha_{2},\ldots,\alpha_{\delta},\boldsymbol{\Omega}
}\left(  u_{1},u_{2},\ldots,u_{\delta}\right)
\label{rational_hyperbolic_surface}\\
&  =\frac{%
{\displaystyle\sum\limits_{i_{1}=0}^{2n_{1}}}
{\displaystyle\sum\limits_{i_{2}=0}^{2n_{2}}}
\cdots%
{\displaystyle\sum\limits_{i_{\delta}=0}^{2n_{\delta}}}
\omega_{i_{1},i_{2},\ldots,i_{\delta}}\mathbf{d}_{i_{1},i_{2},\ldots
,i_{\delta}}H_{2n_{1},i_{1}}^{\alpha_{1}}\left(  u_{1}\right)  H_{2n_{2}%
,i_{2}}^{\alpha_{2}}\left(  u_{2}\right)  \cdot\ldots\cdot H_{2n_{\delta
},i_{\delta}}^{\alpha_{\delta}}\left(  u_{\delta}\right)  }{%
{\displaystyle\sum\limits_{j_{1}=0}^{2n_{1}}}
{\displaystyle\sum\limits_{j_{2}=0}^{2n_{2}}}
\cdots%
{\displaystyle\sum\limits_{j_{\delta}=0}^{2n_{\delta}}}
\omega_{j_{1},j_{2},\ldots,j_{\delta}}H_{2n_{1},j_{1}}^{\alpha_{1}}\left(
u_{1}\right)  H_{2n_{2},j_{2}}^{\alpha_{2}}\left(  u_{2}\right)  \cdot
\ldots\cdot H_{2n_{\delta},j_{\delta}}^{\alpha_{\delta}}\left(  u_{\delta
}\right)  }\nonumber
\end{align}
and%
\begin{align*}
\mathbf{h}_{\mathbf{n},\wp}^{\mathbf{%
\boldsymbol{\alpha}%
},\boldsymbol{\Omega}}\left(  \mathbf{u}\right)   &  =\mathbf{h}_{n_{1},n_{2}%
,\ldots,n_{\delta},\wp}^{\alpha_{1},\alpha_{2},\ldots,\alpha_{\delta},\boldsymbol{\Omega}
}\left(  u_{1},u_{2},\ldots,u_{\delta}\right) \\
&  =\sum_{i_{1}=0}^{2n_{1}}\sum_{i_{2}=0}^{2n_{2}}\cdots\sum_{i_{\delta}%
=0}^{2n_{\delta}}\left[
\begin{array}
[c]{c}%
\omega_{i_{1},i_{2},\ldots,i_{\delta}}\mathbf{d}_{i_{1},i_{2},\ldots
,i_{\delta}}\\
\omega_{i_{1},i_{2},\ldots,i_{\delta}}%
\end{array}
\right]  H_{2n_{1},i_{1}}^{\alpha_{1}}\left(  u_{1}\right)  H_{2n_{2},i_{2}%
}^{\alpha_{2}}\left(  u_{2}\right)  \cdot\ldots\cdot H_{2n_{\delta},i_{\delta
}}^{\alpha_{\delta}}\left(  u_{\delta}\right)  ,
\end{align*}
respectively.
\end{definition}

\begin{remark}
[Hybrid multivariate surfaces]Naturally, one can\ also mix the trigonometric
or hyperbolic type of B-basis functions in directions $\left[  u_{j}\right]
_{j=1}^{\delta}$, i.e., one can also define higher dimensional hybrid
multivariate (rational) surfaces.
\end{remark}

\section{ Basis transformations\label{sec:basis_transformations}}

We are going to derive recursive formulae for the transformation of B-bases
$\overline{\mathcal{T}}_{2n}^{\alpha}$ and $\overline{\mathcal{H}}%
_{2n}^{\alpha}$ to the canonical bases $\mathcal{T}_{2n}^{\alpha}$ and
$\mathcal{H}_{2n}^{\alpha}$ of the vector spaces $\mathbb{T}_{2n}^{\alpha}$
and $\mathbb{H}_{2n}^{\alpha}$, respectively.

\subsection{The trigonometric
case\label{sec:trigonometric_basis_transformation}}

Let $k\in\left\{  0,1,\ldots,n\right\}  $ be an arbitrarily fixed natural
number. Assume that the unique representations of trigonometric functions
$\sin\left(  ku\right)  $ and $\cos\left(  ku\right)  $ in the basis
(\ref{Sanchez_basis}) of order $n$ are%
\begin{equation}
\sin\left(  ku\right)  =\sum_{i=0}^{2n}\lambda_{k,i}^{n}T_{2n,i}^{\alpha
}\left(  u\right)  ,~u\in\left[  0,\alpha\right]  \label{sine_form}%
\end{equation}
and%
\begin{equation}
\cos\left(  ku\right)  =\sum_{i=0}^{2n}\mu_{k,i}^{n}T_{2n,i}^{\alpha}\left(
u\right)  ,~u\in\left[  0,\alpha\right]  , \label{cosine_form}%
\end{equation}
respectively, where coefficients $\left\{  \lambda_{k,i}^{n}\right\}
_{i=0}^{2n}$ and $\left\{  \mu_{k,i}^{n}\right\}  _{i=0}^{2n}$ are unique real
numbers. The basis transformation from the first order B-basis $\overline
{\mathcal{T}}_{2}^{\alpha}$ to the first order trigonometric canonical basis
$\mathcal{T}_{2}^{\alpha}$ can be expressed in the matrix form%
\[
\left[
\begin{array}
[c]{c}%
1\\
\sin\left(  u\right) \\
\cos\left(  u\right)
\end{array}
\right]  =\left[
\begin{array}
[c]{ccc}%
\mu_{0,0}^{1} & \mu_{0,1}^{1} & \mu_{0,2}^{1}\\
\lambda_{1,0}^{1} & \lambda_{1,1}^{1} & \lambda_{1,2}^{1}\\
\mu_{1,0}^{1} & \mu_{1,1}^{1} & \mu_{1,2}^{1}%
\end{array}
\right]  \left[
\begin{array}
[c]{c}%
T_{2,0}^{\alpha}\left(  u\right) \\
T_{2,1}^{\alpha}\left(  u\right) \\
T_{2,2}^{\alpha}\left(  u\right)
\end{array}
\right]  ,~\forall u\in\left[  0,\alpha\right]  ,
\]
where%
\begin{equation}
\left\{
\begin{array}
[c]{l}%
\mu_{0,0}^{1}=\mu_{0,1}^{1}=\mu_{0,2}^{1}=1,\\
\\
\lambda_{1,0}^{1}=0,~\lambda_{1,1}^{1}=\tan\left(  \frac{\alpha}{2}\right)
,~\lambda_{1,2}^{1}=\sin\left(  \alpha\right)  ,\\
\\
\mu_{1,0}^{1}=\mu_{1,1}^{1}=1,~\mu_{1,2}^{1}=\cos\left(  \alpha\right)  .
\end{array}
\right.  \label{trigonometric_initial_conditions}%
\end{equation}

Using initial conditions (\ref{trigonometric_initial_conditions}), our
objective is to derive recursive formulae for the matrix elements of the
linear transformation that changes the higher order B-basis $\overline
{\mathcal{T}}_{2\left(  n+1\right)  }^{\alpha}$ to the canonical trigonometric
basis $\mathcal{T}_{2\left(  n+1\right)  }^{\alpha}$.

Performing order elevation on functions (\ref{sine_form}) and
(\ref{cosine_form}), one obtains that%
\[
\sin\left(  ku\right)  =\sum_{r=0}^{2\left(  n+1\right)  }\lambda_{k,r}%
^{n+1}T_{2\left(  n+1\right)  ,r}^{\alpha}\left(  u\right)
\]
and%
\[
\cos\left(  ku\right)  =\sum_{r=0}^{2\left(  n+1\right)  }\mu_{k,r}%
^{n+1}T_{2\left(  n+1\right)  ,r}^{\alpha}\left(  u\right)  ,
\]
where%
\begin{align*}
\lambda_{k,0}^{n+1}  &  =\lambda_{k,0}^{n},\\
\lambda_{k,1}^{n+1}  &  =\lambda_{k,0}^{n}\frac{t_{2n,0}^{\alpha}%
t_{2,1}^{\alpha}}{t_{2\left(  n+1\right)  ,1}^{\alpha}}+\lambda_{k,1}^{n}%
\frac{t_{2n,1}^{\alpha}t_{2,0}^{\alpha}}{t_{2\left(  n+1\right)  ,1}^{\alpha}%
},\\
\lambda_{k,r}^{n+1}  &  =\lambda_{k,r-2}^{n}\frac{t_{2n,r-2}^{\alpha}%
t_{2,2}^{\alpha}}{t_{2\left(  n+1\right)  ,r}^{\alpha}}+\lambda_{k,r-1}%
^{n}\frac{t_{2n,r-1}^{\alpha}t_{2,1}^{\alpha}}{t_{2\left(  n+1\right)
,r}^{\alpha}}+\lambda_{k,r}^{n}\frac{t_{2n,r}^{\alpha}t_{2,0}^{\alpha}%
}{t_{2\left(  n+1\right)  ,r}^{\alpha}},~r=2,3,\ldots,2n,\\
\lambda_{k,2n+1}^{n+1}  &  =\lambda_{k,2n-1}^{n}\frac{t_{2n,2n-1}^{\alpha
}t_{2,2}^{\alpha}}{t_{2\left(  n+1\right)  ,2n+1}^{\alpha}}+\lambda_{k,2n}%
^{n}\frac{t_{2n,2n}^{\alpha}t_{2,1}^{\alpha}}{t_{2\left(  n+1\right)
,2n+1}^{\alpha}},\\
\lambda_{k,2\left(  n+1\right)  }^{n+1}  &  =\lambda_{k,2n}^{n}%
\end{align*}
and%
\begin{align*}
\mu_{k,0}^{n+1}  &  =\mu_{k,0}^{n},\\
\mu_{k,1}^{n+1}  &  =\mu_{k,0}^{n}\frac{t_{2n,0}^{\alpha}t_{2,1}^{\alpha}%
}{t_{2\left(  n+1\right)  ,1}^{\alpha}}+\mu_{k,1}^{n}\frac{t_{2n,1}^{\alpha
}t_{2,0}^{\alpha}}{t_{2\left(  n+1\right)  ,1}^{\alpha}},\\
\mu_{k,r}^{n+1}  &  =\mu_{k,r-2}^{n}\frac{t_{2n,r-2}^{\alpha}t_{2,2}^{\alpha}%
}{t_{2\left(  n+1\right)  ,r}^{\alpha}}+\mu_{k,r-1}^{n}\frac{t_{2n,r-1}%
^{\alpha}t_{2,1}^{\alpha}}{t_{2\left(  n+1\right)  ,r}^{\alpha}}+\mu_{k,r}%
^{n}\frac{t_{2n,r}^{\alpha}t_{2,0}^{\alpha}}{t_{2\left(  n+1\right)
,r}^{\alpha}},~r=2,3,\ldots,2n,\\
\mu_{k,2n+1}^{n+1}  &  =\mu_{k,2n-1}^{n}\frac{t_{2n,2n-1}^{\alpha}%
t_{2,2}^{\alpha}}{t_{2\left(  n+1\right)  ,2n+1}^{\alpha}}+\mu_{k,2n}^{n}%
\frac{t_{2n,2n}^{\alpha}t_{2,1}^{\alpha}}{t_{2\left(  n+1\right)
,2n+1}^{\alpha}},\\
\mu_{k,2\left(  n+1\right)  }^{n+1}  &  =\mu_{k,2n}^{n},
\end{align*}
respectively. Moreover, due to initial conditions
(\ref{trigonometric_initial_conditions}) and simple trigonometric identities
\begin{align}
\sin\left(  a+b\right)   &  =\sin\left(  a\right)  \cos\left(  b\right)
+\cos\left(  a\right)  \sin\left(  b\right)  ,\label{sin_of_sum}\\
\cos\left(  a+b\right)   &  =\cos\left(  a\right)  \cos\left(  b\right)
-\sin\left(  a\right)  \sin\left(  b\right)  , \label{cos_of_sum}%
\end{align}
one has that%
\begin{align*}
\sin\left(  \left(  n+1\right)  u\right)   &  =\left(  \sum_{i=0}^{2n}%
\lambda_{n,i}^{n}T_{2n,i}^{\alpha}\left(  u\right)  \right)  \left(
\sum_{j=0}^{2}\mu_{1,j}^{1}T_{2,j}^{\alpha}\left(  u\right)  \right)  +\left(
\sum_{i=0}^{2n}\mu_{n,i}^{n}T_{2n,i}^{\alpha}\left(  u\right)  \right)
\left(  \sum_{j=0}^{2}\lambda_{1,j}^{1}T_{2,j}^{\alpha}\left(  u\right)
\right) \\
&  =\sum_{r=0}^{2\left(  n+1\right)  }\lambda_{n+1,r}^{n+1}T_{2\left(
n+1\right)  ,r}^{\alpha}\left(  u\right)  ,\\
\cos\left(  \left(  n+1\right)  u\right)   &  =\left(  \sum_{i=0}^{2n}%
\mu_{n,i}^{n}T_{2n,i}^{\alpha}\left(  u\right)  \right)  \left(  \sum
_{j=0}^{2}\mu_{1,j}^{1}T_{2,j}^{\alpha}\left(  u\right)  \right)  -\left(
\sum_{i=0}^{2n}\lambda_{n,i}^{n}T_{2n,i}^{\alpha}\left(  u\right)  \right)
\left(  \sum_{j=0}^{2}\lambda_{1,j}^{1}T_{2,j}^{\alpha}\left(  u\right)
\right) \\
&  =\sum_{r=0}^{2\left(  n+1\right)  }\mu_{n+1,r}^{n+1}T_{2\left(  n+1\right)
,r}^{\alpha}\left(  u\right)  ,
\end{align*}
where%
\begin{align*}
\lambda_{n+1,0}^{n+1}=  &  \lambda_{n,0}^{n}\mu_{1,0}^{1}+\mu_{n,0}^{n}%
\lambda_{1,0}^{1},\\
\lambda_{n+1,1}^{n+1}=  &  \left(  \lambda_{n,0}^{n}\mu_{1,1}^{1}+\mu
_{n,0}^{n}\lambda_{1,1}^{1}\right)  \frac{t_{2n,0}^{\alpha}t_{2,1}^{\alpha}%
}{t_{2\left(  n+1\right)  ,1}^{\alpha}}+\left(  \lambda_{n,1}^{n}\mu_{1,0}%
^{1}+\mu_{n,1}^{n}\lambda_{1,0}^{1}\right)  \frac{t_{2n,1}^{\alpha}%
t_{2,0}^{\alpha}}{t_{2\left(  n+1\right)  ,1}^{\alpha}},\\
\lambda_{n+1,r}^{n+1}=  &  \left(  \lambda_{n,r-2}^{n}\mu_{1,2}^{1}%
+\mu_{n,r-2}^{n}\lambda_{1,2}^{1}\right)  \frac{t_{2n,r-2}^{\alpha}%
t_{2,2}^{\alpha}}{t_{2\left(  n+1\right)  ,r}^{\alpha}}+\left(  \lambda
_{n,r-1}^{n}\mu_{1,1}^{1}+\mu_{n,r-1}^{n}\lambda_{1,1}^{1}\right)
\frac{t_{2n,r-1}^{\alpha}t_{2,1}^{\alpha}}{t_{2\left(  n+1\right)  ,r}%
^{\alpha}}\\
&  +\left(  \lambda_{n,r}^{n}\mu_{1,0}^{1}+\mu_{n,r}^{n}\lambda_{1,0}%
^{1}\right)  \frac{t_{2n,r}^{\alpha}t_{2,0}^{\alpha}}{t_{2\left(  n+1\right)
,r}^{\alpha}},~r=2,3,\ldots,2n,\\
\lambda_{n+1,2n+1}^{n+1}=  &  \left(  \lambda_{n,2n-1}^{n}\mu_{1,2}^{1}%
+\mu_{n,2n-1}^{n}\lambda_{1,2}^{1}\right)  \frac{t_{2n,2n-1}^{\alpha}%
t_{2,2}^{\alpha}}{t_{2\left(  n+1\right)  ,2n+1}^{\alpha}}+\left(
\lambda_{n,2n}^{n}\mu_{1,1}^{1}+\mu_{n,2n}^{n}\lambda_{1,1}^{1}\right)
\frac{t_{2n,2n}^{\alpha}t_{2,1}^{\alpha}}{t_{2\left(  n+1\right)
,2n+1}^{\alpha}},\\
\lambda_{n+1,2\left(  n+1\right)  }^{n+1}=  &  \lambda_{n,2n}^{n}\mu_{1,2}%
^{1}+\mu_{n,2n}^{n}\lambda_{1,2}^{1}%
\end{align*}
and%
\begin{align*}
\mu_{n+1,0}^{n+1}=  &  \mu_{n,0}^{n}\mu_{1,0}^{1}-\lambda_{n,0}^{n}%
\lambda_{1,0}^{1},\\
\mu_{n+1,1}^{n+1}=  &  \left(  \mu_{n,0}^{n}\mu_{1,1}^{1}-\lambda_{n,0}%
^{n}\lambda_{1,1}^{1}\right)  \frac{t_{2n,0}^{\alpha}t_{2,1}^{\alpha}%
}{t_{2\left(  n+1\right)  ,1}^{\alpha}}+\left(  \mu_{n,1}^{n}\mu_{1,0}%
^{1}-\lambda_{n,1}^{n}\lambda_{1,0}^{1}\right)  \frac{t_{2n,1}^{\alpha}%
t_{2,0}^{\alpha}}{t_{2\left(  n+1\right)  ,1}^{\alpha}},\\
\mu_{n+1,r}^{n+1}=  &  \left(  \mu_{n,r-2}^{n}\mu_{1,2}^{1}-\lambda
_{n,r-2}^{n}\lambda_{1,2}^{1}\right)  \frac{t_{2n,r-2}^{\alpha}t_{2,2}%
^{\alpha}}{t_{2\left(  n+1\right)  ,r}^{\alpha}}+\left(  \mu_{n,r-1}^{n}%
\mu_{1,1}^{1}-\lambda_{n,r-1}^{n}\lambda_{1,1}^{1}\right)  \frac
{t_{2n,r-1}^{\alpha}t_{2,1}^{\alpha}}{t_{2\left(  n+1\right)  ,r}^{\alpha}}\\
&  +\left(  \mu_{n,r}^{n}\mu_{1,0}^{1}-\lambda_{n,r}^{n}\lambda_{1,0}%
^{1}\right)  \frac{t_{2n,r}^{\alpha}t_{2,0}^{\alpha}}{t_{2\left(  n+1\right)
,r}^{\alpha}},~r=2,3,\ldots,2n,\\
\mu_{n+1,2n+1}^{n+1}=  &  \left(  \mu_{n,2n-1}^{n}\mu_{1,2}^{1}-\lambda
_{n,2n-1}^{n}\lambda_{1,2}^{1}\right)  \frac{t_{2n,2n-1}^{\alpha}%
t_{2,2}^{\alpha}}{t_{2\left(  n+1\right)  ,2n+1}^{\alpha}}+\left(  \mu
_{n,2n}^{n}\mu_{1,1}^{1}-\lambda_{n,2n}^{n}\lambda_{1,1}^{1}\right)
\frac{t_{2n,2n}^{\alpha}t_{2,1}^{\alpha}}{t_{2\left(  n+1\right)
,2n+1}^{\alpha}},\\
\mu_{n+1,2\left(  n+1\right)  }^{n+1}=  &  \mu_{n,2n}^{n}\mu_{1,2}^{1}%
-\lambda_{n,2n}^{n}\lambda_{1,2}^{1},
\end{align*}
respectively. Summarizing all calculations above, we have proved the next theorem.

\begin{theorem}
[Trigonometric basis transformation]%
\label{thm:trigonometric_basis_transformation}The matrix form of the linear
transformation from the normalized B-basis $\overline{\mathcal{T}}_{2\left(  n+1\right)
}^{\alpha}$ to the canonical trigonometric basis $\mathcal{T}_{2\left(
n+1\right)  }^{\alpha}$ is%
\[
\left[
\begin{array}
[c]{c}%
1\\
\sin\left(  u\right)  \\
\cos\left(  u\right)  \\
\vdots\\
\sin\left(  \left(  n+1\right)  u\right)  \\
\cos\left(  \left(  n+1\right)  u\right)
\end{array}
\right]  =\left[
\begin{array}
[c]{cccccc}%
1 & 1 & 1 & \cdots & 1 & 1\\
\lambda_{1,0}^{n+1} & \lambda_{1,1}^{n+1} & \lambda_{1,2}^{n+1} & \cdots &
\lambda_{1,2n+1}^{n+1} & \lambda_{1,2\left(  n+1\right)  }^{n+1}\\
\mu_{1,0}^{n+1} & \mu_{1,1}^{n+1} & \mu_{1,2}^{n+1} & \cdots & \mu
_{1,2n+1}^{n+1} & \mu_{1,2\left(  n+1\right)  }^{n+1}\\
\vdots & \vdots & \vdots &  & \vdots & \vdots\\
\lambda_{n+1,0}^{n+1} & \lambda_{n+1,1}^{n+1} & \lambda_{n+1,2}^{n+1} & \cdots
& \lambda_{n+1,2n+1}^{n+1} & \lambda_{n+1,2\left(  n+1\right)  }^{n+1}\\
\mu_{n+1,0}^{n+1} & \mu_{n+1,1}^{n+1} & \mu_{n+1,2}^{n+1} & \cdots &
\mu_{n+1,2n+1}^{n+1} & \mu_{n+1,2\left(  n+1\right)  }^{n+1}%
\end{array}
\right]  \left[
\begin{array}
[c]{c}%
T_{2\left(  n+1\right)  ,0}^{\alpha}\left(  u\right)  \\
T_{2\left(  n+1\right)  ,1}^{\alpha}\left(  u\right)  \\
T_{2\left(  n+1\right)  ,2}^{\alpha}\left(  u\right)  \\
\vdots\\
T_{2\left(  n+1\right)  ,2n+1}^{\alpha}\left(  u\right)  \\
T_{2\left(  n+1\right)  ,2\left(  n+1\right)  }^{\alpha}\left(  u\right)
\end{array}
\right]
\]
for all parameters $u\in\left[  0,\alpha\right]  $.
\end{theorem}

\begin{remark}
The matrix of the $\left(  n+1\right)  $th order basis transformation that appears in Theorem \ref{thm:trigonometric_basis_transformation} can be efficiently calculated by parallel programming since its rows and their entries are independent of each other. Based on the entries of the first and $n$th order transformation matrices that are already calculated in previous steps, each thread block has to build up a single row of the $\left(  n+1\right)  $th order basis transformation matrix, while each thread within a block has to calculate a single entry of the corresponding row.
\end{remark}

\subsection{The hyperbolic case}

In this case we can proceed as in Subsection
\ref{sec:trigonometric_basis_transformation}. Naturally, instead of
trigonometric sine, cosine, tangent functions and identities (\ref{sin_of_sum}%
) and (\ref{cos_of_sum}) one has to apply the hyperbolic variant of these
functions and identities, respectively. The only difference consists in a sign
change in the hyperbolic counterpart of the identity (\ref{cos_of_sum}),
since
\begin{equation}
\cosh\left(  a+b\right)  =\cosh\left(  a\right)  \cosh\left(  b\right)
+\sinh\left(  a\right)  \sinh\left(  b\right)  . \label{cosh_of_sum}%
\end{equation}
Let $k\in\left\{  0,1,\ldots,n\right\}  $ be an arbitrarily fixed natural
number and denote the representations of hyperbolic functions $\sinh\left(
ku\right)  $ and $\cosh\left(  ku\right)  $ in the B-basis (\ref{Wang_basis})
by%
\begin{equation}
\sinh\left(  ku\right)  =\sum_{i=0}^{2n}\sigma_{k,i}^{n}H_{2n,i}^{\alpha
}\left(  u\right)  ,~u\in\left[  0,\alpha\right]
\end{equation}
and%
\begin{equation}
\cosh\left(  ku\right)  =\sum_{i=0}^{2n}\rho_{k,i}^{n}H_{2n,i}^{\alpha}\left(
u\right)  ,~u\in\left[  0,\alpha\right]
\end{equation}
respectively, where coefficients $\left\{  \sigma_{k,i}^{n}\right\}
_{i=0}^{2n}$ and $\left\{  \rho_{k,i}^{n}\right\}  _{i=0}^{2n}$ are unique
scalars. The basis transformation from the first order B-basis $\overline
{\mathcal{H}}_{2}^{\alpha}$ to the first order hyperbolic canonical basis
$\mathcal{H}_{2}^{\alpha}$ can be written in the matrix form%
\[
\left[
\begin{array}
[c]{c}%
1\\
\sinh\left(  u\right) \\
\cosh\left(  u\right)
\end{array}
\right]  =\left[
\begin{array}
[c]{ccc}%
\rho_{0,0}^{1} & \rho_{0,1}^{1} & \rho_{0,2}^{1}\\
\sigma_{1,0}^{1} & \sigma_{1,1}^{1} & \sigma_{1,2}^{1}\\
\rho_{1,0}^{1} & \rho_{1,1}^{1} & \rho_{1,2}^{1}%
\end{array}
\right]  \left[
\begin{array}
[c]{c}%
H_{2,0}^{\alpha}\left(  u\right) \\
H_{2,1}^{\alpha}\left(  u\right) \\
H_{2,2}^{\alpha}\left(  u\right)
\end{array}
\right]  ,~\forall u\in\left[  0,\alpha\right]  ,
\]
where%
\begin{equation}
\left\{
\begin{array}
[c]{l}%
\rho_{0,0}^{1}=\rho_{0,1}^{1}=\rho_{0,2}^{1}=1,\\
\\
\sigma_{1,0}^{1}=0,~\sigma_{1,1}^{1}=\tanh\left(  \frac{\alpha}{2}\right)
,~\sigma_{1,2}^{1}=\sinh\left(  \alpha\right)  ,\\
\\
\rho_{1,0}^{1}=\rho_{1,1}^{1}=1,~\rho_{1,2}^{1}=\cosh\left(  \alpha\right)  .
\end{array}
\right.  \label{hyperbolic_initial_conditions}%
\end{equation}
Using initial conditions (\ref{hyperbolic_initial_conditions}) and normalizing
constants $\left[  h_{2,j}^{\alpha}\right]  _{j=0}^{2},\left[  h_{2n,i}%
^{\alpha}\right]  _{i=0}^{2n}$ and $\left[  h_{2\left(  n+1\right)
,r}^{\alpha}\right]  _{r=0}^{2\left(  n+1\right)  }$ instead of $\left[
t_{2,j}^{\alpha}\right]  _{j=0}^{2},\left[  t_{2n,i}^{\alpha}\right]
_{i=0}^{2n}$ and $\left[  t_{2\left(  n+1\right)  ,r}^{\alpha}\right]
_{r=0}^{2\left(  n+1\right)  }$, respectively, one obtains recursive formulae
for the unique order elevated coefficients $\left[  \sigma_{k,r}^{n+1}\right]
_{k=0,r=0}^{n+1,2\left(  n+1\right)  }$ and $\left[  \rho_{k,r}^{n+1}\right]
_{k=0,r=0}^{n,2\left(  n+1\right)  }$ in a similar way as it was done in the
trigonometric case for constants $\left[  \lambda_{k,r}^{n+1}\right]
_{k=0,r=0}^{n+1,2\left(  n+1\right)  }$ and $\left[  \mu_{k,r}^{n+1}\right]
_{k=0,r=0}^{n,2\left(  n+1\right)  }$, respectively, while applying identity
(\ref{cosh_of_sum}) for constants $\left[  \rho_{n+1,r}^{n+1}\right]
_{r=0}^{2\left(  n+1\right)  }$ we have that%
\begin{align*}
\rho_{n+1,0}^{n+1}=  &  \rho_{n,0}^{n}\rho_{1,0}^{1}+\sigma_{n,0}^{n}%
\sigma_{1,0}^{1},\\
\rho_{n+1,1}^{n+1}=  &  \left(  \rho_{n,0}^{n}\rho_{1,1}^{1}+\sigma_{n,0}%
^{n}\sigma_{1,1}^{1}\right)  \frac{h_{2n,0}^{\alpha}h_{2,1}^{\alpha}%
}{h_{2\left(  n+1\right)  ,1}^{\alpha}}+\left(  \rho_{n,1}^{n}\rho_{1,0}%
^{1}+\sigma_{n,1}^{n}\sigma_{1,0}^{1}\right)  \frac{h_{2n,1}^{\alpha}%
h_{2,0}^{\alpha}}{h_{2\left(  n+1\right)  ,1}^{\alpha}},\\
\rho_{n+1,r}^{n+1}=  &  \left(  \rho_{n,r-2}^{n}\rho_{1,2}^{1}+\sigma
_{n,r-2}^{n}\sigma_{1,2}^{1}\right)  \frac{h_{2n,r-2}^{\alpha}h_{2,2}^{\alpha
}}{h_{2\left(  n+1\right)  ,r}^{\alpha}}+\left(  \rho_{n,r-1}^{n}\rho
_{1,1}^{1}+\sigma_{n,r-1}^{n}\sigma_{1,1}^{1}\right)  \frac{h_{2n,r-1}%
^{\alpha}h_{2,1}^{\alpha}}{h_{2\left(  n+1\right)  ,r}^{\alpha}}\\
&  +\left(  \rho_{n,r}^{n}\rho_{1,0}^{1}+\sigma_{n,r}^{n}\sigma_{1,0}%
^{1}\right)  \frac{h_{2n,r}^{\alpha}h_{2,0}^{\alpha}}{h_{2\left(  n+1\right)
,r}^{\alpha}},~r=2,3,\ldots,2n,\\
\rho_{n+1,2n+1}^{n+1}=  &  \left(  \rho_{n,2n-1}^{n}\rho_{1,2}^{1}%
+\sigma_{n,2n-1}^{n}\sigma_{1,2}^{1}\right)  \frac{h_{2n,2n-1}^{\alpha}%
h_{2,2}^{\alpha}}{h_{2\left(  n+1\right)  ,2n+1}^{\alpha}}+\left(  \rho
_{n,2n}^{n}\rho_{1,1}^{1}+\sigma_{n,2n}^{n}\sigma_{1,1}^{1}\right)
\frac{h_{2n,2n}^{\alpha}h_{2,1}^{\alpha}}{h_{2\left(  n+1\right)
,2n+1}^{\alpha}},\\
\rho_{n+1,2\left(  n+1\right)  }^{n+1}=  &  \rho_{n,2n}^{n}\rho_{1,2}%
^{1}+\sigma_{n,2n}^{n}\sigma_{1,2}^{1}.
\end{align*}
Summarizing all calculations, one can formulate the following theorem.

\begin{theorem}
[Hyperbolic basis transformation]\label{thm:hyperbolic_basis_transformation}%
The matrix form of the linear transformation from the normalized B-basis $\overline
{\mathcal{H}}_{2\left(  n+1\right)  }^{\alpha}$ to the canonical hyperbolic
basis $\mathcal{H}_{2\left(  n+1\right)  }^{\alpha}$ is%
\[
\left[
\begin{array}
[c]{c}%
1\\
\sinh\left(  u\right) \\
\cosh\left(  u\right) \\
\vdots\\
\sinh\left(  \left(  n+1\right)  u\right) \\
\cosh\left(  \left(  n+1\right)  u\right)
\end{array}
\right]  =\left[
\begin{array}
[c]{cccccc}%
1 & 1 & 1 & \cdots & 1 & 1\\
\sigma_{1,0}^{n+1} & \sigma_{1,1}^{n+1} & \sigma_{1,2}^{n+1} & \cdots &
\sigma_{1,2n+1}^{n+1} & \sigma_{1,2\left(  n+1\right)  }^{n+1}\\
\rho_{1,0}^{n+1} & \rho_{1,1}^{n+1} & \rho_{1,2}^{n+1} & \cdots &
\rho_{1,2n+1}^{n+1} & \rho_{1,2\left(  n+1\right)  }^{n+1}\\
\vdots & \vdots & \vdots &  & \vdots & \vdots\\
\sigma_{n+1,0}^{n+1} & \sigma_{n+1,1}^{n+1} & \sigma_{n+1,2}^{n+1} & \cdots &
\sigma_{n+1,2n+1}^{n+1} & \sigma_{n+1,2\left(  n+1\right)  }^{n+1}\\
\rho_{n+1,0}^{n+1} & \rho_{n+1,1}^{n+1} & \rho_{n+1,2}^{n+1} & \cdots &
\rho_{n+1,2n+1}^{n+1} & \rho_{n+1,2\left(  n+1\right)  }^{n+1}%
\end{array}
\right]  \left[
\begin{array}
[c]{c}%
H_{2\left(  n+1\right)  ,0}^{\alpha}\left(  u\right) \\
H_{2\left(  n+1\right)  ,1}^{\alpha}\left(  u\right) \\
H_{2\left(  n+1\right)  ,2}^{\alpha}\left(  u\right) \\
\vdots\\
H_{2\left(  n+1\right)  ,2n+1}^{\alpha}\left(  u\right) \\
H_{2\left(  n+1\right)  ,2\left(  n+1\right)  }^{\alpha}\left(  u\right)
\end{array}
\right]
\]
for all parameters $u\in\left[  0,\alpha\right]  $.
\end{theorem}

\section{Control point based exact description\label{sec:exact_description}}

Using curves of the type (\ref{trigonometric_curve}) or
(\ref{hyperbolic_curve}), following subsections provide control point
configurations for the exact description of higher order (mixed partial)
derivatives of any smooth parametric curve (higher dimensional multivariate
surface) specified by coordinate functions given in traditional parametric
form, i.e., in vector spaces (\ref{truncated_Fourier_vector_space}) or
(\ref{hyperbolic_vector_space}), respectively. The obtained results will also
be extended for the control point based exact description of the rational
counterpart of these curves and multivariate surfaces. Core properties of this
section are formulated by the next lemmas.

\begin{lemma}
[Exact description of trigonometric polynomials]\label{lem:tp}Consider the
trigonometric polynomial%
\begin{equation}
g\left(  u\right)  =\sum_{p\in P}c_{p}\cos\left(  pu+\psi_{p}\right)
+\sum_{q\in Q}s_{q}\sin\left(  qu+\varphi_{q}\right)  ,~u\in\left[
0,\alpha\right]  ,~\alpha\in\left(  0,\pi\right)
\label{trigonometric_polynomial}%
\end{equation}
of order at most $n$, where $P,Q\subset%
\mathbb{N}
$ and $c_{p},\psi_{p},s_{q},\varphi_{q}\in%
\mathbb{R}
$. Then, we have the equality%
\[
\frac{\text{\emph{d}}^{r}}{\text{\emph{d}}u^{r}}g\left(  u\right)  =\sum
_{i=0}^{2n}d_{i}\left(  r\right)  T_{2n,i}^{\alpha}\left(  u\right)  ,~\forall
u\in\left[  0,\alpha\right]  ,~\forall r\in%
\mathbb{N}
,
\]
where trigonometric ordinates $\left[  d_{i}\left(  r\right)  \right]
_{i=0}^{2n}$ are of the form%
\begin{align}
d_{i}\left(  r\right)  =  &  \sum_{p\in P}c_{p}p^{r}\left(  \mu_{p,i}^{n}%
\cos\left(  \psi_{p}+\frac{r\pi}{2}\right)  -\lambda_{p,i}^{n}\sin\left(
\psi_{p}+\frac{r\pi}{2}\right)  \right) \label{trigonometric_ordinates}\\
&  +\sum_{q\in Q}s_{q}q^{r}\left(  \lambda_{q,i}^{n}\cos\left(  \varphi
_{q}+\frac{r\pi}{2}\right)  +\mu_{q,i}^{n}\sin\left(  \varphi_{q}+\frac{r\pi
}{2}\right)  \right)  .\nonumber
\end{align}

\end{lemma}

\begin{proof}
The $r$th order derivative of the trigonometric polynomial \textbf{(}%
\ref{trigonometric_polynomial}\textbf{)} can be written in the form%
\begin{align*}
\frac{\text{d}^{r}}{\text{d}u^{r}}g\left(  u\right)  =  &  \sum_{p\in P}%
c_{p}p^{r}\cos\left(  pu+\psi_{p}+\frac{r\pi}{2}\right)  +\sum_{q\in Q}%
s_{q}q^{r}\sin\left(  qu+\varphi_{q}+\frac{r\pi}{2}\right) \\
& \\
=  &  \sum_{p\in P}c_{p}p^{r}\left(  \cos\left(  pu\right)  \cos\left(
\psi_{p}+\frac{r\pi}{2}\right)  -\sin\left(  pu\right)  \sin\left(  \psi
_{p}+\frac{r\pi}{2}\right)  \right) \\
&  +\sum_{q\in Q}s_{q}q^{r}\left(  \sin\left(  qu\right)  \cos\left(
\varphi_{q}+\frac{r\pi}{2}\right)  +\cos\left(  qu\right)  \sin\left(
\varphi_{q}+\frac{r\pi}{2}\right)  \right) \\
& \\
=  &  \sum_{p\in P}c_{p}p^{r}\cos\left(  \psi_{p}+\frac{r\pi}{2}\right)
\cos\left(  pu\right)  -\sum_{p\in P}c_{p}p^{r}\sin\left(  \psi_{p}+\frac
{r\pi}{2}\right)  \sin\left(  pu\right) \\
&  +\sum_{q\in Q}s_{q}q^{r}\cos\left(  \varphi_{q}+\frac{r\pi}%
{2}\right)  \sin\left(  qu\right)  +\sum_{q\in Q}s_{q}q^{r}\sin\left(
\varphi_{q}+\frac{r\pi}{2}\right)  \cos\left(  qu\right) \\
& \\
=  &  \sum_{p\in P}c_{p}p^{r}\cos\left(  \psi_{p}+\frac{r\pi}{2}\right)
\left(  \sum_{i=0}^{2n}\mu_{p,i}^{n}T_{2n,i}^{\alpha}\left(  u\right)
\right)  -\sum_{p\in P}c_{p}p^{r}\sin\left(  \psi_{p}+\frac{r\pi}{2}\right)
\left(  \sum_{i=0}^{2n}\lambda_{p,i}^{n}T_{2n,i}^{\alpha}\left(  u\right)
\right) \\
&  +\sum_{q\in Q}s_{q}q^{r}\cos\left(  \varphi_{q}+\frac{r\pi}{2}\right)
\left(  \sum_{i=0}^{2n}\lambda_{q,i}^{n}T_{2n,i}^{\alpha}\left(  u\right)
\right)  +\sum_{q\in Q}s_{q}q^{r}\sin\left(  \varphi_{q}+\frac{r\pi}%
{2}\right)  \left(  \sum_{i=0}^{2n}\mu_{q,i}^{n}T_{2n,i}^{\alpha}\left(
u\right)  \right)
\end{align*}
for all parameters $u\in\left[  0,\alpha\right]  $, where we have applied
Theorem \ref{thm:trigonometric_basis_transformation} for order $n$. Collecting
the coefficients of basis functions $\left\{  T_{2n,i}^{\alpha}\right\}
_{i=0}^{2n}$, one obtains the ordinates specified by
(\ref{trigonometric_ordinates}).
\end{proof}

\begin{lemma}
[Exact description of hyperbolic polynomials]\label{lem:hp}Consider the
hyperbolic function%
\[
g\left(  u\right)  =\sum_{p\in P}c_{p}\cosh\left(  pu+\psi_{p}\right)
+\sum_{q\in Q}s_{q}\sinh\left(  qu+\varphi_{q}\right)  ,~u\in\left[
0,\alpha\right]  ,~\alpha>0
\]
of order at most $n$, where $P,Q\subset%
\mathbb{N}
$ and $c_{p},\psi_{p},s_{q},\varphi_{q}\in%
\mathbb{R}
$. Then, one has that%
\[
\frac{\text{\emph{d}}^{r}}{\text{\emph{d}}u^{r}}g\left(  u\right)  =\sum
_{i=0}^{2n}d_{i}\left(  r\right)  H_{2n,i}^{\alpha}\left(  u\right)  ,~\forall
u\in\left[  0,\alpha\right]  ,~\forall r\in%
\mathbb{N}
,
\]
where%
\begin{equation}
d_{i}\left(  r\right)  =\left\{
\begin{array}
[c]{ll}%
{\displaystyle\sum\limits_{p\in P}}
c_{p}p^{r}\left(  \rho_{p,i}^{n}\cosh\left(  \psi_{p}\right)  +\sigma
_{p,i}^{n}\sinh\left(  \psi_{p}\right)  \right)  & \\
+%
{\displaystyle\sum\limits_{q\in Q}}
s_{q}q^{r}\left(  \sigma_{q,i}^{n}\cosh\left(  \varphi_{q}\right)  +\rho
_{q,i}^{n}\sinh\left(  \varphi_{q}\right)  \right)  , & r=2z,\\
& \\%
{\displaystyle\sum\limits_{p\in P}}
c_{p}p^{r}\left(  \sigma_{p,i}^{n}\cosh\left(  \psi_{p}\right)  +\rho
_{p,i}^{n}\sinh\left(  \psi_{p}\right)  \right)  & \\
+%
{\displaystyle\sum\limits_{q\in Q}}
s_{q}q^{r}\left(  \rho_{q,i}^{n}\cosh\left(  \varphi_{q}\right)  +\sigma
_{q,i}^{n}\sinh\left(  \varphi_{q}\right)  \right)  , & r=2z+1.
\end{array}
\right.  \label{hyperbolic_ordinates}%
\end{equation}

\end{lemma}

\begin{proof}
Using the basis transformation formulated in Theorem
\ref{thm:hyperbolic_basis_transformation} for order $n$ and applying
derivative formulae%
\begin{align*}
\frac{\text{d}^{r}}{\text{d}u^{r}}\cosh\left(  au+b\right)   &  =\left\{
\begin{array}
[c]{ll}%
a^{r}\cosh\left(  au+b\right)  , & r=2z,\\
& \\
a^{r}\sinh\left(  au+b\right)  , & r=2z+1,
\end{array}
\right. \\
& \\
\frac{\text{d}^{r}}{\text{d}u^{r}}\sinh\left(  au+b\right)   &  =\left\{
\begin{array}
[c]{ll}%
a^{r}\sinh\left(  au+b\right)  , & r=2z,\\
& \\
a^{r}\cosh\left(  au+b\right)  , & r=2z+1
\end{array}
\right.
\end{align*}
with hyperbolic identities (\ref{cosh_of_sum}) and%
\[
\sinh\left(  a+b\right)  =\sinh\left(  a\right)  \cosh\left(  b\right)
+\cosh\left(  a\right)  \sinh\left(  b\right)  ,
\]
one can follow the steps of the proof of the previous Lemma \ref{lem:tp}.
\end{proof}

\subsection{Description of (rational) trigonometric curves and
surfaces\label{sec:exact_description_trigonometric}}

Consider the smooth parametric curve%
\begin{equation}
\mathbf{g}\left(  u\right)  =\left[  g^{\ell}\left(  u\right)  \right]
_{\ell=1}^{\delta},~u\in\left[  0,\alpha\right]  ,~\alpha\in\left(
0,\pi\right)  \label{traditional_trigonometric_curve}%
\end{equation}
with coordinate functions of the form%
\[
g^{\ell}\left(  u\right)  =\sum_{p\in P_{\ell}}c_{p}^{\ell}\cos\left(
pu+\psi_{p}^{\ell}\right)  +\sum_{q\in Q_{\ell}}s_{q}^{\ell}\sin\left(
qu+\varphi_{q}^{\ell}\right)
\]
and the vector space (\ref{truncated_Fourier_vector_space}) of order
\[
n\geq n_{\min}=\max\left\{  z:z\in\cup_{\ell=1}^{\delta}\left(  P_{\ell}\cup
Q_{\ell}\right)  \right\}  ,
\]
where $P_{\ell},Q_{\ell}\subset%
\mathbb{N}
$ and $c_{p}^{\ell},\psi_{p}^{\ell},s_{q}^{\ell},\varphi_{q}^{\ell}\in%
\mathbb{R}
$.

\begin{theorem}
[Control point based exact description of trigonometric curves]%
\label{thm:cpbed_trigonometric_curves}The $r$th ($r\in%
\mathbb{N}
$) order\ derivative of the curve (\ref{traditional_trigonometric_curve}) can
be written in the form%
\[
\frac{\text{\emph{d}}^{r}}{\text{\emph{d}}u^{r}}\mathbf{g}\left(  u\right)
=\sum_{i=0}^{2n}\mathbf{d}_{i}\left(  r\right)  T_{2n,i}^{\alpha}\left(
u\right)  ,~\forall u\in\left[  0,\alpha\right]  ,
\]
where control points $\mathbf{d}_{i}\left(  r\right)  =\left[  d_{i}^{\ell
}\left(  r\right)  \right]  _{\ell=1}^{\delta}$ are determined by coordinates%
\begin{align}
d_{i}^{\ell}\left(  r\right)  =  &  \sum_{p\in P_{\ell}}c_{p}^{\ell}%
p^{r}\left(  \mu_{p,i}^{n}\cos\left(  \psi_{p}^{\ell}+\frac{r\pi}{2}\right)
-\lambda_{p,i}^{n}\sin\left(  \psi_{p}^{\ell}+\frac{r\pi}{2}\right)  \right)
\label{trigonometric_control_points}\\
&  +\sum_{q\in Q_{\ell}}s_{q}^{\ell}q^{r}\left(  \lambda_{q,i}^{n}\cos\left(
\varphi_{q}^{\ell}+\frac{r\pi}{2}\right)  +\mu_{q,i}^{n}\sin\left(
\varphi_{q}^{\ell}+\frac{r\pi}{2}\right)  \right)  .\nonumber
\end{align}
(In case of zeroth order derivatives we will use the simpler notation
$\mathbf{d}_{i}=\left[  d_{i}^{\ell}\right]  _{\ell=1}^{\delta}=\left[
d_{i}^{\ell}\left(  0\right)  \right]  _{\ell=1}^{\delta}=\mathbf{d}%
_{i}\left(  0\right)  $ for all $i=0,1,\ldots,2n$.)
\end{theorem}

\begin{proof}
Using Lemma \ref{lem:tp}, the $r$th order derivative of the $\ell$th
coordinate function of the curve (\ref{traditional_trigonometric_curve}) can
be written in the form%
\[
\frac{\text{d}^{r}}{\text{d}u^{r}}g^{\ell}\left(  u\right)  =\sum_{i=0}%
^{2n}d_{i}^{\ell}\left(  r\right)  T_{2n,i}^{\alpha}\left(  u\right)
,~\forall u\in\left[  0,\alpha\right]  ,
\]
where the $i$th ordinate has exactly the form of
(\ref{trigonometric_control_points}). Repeating this reformulation for all
$\ell=1,2,\ldots,\delta$ and collecting the coefficients of basis functions
$\left\{  T_{2n,i}^{\alpha}\right\}  _{i=0}^{2n}$ one obtains all coordinates
of control points $\mathbf{d}_{i}\left(  r\right)  =\left[  d_{i}^{\ell
}\left(  r\right)  \right]  _{\ell=1}^{\delta}$ that can be substituted into
the description of trigonometric curves of the type (\ref{trigonometric_curve}).
\end{proof}

\begin{example}
[Application of Theorem \ref{thm:cpbed_trigonometric_curves} -- plane
curves]Cases (a) and (b) of Fig. \ref{fig:ed_non_rational_tpc} show the
control point based exact descriptions of the hypocycloidal arc%
\begin{equation}
\mathbf{g}\left(  u\right)  =\left[
\begin{array}
[c]{c}%
g^{1}\left(  u\right) \\
\\
g^{2}\left(  u\right)
\end{array}
\right]  =\left[
\begin{array}
[c]{l}%
4\cos\left(  u-\frac{\pi}{3}\right)  +\cos\left(  4u-\frac{\pi}{3}\right) \\
\\
4\sin\left(  u-\frac{\pi}{3}\right)  -\sin\left(  4u-\frac{\pi}{3}\right)
\end{array}
\right]  ,~u\in\left(  0,\frac{3\pi}{4}\right)  \label{hypocycloid}%
\end{equation}
and of the arc
\begin{equation}
\mathbf{g}\left(  u\right)  =\left[
\begin{array}
[c]{c}%
g^{1}\left(  u\right) \\
\\
g^{2}\left(  u\right)
\end{array}
\right]  =\frac{1}{2}\left[
\begin{array}
[c]{l}%
\sin\left(  u-\frac{\pi}{12}\right)  +\sin\left(  3u-\frac{\pi}{4}\right) \\
\\
\cos\left(  u-\frac{\pi}{12}\right)  -\cos\left(  3u-\frac{\pi}{4}\right)
\end{array}
\right]  ,~u\in\left(  0,\frac{2\pi}{3}\right)  \label{quadrifolium}%
\end{equation}
of a quadrifolium, respectively.
\end{example}

%

\begin{figure}
[!h]
\begin{center}
\includegraphics[
height=3.4636in,
width=6.058in
]%
{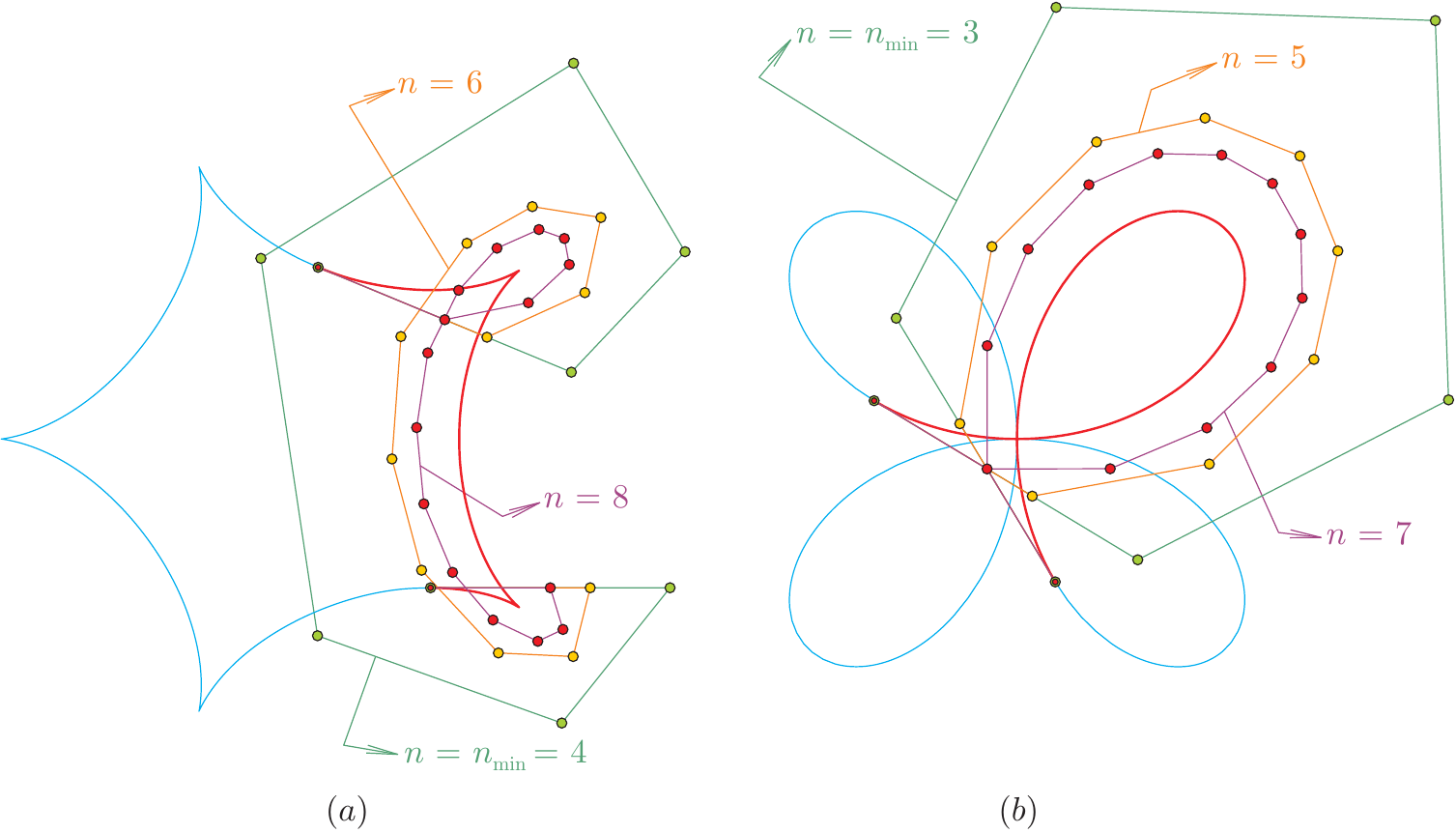}%
\caption{Order elevated control point based exact description of trigonometric
curves (\ref{hypocycloid}) and (\ref{quadrifolium}) by means of Theorem
\ref{thm:cpbed_trigonometric_curves}. (\emph{a}) A hypocycloidal arc
($\alpha=\frac{3\pi}{4}$). (\emph{b}) An arc of a quadrifolium ($\alpha
=\frac{2\pi}{3}$).}%
\label{fig:ed_non_rational_tpc}%
\end{center}
\end{figure}

\begin{example}
[Application of Theorem \ref{thm:cpbed_trigonometric_curves} -- space
curve]Fig. \ref{fig:torus_knot} illustrates the control point based exact
descriptions of the arc%
\begin{equation}
\mathbf{g}\left(  u\right)  =\left[
\begin{array}
[c]{c}%
g^{1}\left(  u\right) \\
\\
g^{2}\left(  u\right) \\
\\
g^{3}\left(  u\right)
\end{array}
\right]  =\left[
\begin{array}
[c]{l}%
\frac{1}{2}\cos\left(  u\right)  +2\cos\left(  3u\right)  +\frac{1}{2}%
\cos\left(  5u\right) \\
\\
\frac{1}{2}\sin\left(  u\right)  +2\sin\left(  3u\right)  +\frac{1}{2}%
\sin\left(  5u\right) \\
\\
\sin\left(  2u\right)
\end{array}
\right]  ,~u\in\left(  0,\frac{\pi}{2}\right)  , \label{torus_knot}%
\end{equation}
of a torus knot.
\end{example}

%

\begin{figure}
[!h]
\begin{center}
\includegraphics[
height=3.8553in,
width=4.9882in
]%
{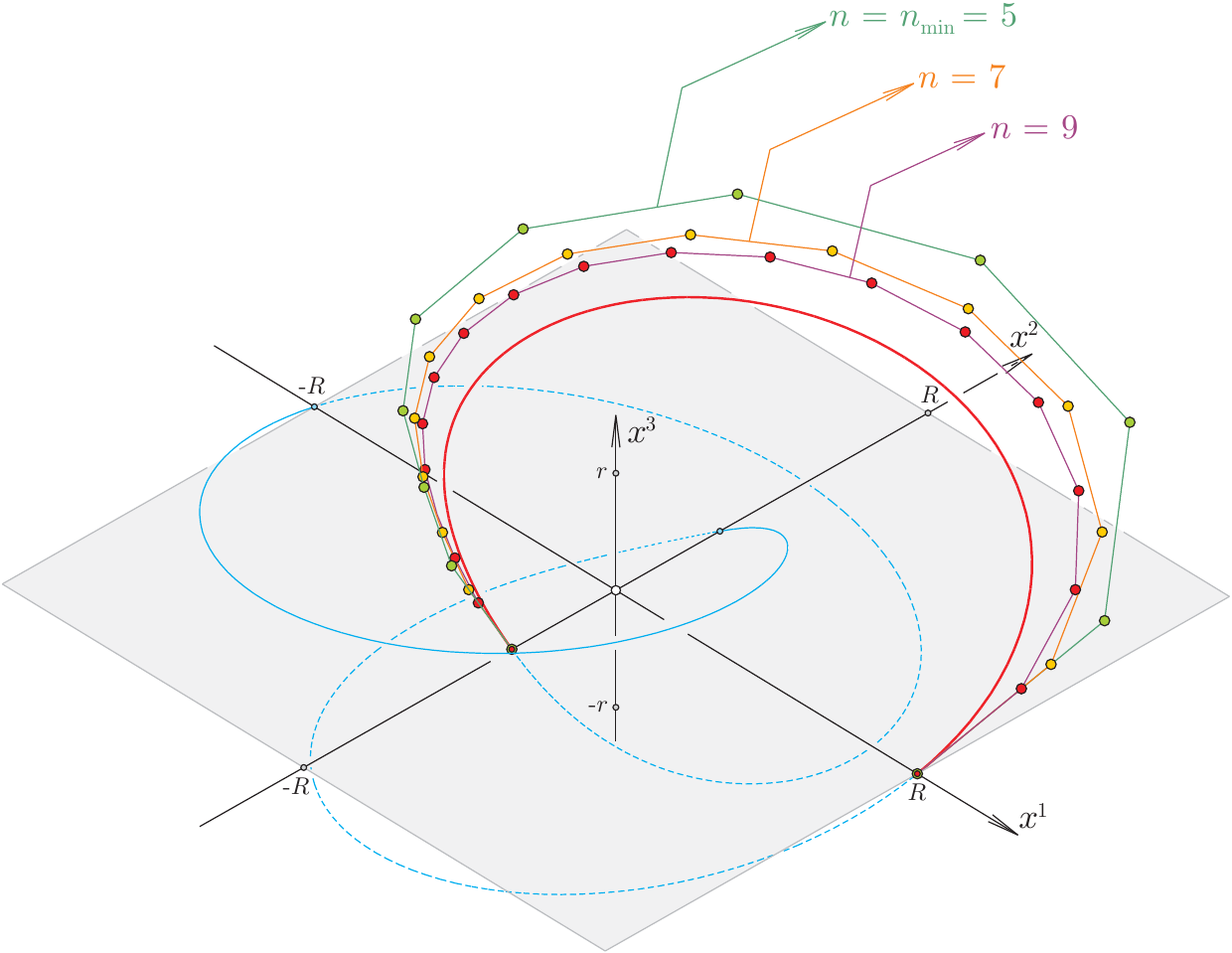}%
\caption{Order elevated control point based exact description of an arc
($\alpha=\frac{\pi}{2}$) of a knot that lies on the surface of the torus
$\left[  \left(  R+r\sin\left(  u_{1}\right)  \right)  \cos\left(
u_{2}\right)  ,\left(  R+r\sin\left(  u_{1}\right)  \right)  \sin\left(
u_{2}\right)  ,r\cos\left(  u_{1}\right)  \right]  $, where $\left(  u_{1}%
,u_{2}\right)  \in\left[  0,2\pi\right]  \times\left[  0,2\pi\right]  $, $R=2$
and $r=1$. Control points were obtained by using Theorem
\ref{thm:cpbed_trigonometric_curves}.}%
\label{fig:torus_knot}%
\end{center}
\end{figure}

Consider now the rational trigonometric curve%
\begin{equation}
\mathbf{g}\left(  u\right)  =\frac{1}{g^{\delta+1}\left(  u\right)  }\left[
g^{\ell}\left(  u\right)  \right]  _{\ell=1}^{\delta},~u\in\left[
0,\alpha\right]  , \label{rational_trigonometric_curve_to_be_described}%
\end{equation}
given in traditional parametric form
\[
g^{\ell}\left(  u\right)  =\sum_{p\in P_{\ell}}c_{p}^{\ell}\cos\left(
pu+\psi_{p}^{\ell}\right)  +\sum_{q\in Q_{\ell}}s_{q}^{\ell}\sin\left(
qu+\varphi_{q}^{\ell}\right)  ,~P_{\ell},Q_{\ell}\subset%
\mathbb{N}
,~c_{p}^{\ell},\psi_{p}^{\ell},s_{q}^{\ell},\varphi_{q}^{\ell}\in%
\mathbb{R}
,~\ell=1,2,\ldots,\delta+1,
\]
where%
\[
g^{\delta+1}\left(  u\right)  >0,~\forall u\in\left[  0,\alpha\right]  .
\]

\begin{algorithm}
[Control point based exact description of rational trigonometric
curves]\label{alg:cpbed_rational_trigonometric_curves}The process that
provides the control point based exact description of the rational curve
(\ref{rational_trigonometric_curve_to_be_described}) consists of the following operations:

\begin{itemize}
\item let%
\[
n\geq n_{\min}=\max\left\{  z:z\in\cup_{\ell=1}^{\delta+1}\left(  P_{\ell}\cup
Q_{\ell}\right)  \right\}
\]
be an arbitrarily fixed order;

\item apply Theorem \ref{thm:cpbed_trigonometric_curves} to the pre-image
\begin{equation}
\mathbf{g}_{\wp}\left(  u\right)  =\left[  g^{\ell}\left(  u\right)  \right]
_{\ell=1}^{\delta+1},~u\in\left[  0,\alpha\right]
\label{pre_image_of_rational_trigonometric_curve_to_be_described}%
\end{equation}
of the curve (\ref{rational_trigonometric_curve_to_be_described}), i.e.,
compute control points
\[
\mathbf{d}_{i}^{\wp}=\left[  d_{i}^{\ell}\right]  _{\ell=1}^{\delta+1}\in%
\mathbb{R}
^{\delta+1},~i=0,1,\ldots,2n
\]
for the exact trigonometric representation of
(\ref{pre_image_of_rational_trigonometric_curve_to_be_described}) in the
pre-image space $%
\mathbb{R}
^{\delta+1}$;

\item project the obtained control points onto the hyperplane $x^{\delta+1}=1$
that results the control points
\[
\mathbf{d}_{i}=\frac{1}{d_{i}^{\delta+1}}\left[  d_{i}^{\ell}\right]  _{\ell
=1}^{\delta}\in%
\mathbb{R}
^{\delta},~i=0,1,\ldots,2n
\]
and weights%
\[
\omega_{i}=d_{i}^{\delta+1},~i=0,1,\ldots,2n
\]
needed for the rational trigonometric representation
(\ref{rational_trigonometric_curve}) of
(\ref{rational_trigonometric_curve_to_be_described});

\item the above generation process does not necessarily ensure the positivity
of all weights, since the last coordinate of some control points in the
pre-image space $%
\mathbb{R}
^{\delta+1}$ can be negative; if this is the case, one can increase in $%
\mathbb{R}
^{\delta+1}$ the order of the trigonometric curve used for the control point
based exact description of the pre-image $\mathbf{g}_{\wp}$, since -- as
stated in Remark \ref{rem:trigonometric_order_elevation} -- order elevation
generates a sequence of control polygons that converges to $\mathbf{g}_{\wp}$
which is a geometric object of one branch that does not intersect the
vanishing plane $x^{\delta+1}=0\,$(i.e., the $\left(  \delta+1\right)  $th
coordinate of all its points are of the same sign); therefore, it is
guaranteed that exists a finite and minimal order $n+z$ ($z\geq1$) for which
all weights are positive.
\end{itemize}
\end{algorithm}

\begin{example}
[Application of Algorithm \ref{alg:cpbed_rational_trigonometric_curves} --
rational curves]Fig. \ref{fig:bernoulli_s_lemniscate} shows the control point
based description of an arc of the rational trigonometric curve%
\begin{equation}
\mathbf{g}\left(  u\right)  =\frac{1}{g^{3}\left(  u\right)  }\left[
\begin{array}
[c]{c}%
g^{1}\left(  u\right) \\
\\
g^{2}\left(  u\right)
\end{array}
\right]  =\frac{1}{\frac{3}{2}-\frac{1}{2}\cos\left(  2u\right)  }\left[
\begin{array}
[c]{r}%
\cos\left(  u\right) \\
\\
\frac{1}{2}\sin\left(  2u\right)
\end{array}
\right]  ,~u\in\left[  0,\frac{2\pi}{3}\right]  ,
\label{Bernoulli_s_lemniscate}%
\end{equation}
also known as Bernoulli's lemniscate.
\end{example}

\begin{figure}
[!h]
\begin{center}
\includegraphics[
height=3.7239in,
width=6.346in
]%
{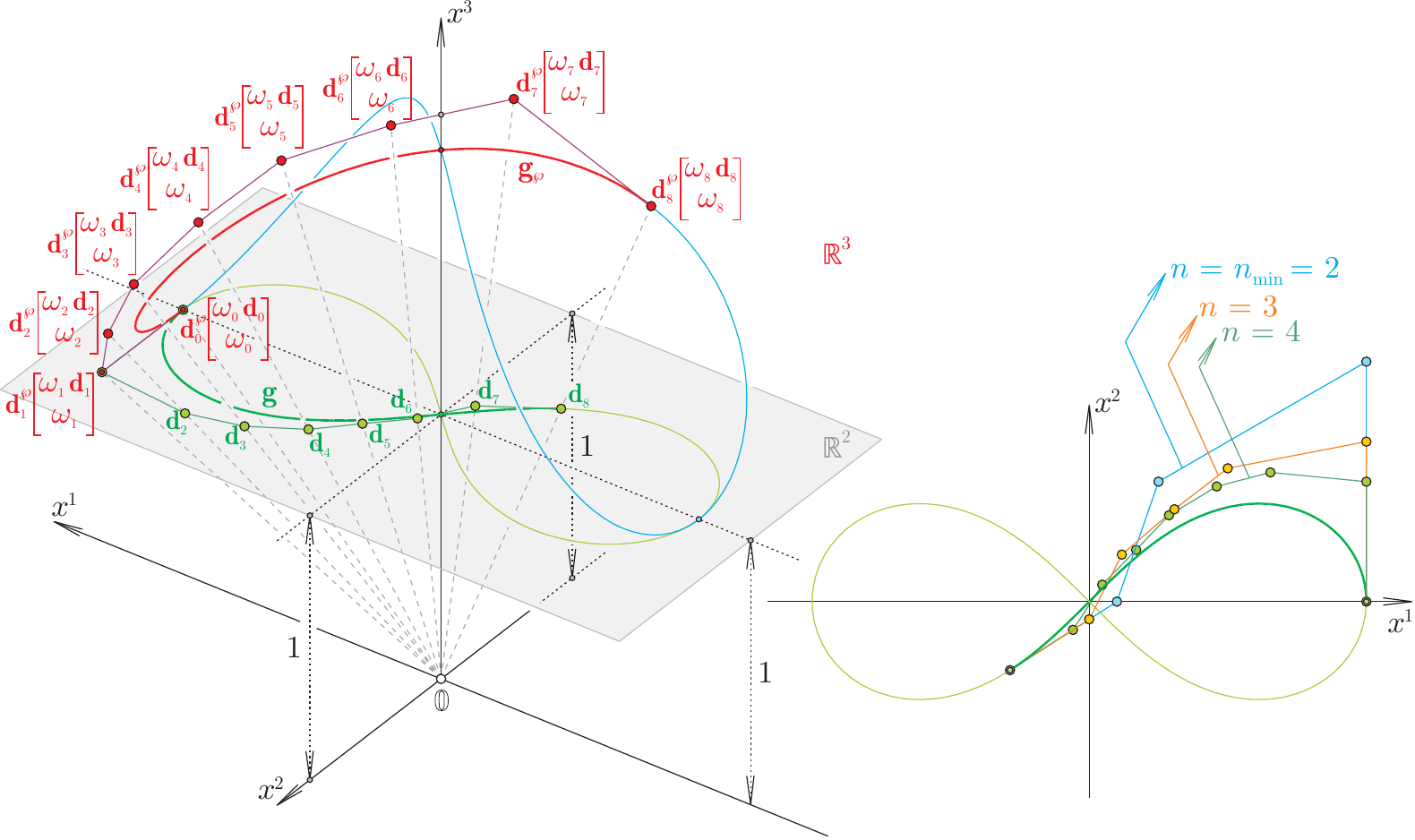}%
\caption{Using Algorithm \ref{alg:cpbed_rational_trigonometric_curves} with
shape parameter $\alpha=\frac{2\pi}{3}$, the left image shows the control
point based exact description of an arc of Bernoulli's lemniscate
(\ref{Bernoulli_s_lemniscate}) by means of a fourth order rational
trigonometric curve $\mathbf{g}$ (green) and its pre-image $\mathbf{g}_{\wp}$
(red). Control points $\mathbf{d}_{i}^{\wp}=\left[
\protect\begin{array}
[c]{cc}%
\omega_{i}\mathbf{d}_{i} & \omega_{i}%
\protect\end{array}
\right]  \in\mathbb{R} ^{3}$ ($i=0,1,\ldots,8$) of the pre-image
$\mathbf{g}_{\wp}=\left[  g^{\ell}\left(  u\right)  \right]  _{\ell=1}^{3}$
were obtained by the application of Theorem
\ref{thm:cpbed_trigonometric_curves}, while control points $\left[
\mathbf{d}_{i}\right]  _{i=0}^{8}$ and weights $\left[  \omega_{i}\right]
_{i=0}^{8}$ needed for the rational representation
(\ref{rational_trigonometric_curve}) were obtained by the central projection
of points $\left[  \mathbf{d}_{i}^{\wp}\right]  _{i=0}^{8}$ onto the
hyperplane $x^{3}=1$. The right image also illustrates the control polygons of
the second and third order exact rational trigonometric representations of the
same arc. (For interpretation of the references to color in this figure
legend, the reader is referred to the web version of this paper.)}%
\label{fig:bernoulli_s_lemniscate}%
\end{center}
\end{figure}

Theorem \ref{thm:cpbed_trigonometric_curves} can also be used to provide
control nets (grids) for the control point based exact description of the
higher order mixed partial derivatives of a general class of multivariate surfaces the
elements of which can be expressed in the form%
\begin{equation}
\mathbf{s}\left(  \mathbf{u}\right)  =\left[
\begin{array}
[c]{cccc}%
s^{1}\left(  \mathbf{u}\right)  & s^{2}\left(  \mathbf{u}\right)  & \cdots &
s^{\delta+\kappa}\left(  \mathbf{u}\right)
\end{array}
\right]  \in%
\mathbb{R}
^{\delta+\kappa},~\mathbf{u}=\left[  u_{j}\right]  _{j=1}^{\delta}\in
\times_{j=1}^{\delta}\left[  0,\alpha_{j}\right]  ,~\alpha_{j}\in\left(
0,\pi\right)  ,~\kappa\geq0 \label{trigonometric_surface_to_be_described}%
\end{equation}
where%
\[
s^{\ell}\left(  \mathbf{u}\right)  =\sum_{\zeta=1}^{m_{\ell}}\prod
_{j=1}^{\delta}\left(  \sum_{p\in P_{\ell,\zeta,j}}c_{p}^{\ell,\zeta,j}%
\cos\left(  pu_{j}+\psi_{p}^{\ell,\zeta,j}\right)  +\sum_{q\in Q_{\ell
,\zeta,j}}s_{q}^{\ell,\zeta,j}\sin\left(  qu_{j}+\varphi_{q}^{\ell,\zeta
,j}\right)  \right)  ,~\ell=1,2,\ldots,\delta+\kappa
\]
and%
\[
P_{\ell,\zeta,j},Q_{\ell,\zeta,j}\subset%
\mathbb{N}
,~m_{\ell}\in%
\mathbb{N}
\setminus\left\{  0\right\}  ,~c_{p}^{\ell,\zeta,j},\psi_{p}^{\ell,\zeta
,j},s_{q}^{\ell,\zeta,j},\varphi_{q}^{\ell,\zeta,j}\in%
\mathbb{R}
.
\]
Indeed, we have the next theorem.

\begin{theorem}
[Control point based exact description of trigonometric surfaces]%
\label{thm:cpbed_trigonometric_surfaces}The control point based exact
description of the $\left(  r_{1}+r_{2}+\ldots+r_{\delta}\right)  $th order
mixed partial derivative of the surface
(\ref{trigonometric_surface_to_be_described}) fulfills the equality%
\begin{equation}
\dfrac{\partial^{r_{1}+r_{2}+\ldots+r_{\delta}}}{\partial u_{1}^{r_{1}%
}\partial u_{2}^{r_{2}}\cdots\partial u_{\delta}^{r_{\delta}}}\mathbf{s}\left(
\mathbf{u}\right)  =%
{\displaystyle\sum\limits_{i_{1}=0}^{2n_{1}}}
{\displaystyle\sum\limits_{i_{2}=0}^{2n_{2}}}
\cdots%
{\displaystyle\sum\limits_{i_{\delta}=0}^{2n_{\delta}}}
\mathbf{d}_{i_{1},i_{2},\ldots,i_{\delta}}\left(  r_{1},r_{2},\ldots
,r_{\delta}\right)
{\displaystyle\prod\limits_{j=1}^{\delta}}
T_{2n_{j},i_{j}}^{\alpha_{j}}\left(  u_{j}\right)
\label{exact_trigonometric_surface_represantation}%
\end{equation}
for all parameter vectors $\mathbf{u\in}\times_{j=1}^{\delta}\left[
0,\alpha_{j}\right]  $, where%
\begin{align*}
n_{j}  &  \geq n_{\min}^{j}=\max\left\{  z_{j}:z_{j}\in\cup_{\ell=1}%
^{\delta+k}\cup_{\zeta=1}^{m_{\ell}}\left(  P_{\ell,\zeta,j}\cup Q_{\ell
,\zeta,j}\right)  \right\}  ,~j=1,2,\ldots,\delta,\\
\mathbf{d}_{i_{1},i_{2},\ldots,i_{\delta}}\left(  r_{1},r_{2},\ldots
,r_{\delta}\right)   &  =\left[  d_{i_{1},i_{2},\ldots,i_{\delta}}^{\ell
}\left(  r_{1},r_{2},\ldots,r_{\delta}\right)  \right]  _{\ell=1}%
^{\delta+\kappa}\in%
\mathbb{R}
^{\delta+\kappa},\\
d_{i_{1},i_{2},\ldots,i_{\delta}}^{\ell}\left(  r_{1},r_{2},\ldots,r_{\delta
}\right)   &  =\sum_{\zeta=1}^{m_{\ell}}\prod_{j=1}^{\delta}d_{i_{j}}%
^{\ell,\zeta}\left(  r_{j}\right)  ,~\ell=1,2,\ldots,\delta+\kappa
\end{align*}
and%
\begin{align*}
d_{i_{j}}^{\ell,\zeta}\left(  r_{j}\right)  =  &  \sum_{p\in P_{\ell,\zeta,j}%
}c_{p}^{\ell,\zeta,j}p^{r_{j}}\left(  \mu_{p,i_{j}}^{n_{j}}\cos\left(
\psi_{p}^{\ell,\zeta,j}+\frac{r_{j}\pi}{2}\right)  -\lambda_{p,i_{j}}^{n_{j}%
}\sin\left(  \psi_{p}^{\ell,\zeta,j}+\frac{r_{j}\pi}{2}\right)  \right) \\
&  +\sum_{q\in Q_{\ell,\zeta,j}}s_{q}^{\ell,\zeta,j}q^{r_{j}}\left(
\lambda_{q,i_{j}}^{n_{j}}\cos\left(  \varphi_{q}^{\ell,\zeta,j}+\frac{r_{j}%
\pi}{2}\right)  +\mu_{q,i_{j}}^{n_{j}}\sin\left(  \varphi_{q}^{\ell,\zeta
,j}+\frac{r_{j}\pi}{2}\right)  \right)  ,\\
\ell &  =1,2,\ldots,\delta+\kappa,~\zeta=1,2,\ldots,m_{\ell},~j=1,2,\ldots
,\delta.
\end{align*}
(In case of zeroth order partial derivatives we will use the simpler notation
\[
\mathbf{d}_{i_{1},i_{2},\ldots,i_{\delta}}=\left[  d_{i_{1},i_{2}%
,\ldots,i_{\delta}}^{\ell}\right]  _{\ell=1}^{\delta+\kappa}=\left[
d_{i_{1},i_{2},\ldots,i_{\delta}}^{\ell}\left(  0,0,\ldots,0\right)  \right]
_{\ell=1}^{\delta+\kappa}=\mathbf{d}_{i_{1},i_{2},\ldots,i_{\delta}}\left(
0,0,\ldots,0\right)
\]
for all $i_{j}=0,1,\ldots,2n_{j}$ and $j=1,2,\ldots,\delta$.)
\end{theorem}

\begin{proof}
Observe, by using Lemma \ref{lem:tp}, that equality%
\begin{align*}
&  \frac{\partial^{r_{1}+r_{2}+\ldots+r_{\delta}}}{\partial u_{1}^{r_{1}%
}\partial u_{2}^{r_{2}}\cdots\partial u_{\delta}^{r_{\delta}}}s^{\ell}\left(
\mathbf{u}\right)= \\
=  &  \frac{\partial^{r_{1}+r_{2}+\ldots+r_{\delta}}}{\partial u_{1}^{r_{1}%
}\partial u_{2}^{r_{2}}\cdots\partial u_{\delta}^{r_{\delta}}}\sum_{\zeta=1}^{m_{\ell}%
}\prod_{j=1}^{\delta}\left(  \sum_{p\in P_{\ell,\zeta,j}}c_{p}^{\ell,\zeta
,j}\cos\left(  pu_{j}+\psi_{p}^{\ell,\zeta,j}\right)  +\sum_{q\in
Q_{\ell,\zeta,j}}s_{q}^{\ell,\zeta,j}\sin\left(  qu_{j}+\varphi_{q}%
^{\ell,\zeta,j}\right)  \right) \\
=  &  \sum_{\zeta=1}^{m_{\ell}}\prod_{j=1}^{\delta}\frac{\text{d}^{r_{j}}%
}{\text{d}u_{j}^{r_{j}}}\left(  \sum_{p\in P_{\ell,\zeta,j}}c_{p}^{\ell
,\zeta,j}\cos\left(  pu_{j}+\psi_{p}^{\ell,\zeta,j}\right)  +\sum_{q\in
Q_{\ell,\zeta,j}}s_{q}^{\ell,\zeta,j}\sin\left(  qu_{j}+\varphi_{q}%
^{\ell,\zeta,j}\right)  \right) \\
=  &  \sum_{\zeta=1}^{m_{\ell}}\prod_{j=1}^{\delta}\left(  \sum_{i_{j}%
=0}^{2n_{j}}d_{i_{j}}^{\ell,\zeta}\left(  r_{j}\right)  T_{2n_{j},i_{j}%
}^{\alpha_{j}}\left(  u_{j}\right)  \right) \\
=  &  \sum_{i_{1}=0}^{2n_{1}}\sum_{i_{2}=0}^{2n_{2}}\cdots\sum_{i_{\delta}%
=0}^{2n_{\delta}}\left(  \sum_{\zeta=1}^{m_{\ell}}\prod_{j=1}^{\delta}%
d_{i_{j}}^{\ell,\zeta}\left(  r_{j}\right)  \right)  T_{2n_{1},i_{1}}%
^{\alpha_{1}}\left(  u_{1}\right)  T_{2n_{2},i_{2}}^{\alpha_{2}}\left(
u_{2}\right)  \cdot\ldots\cdot T_{2n_{\delta},i_{\delta}}^{\alpha_{\delta}%
}\left(  u_{\delta}\right)
\end{align*}
holds for all parameter vectors $\mathbf{u}=\left[  u_{j}\right]
_{j=1}^{\delta}\in\times_{j=1}^{\delta}\left[  0,\alpha_{j}\right]  $ and
coordinate functions $\ell=1,2,\ldots,\delta+\kappa$, i.e., $d_{i_{j}}%
^{\ell,\zeta}\left(  r_{j}\right)  $ can be calculated by means of formula
(\ref{trigonometric_control_points}) for all indices $\ell=1,2,\ldots
,\delta+\kappa$, $\zeta=1,2,\ldots,m_{\ell}$ and $j=1,2,\ldots,\delta$.
\end{proof}

\begin{example}
[Application of Theorem \ref{thm:cpbed_trigonometric_surfaces} --
surfaces]Using trigonometric surfaces of type (\ref{trigonometric_surface})
and applying Theorem \ref{thm:cpbed_trigonometric_surfaces}, Figs.
\ref{fig:torus} and \ref{fig:surface_of_revolution_star} show several control
point constellations for the exact description of the toroidal patch%
\begin{align}
\mathbf{s}\left(  u_{1},u_{2}\right)   &  =\left[
\begin{array}
[c]{c}%
s^{1}\left(  u_{1},u_{2}\right) \\
\\
s^{2}\left(  u_{1},u_{2}\right) \\
\\
s^{3}\left(  u_{1},u_{2}\right)
\end{array}
\right]  =\left[
\begin{array}
[c]{c}%
\left(  3+\frac{15}{8}\left(  \sqrt{5}-1\right)  \sin\left(  u_{1}\right)
\right)  \cos\left(  u_{2}\right) \\
\\
\left(  3+\frac{15}{8}\left(  \sqrt{5}-1\right)  \sin\left(  u_{1}\right)
\right)  \sin\left(  u_{2}\right) \\
\\
\frac{15}{8}\left(  \sqrt{5}-1\right)  \cos\left(  u_{1}\right)
\end{array}
\right]  ,\label{torus}\\
\left(  u_{1},u_{2}\right)   &  \in\left[  0,\frac{3\pi}{4}\right]
\times\left[  0,\frac{\pi}{2}\right] \nonumber
\end{align}
and of the patch%
\begin{align}
\mathbf{s}\left(  u_{1},u_{2}\right)   &  =\left[
\begin{array}
[c]{c}%
s^{1}\left(  u_{1},u_{2}\right) \\
s^{2}\left(  u_{1},u_{2}\right) \\
s^{3}\left(  u_{1},u_{2}\right)
\end{array}
\right]  =\left[
\begin{array}
[c]{c}%
\left(  12+6\sin\left(  u_{1}\right)  -\sin\left(  6u_{1}\right)  \right)
\cos\left(  u_{2}\right) \\
\left(  12+6\sin\left(  u_{1}\right)  -\sin\left(  6u_{1}\right)  \right)
\sin\left(  u_{2}\right) \\
6\cos\left(  u_{1}\right)  +\cos\left(  6u_{1}\right)
\end{array}
\right]  ,\label{surface_of_revolution_star}\\
&  \left(  u_{1},u_{2}\right)  \in \left[  0,\frac{\pi}{2}\right]
\times\left[  0,\frac{2\pi}{3}\right] \nonumber
\end{align}
that also lies on a surface of revolution generated by the rotation of the
hypocycloid%
\begin{equation}
\mathbf{g}\left(  u\right)  =\left[
\begin{array}
[c]{c}%
g^{1}\left(  u\right) \\
g^{2}\left(  u\right)
\end{array}
\right]  =\left[
\begin{array}
[c]{c}%
6\sin\left(  u\right)  -\sin\left(  6u\right) \\
6\cos\left(  u\right)  +\cos\left(  6u\right)
\end{array}
\right]  ,~u\in\left[  0,2\pi\right]  \label{star_hypocycloid}%
\end{equation}
about the axis $z$, respectively.
\end{example}

\begin{figure}
[!h]
\begin{center}
\includegraphics[
natheight=4.094900in,
natwidth=6.298400in,
height=3.6729in,
width=6.3261in
]%
{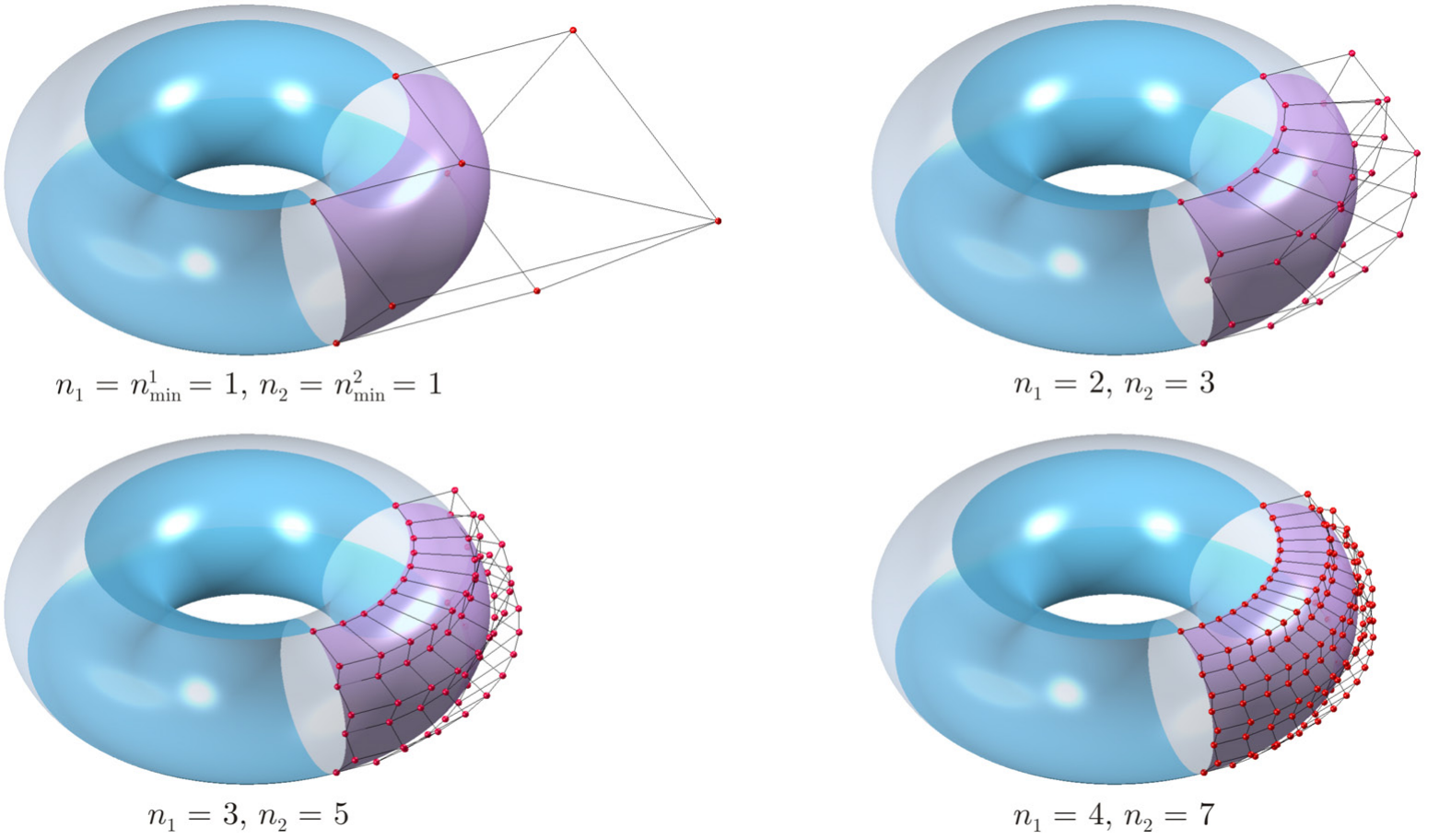}%
\caption{Control point based exact description of the toroidal patch
(\ref{torus}) by means of $3$-dimensional $2$-variate trigonometric surfaces
of the type (\ref{trigonometric_surface}). Control nets corresponding to
different orders were obtained by using Theorem
\ref{thm:cpbed_trigonometric_surfaces} with parameter settings $\delta=2$,
$\kappa=1$; $\alpha_{1}=\frac{3\pi}{4}$, $\alpha_{2}=\frac{\pi}{2}$;
$m_{1}=m_{2}=m_{3}=1$; $P_{1,1,1}=\left\{  0\right\}  $, $P_{1,1,2}=\left\{
1\right\}  $, $Q_{1,1,1}=\left\{  1\right\}  $,~$Q_{1,1,2}=\varnothing$,
$c_{0}^{1,1,1}=3$, $c_{1}^{1,1,2}=1$, $s_{1}^{1,1,1}=\frac{15}{8}\left(
\sqrt{5}-1\right)  $, $\psi_{0}^{1,1,1}=\psi_{1}^{1,1,2}=\varphi_{1}%
^{1,1,1}=0$; $P_{2,1,1}=\left\{  0\right\}  $, $P_{2,1,2}=\varnothing$,
$Q_{2,1,1}=Q_{2,1,2}=\left\{  1\right\}  $, $c_{0}^{2,1,1}=3$, $s_{1}%
^{2,1,1}=\frac{15}{8}\left(  \sqrt{5}-1\right)  $, $s_{1}^{2,1,2}=1$,
$\psi_{0}^{2,1,1}=\varphi_{1}^{2,1,1}=\varphi_{1}^{2,1,2}=0$; $P_{3,1,1}%
=\left\{  1\right\}  $, $P_{3,1,2}=\left\{  0\right\}  $, $Q_{3,1,1}%
=Q_{3,1,2}=\varnothing$, $c_{1}^{3,1,1}=\frac{15}{8}\left(  \sqrt{5}-1\right)
$, $c_{0}^{3,1,2}=1$, $\psi_{0}^{3,1,1}=\psi_{1}^{3,1,2}=0$.}%
\label{fig:torus}%
\end{center}
\end{figure}

\begin{figure}
[!h]
\begin{center}
\includegraphics[
natheight=2.546900in,
natwidth=6.298400in,
height=2.6247in,
width=6.3261in
]%
{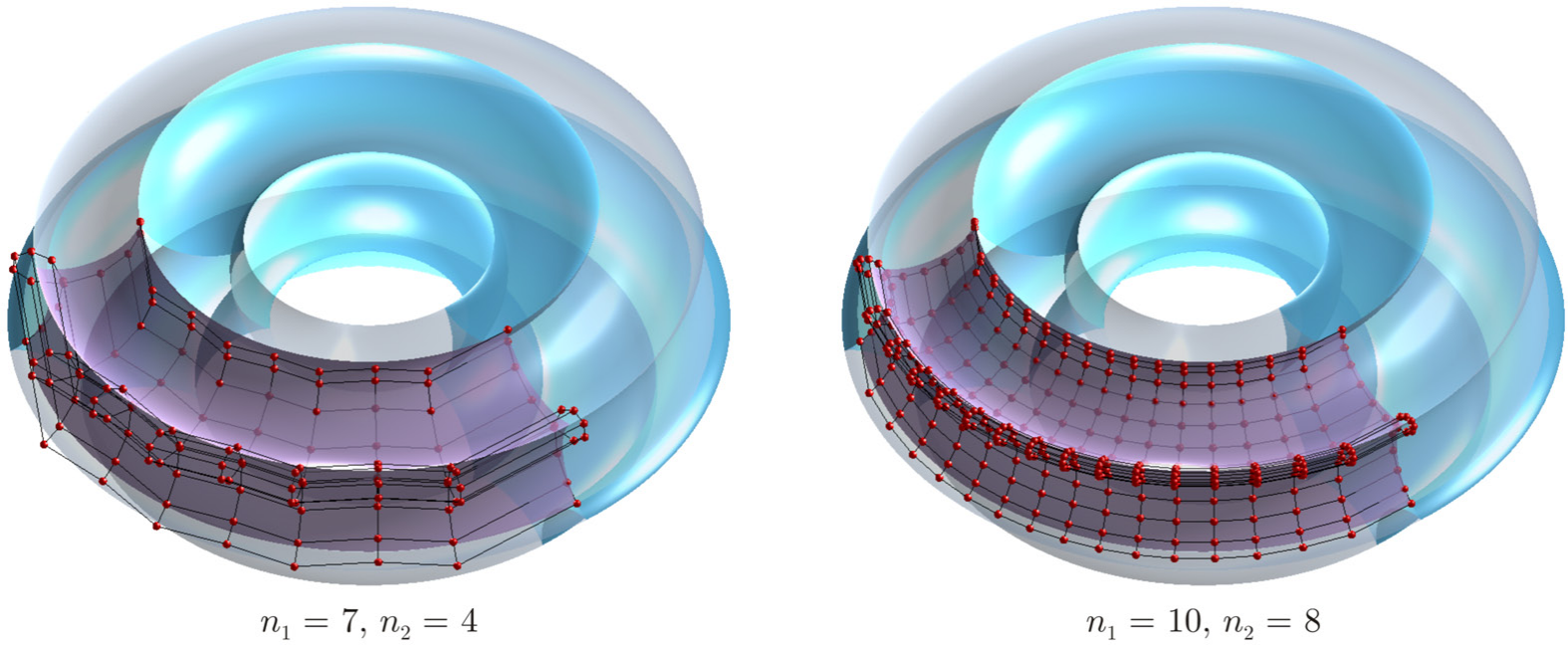}%
\caption{Control point based exact description of the trigonometric patch
(\ref{surface_of_revolution_star}) that lies on a surface of revolution
generated by the rotation of the hypocycloid (\ref{star_hypocycloid}) about
the axis $z$. Control nets were generated by formulae described in Theorem
\ref{thm:cpbed_trigonometric_surfaces} ($\delta=2$, $\kappa=1$; $\alpha
_{1}=\frac{\pi}{2}$, $\alpha_{2}=\frac{2\pi}{3}$; $m_1=m_2=m_3=1$).}%
\label{fig:surface_of_revolution_star}%
\end{center}
\end{figure}

\begin{example}
[Application of Theorem \ref{thm:cpbed_trigonometric_surfaces} -- volumes]%
$3$-dimensional trigonometric volumes can also be exactly described by means
of $3$-variate tensor product surfaces of the type
(\ref{trigonometric_surface}). Figs.
\ref{fig:non_rational_trigonometric_volume} and
\ref{fig:non_rational_trigonometric_volume2} illustrate control grids that
generate the volumes%
\begin{align}
\mathbf{s}\left(  u_{1},u_{2},u_{3}\right)   &  =\left[
\begin{array}
[c]{c}%
s^{1}\left(  u_{1},u_{2},u_{3}\right) \\
s^{2}\left(  u_{1},u_{2},u_{3}\right) \\
s^{3}\left(  u_{1},u_{2},u_{3}\right)
\end{array}
\right] \nonumber\\
&  =\left[
\begin{array}
[c]{c}%
\left(  6+\cos\left(  u_{1}+\frac{\pi}{3}\right)  \right)  \cos\left(
u_{2}-\frac{\pi}{6}\right)  \cos\left(  u_{3}+\frac{\pi}{3}\right) \\
\\
\left(  6+\cos\left(  u_{1}+\frac{\pi}{3}\right)  \right)  \cos\left(
u_{2}-\frac{\pi}{6}\right)  \sin\left(  u_{3}+\frac{\pi}{3}\right) \\
\\
\cos\left(  u_{1}+\frac{\pi}{3}\right)  \sin\left(  u_{2}-\frac{\pi}%
{6}\right)
\end{array}
\right]  ,\label{non_rational_trigonometric_volume1}\\
\left(  u_{1},u_{2},u_{3}\right)   &  \in\left[  0,\frac{\pi}{2}\right]
\times\left[  0,\frac{\pi}{2}\right]  \times\left[  0,\frac{2\pi}{3}\right]
\nonumber
\end{align}
and%
\begin{align}
\mathbf{s}\left(  u_{1},u_{2},u_{3}\right)   &  =\left[
\begin{array}
[c]{c}%
s^{1}\left(  u_{1},u_{2},u_{3}\right) \\
s^{2}\left(  u_{1},u_{2},u_{3}\right) \\
s^{3}\left(  u_{1},u_{2},u_{3}\right)
\end{array}
\right] \nonumber\\
&  =\left[
\begin{array}
[c]{c}%
\left(  2+\frac{3}{4}\sin\left(  u_{1}\right)  -\frac{1}{4}\sin\left(
3u_{1}\right)  \right)  \cos\left(  u_{2}\right)  \left(  \frac{3}{2}-\frac
{1}{2}\cos\left(  2u_{3}\right)  \right) \\
\\
\left(  \frac{5}{2}-\frac{1}{2}\cos\left(  2u_{1}\right)  \right)  \sin\left(
u_{2}\right)  \left(  1+\frac{3}{4}\sin\left(  u_{3}\right)  -\frac{1}{4}%
\sin\left(  3u_{3}\right)  \right) \\
\\
\left(  \frac{1}{2}+\frac{3}{8}\cos\left(  u_{1}\right)  +\frac{1}{2}%
\cos\left(  2u_{1}\right)  +\frac{1}{8}\cos\left(  3u_{1}\right)  \right)
\left(  \frac{3}{2}+\cos\left(  u_{2}\right)  +\sin\left(  u_{3}\right)
\right)
\end{array}
\right]  ,\label{non_rational_trigonometric_volume2}\\
\left(  u_{1},u_{2},u_{3}\right)   &  \in\left[  0,\frac{\pi}{2}\right]
\times\left[  0,\frac{2\pi}{3}\right]  \times\left[  0,\frac{\pi}{2}\right]
,\nonumber
\end{align}
respectively.
\end{example}

\begin{figure}
[!h]
\begin{center}
\includegraphics[
natheight=2.801100in,
natwidth=6.298400in,
height=2.8011in,
width=6.2984in
]%
{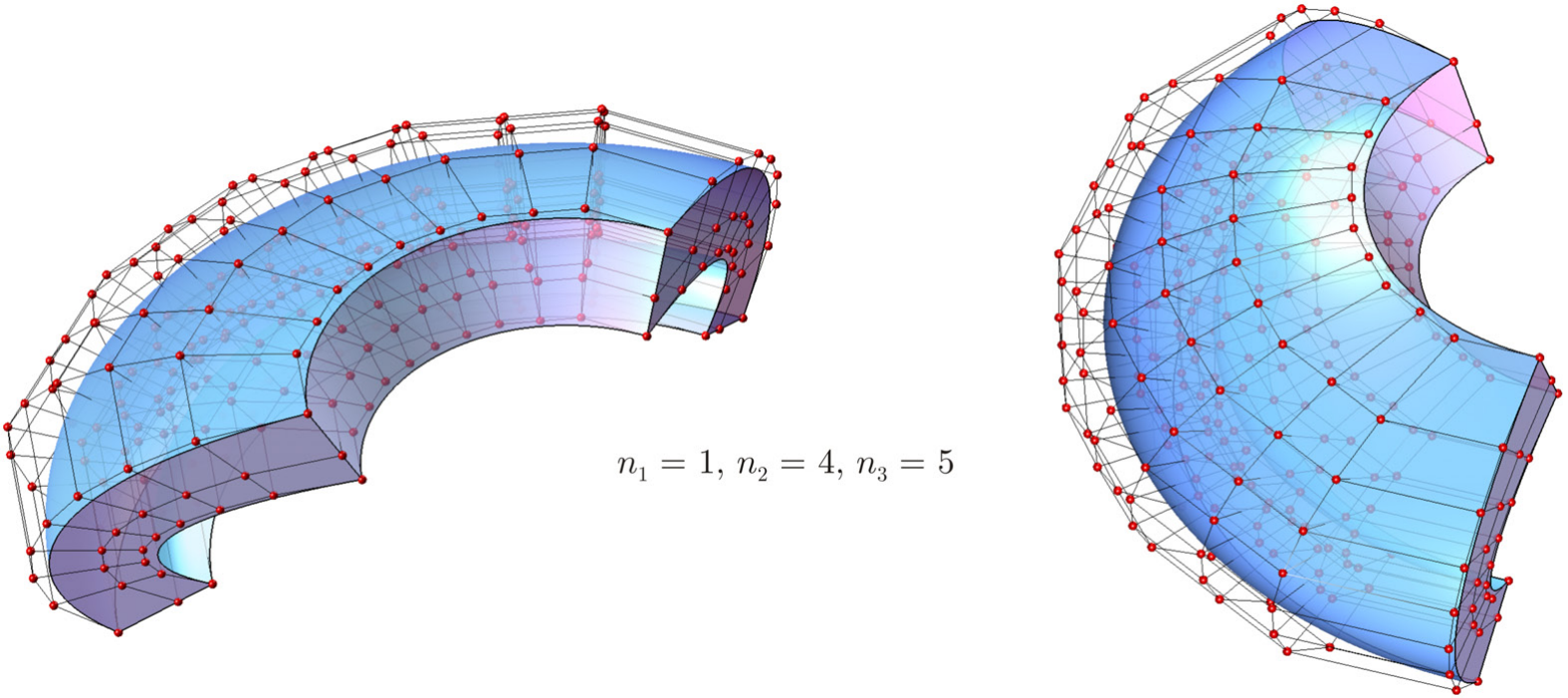}%
\caption{Different views of the same $3$-dimensional trigonometric volume
(\ref{non_rational_trigonometric_volume1}) along with its control grid
calculated by means of Theorem \ref{thm:cpbed_trigonometric_surfaces}
($\delta=3$, $\kappa=0$; $\alpha_{1}=\frac{\pi}{2}$, $\alpha_{2}=\frac{\pi}%
{2}$, $\alpha_{3}=\frac{2\pi}{3}$; $m_1=m_2=m_3=1$).}%
\label{fig:non_rational_trigonometric_volume}%
\end{center}
\end{figure}

\begin{figure}
[!h]
\begin{center}
\includegraphics[
natheight=2.633300in,
natwidth=6.138400in,
height=2.6602in,
width=6.1661in
]%
{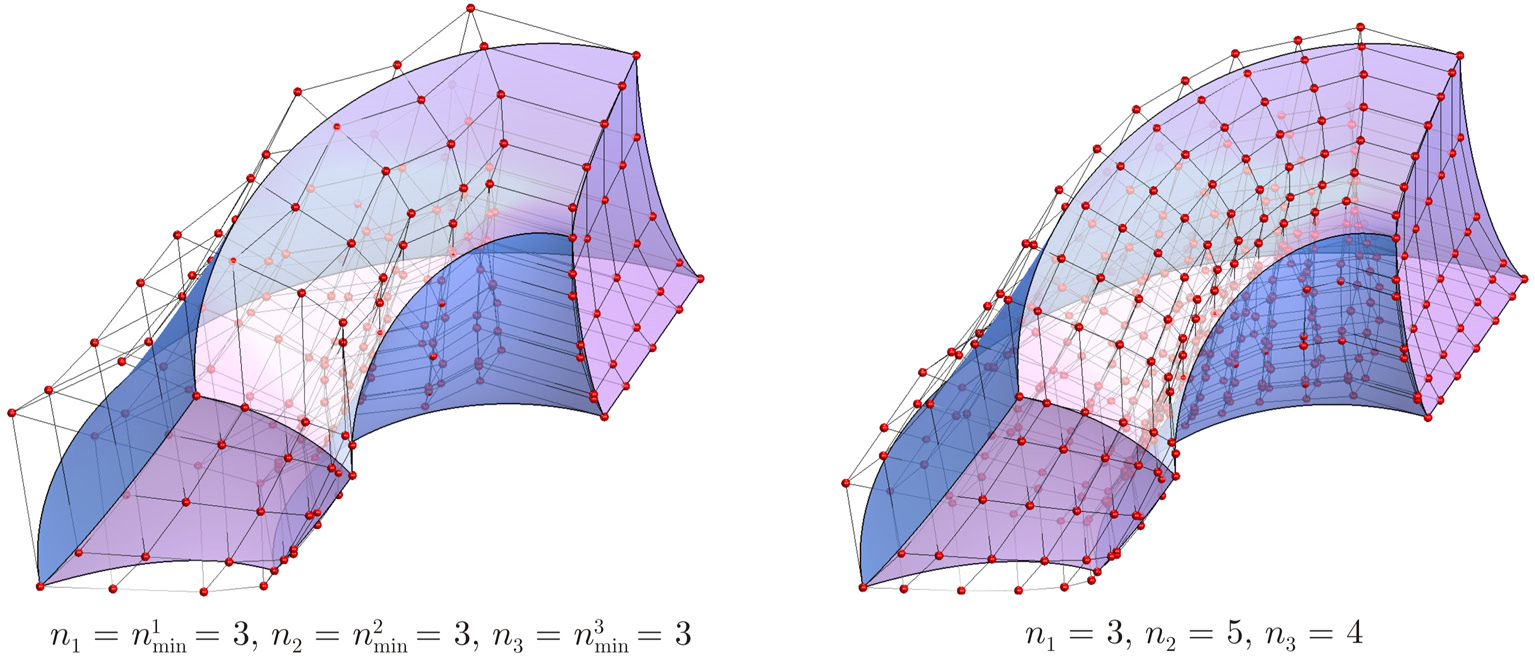}%
\caption{Control point based exact description of the $3$-dimensional
trigonometric volume (\ref{non_rational_trigonometric_volume2}) by means of
$3$-variate tensor product surfaces of the type (\ref{trigonometric_surface})
with different orders. Control grids were obtained by using Theorem
\ref{thm:cpbed_trigonometric_surfaces} ($\delta=3$, $\kappa=0$; $\alpha
_{1}=\frac{\pi}{2}$, $\alpha_{2}=\frac{2\pi}{3}$, $\alpha_{3}=\frac{\pi}{2}%
$; $m_1=m_2=1$, $m_3=2$).}%
\label{fig:non_rational_trigonometric_volume2}%
\end{center}
\end{figure}

Theorem \ref{thm:cpbed_trigonometric_surfaces} can also be used to provide
control point based exact description of the rational trigonometric surface%
\begin{equation}
\mathbf{s}\left(  \mathbf{u}\right)  =\frac{1}{s^{\delta+\kappa+1}\left(
\mathbf{u}\right)  }\left[
\begin{array}
[c]{cccc}%
s^{1}\left(  \mathbf{u}\right)  & s^{2}\left(  \mathbf{u}\right)  & \cdots &
s^{\delta+\kappa}\left(  \mathbf{u}\right)
\end{array}
\right]  \in%
\mathbb{R}
^{\delta+\kappa}, \label{rational_trigonometric_surface_to_be_described}%
\end{equation}
where%
\begin{align*}
\mathbf{u}  &  =\left[  u_{j}\right]  _{j=1}^{\delta}\in\times_{j=1}^{\delta
}\left[  0,\alpha_{j}\right]  ,~\alpha_{j}\in\left(  0,\pi\right)
,~\kappa\geq0,\\
s^{\ell}\left(  \mathbf{u}\right)   &  =\sum_{\zeta=1}^{m_{\ell}}\prod
_{j=1}^{\delta}\left(  \sum_{p\in P_{\ell,\zeta,j}}c_{p}^{\ell,\zeta,j}%
\cos\left(  pu_{j}+\psi_{p}^{\ell,\zeta,j}\right)  +\sum_{q\in Q_{\ell
,\zeta,j}}s_{q}^{\ell,\zeta,j}\sin\left(  qu_{j}+\varphi_{q}^{\ell,\zeta
,j}\right)  \right)  ,\\
\ell &  =1,2,\ldots,\delta+\kappa+1,\\
P_{\ell,\zeta,j},Q_{\ell,\zeta,j}  &  \subset%
\mathbb{N}
,~m_{\ell}\in%
\mathbb{N}
\setminus\left\{  0\right\}  ,~c_{p}^{\ell,\zeta,j},\psi_{p}^{\ell,\zeta
,j},s_{q}^{\ell,\zeta,j},\varphi_{q}^{\ell,\zeta,j}\in%
\mathbb{R}%
\end{align*}
and%
\[
s^{\delta+\kappa+1}\left(  \mathbf{u}\right)  >0,~\forall\mathbf{u}\in
\times_{j=1}^{\delta}\left[  0,\alpha_{j}\right]  .
\]
Similarly to Algorithm \ref{alg:cpbed_rational_trigonometric_curves} one can
formulate the next process.

\begin{algorithm}
[Control point based exact description of rational trigonometric
surfaces]\label{alg:cpbed_rational_trigonometric_surfaces}Operations that
ensure the control point based exact description of the surface
(\ref{rational_trigonometric_surface_to_be_described}) are as follows:

\begin{itemize}
\item let%
\[
n_{j}\geq n_{\min}^{j}=\max\left\{  z_{j}:z_{j}\in\cup_{\ell=1}^{\delta+\kappa+1}%
\cup_{\zeta=1}^{m_{\ell}}\left(  P_{\ell,\zeta,j}\cup Q_{\ell,\zeta,j}\right)
\right\}  ,~j=1,2,\ldots,\delta
\]
be arbitrarily fixed orders in directions $u_{1},u_{2},\ldots,u_{\delta}$;

\item apply Theorem \ref{thm:cpbed_trigonometric_surfaces} to the pre-image
\begin{equation}
\mathbf{s}_{\wp}\left(  \mathbf{u}\right)  =\left[  s^{\ell}\left(
\mathbf{u}\right)  \right]  _{\ell=1}^{\delta+\kappa+1}\in%
\mathbb{R}
^{\delta+\kappa+1},~\mathbf{u}=\left[  u_{j}\right]  _{j=1}^{\delta}\in
\times_{j=1}^{\delta}\left[  0,\alpha_{j}\right]
\label{pre_image_of_rational_trigonometric_surface_to_be_described}%
\end{equation}
of the surface (\ref{rational_trigonometric_surface_to_be_described}), i.e.,
compute control points%
\[
\mathbf{d}_{i_{1},i_{2},\ldots,i_{\delta}}^{\wp}=\left[  d_{i_{1},i_{2}%
,\ldots,i_{\delta}}^{\ell}\right]  _{\ell=1}^{\delta+\kappa+1}\in%
\mathbb{R}
^{\delta+\kappa+1},~i_{j}=0,1,\ldots,2n_{j},~j=1,2,\ldots,\delta
\]
for the exact trigonometric representation of
(\ref{pre_image_of_rational_trigonometric_surface_to_be_described}) in the
pre-image space $%
\mathbb{R}
^{\delta+\kappa+1}$;

\item project the obtained control points onto the hyperplane $x^{\delta
+\kappa+1}=1$ that results the control points
\[
\mathbf{d}_{i_{1},i_{2},\ldots,i_{\delta}}=\frac{1}{d_{i_{1},i_{2}%
,\ldots,i_{\delta}}^{\delta+\kappa+1}}\left[  d_{i_{1},i_{2},\ldots,i_{\delta
}}^{\ell}\right]  _{\ell=1}^{\delta+\kappa}\in%
\mathbb{R}
^{\delta+\kappa},~i_{j}=0,1,\ldots,2n_{j},~j=1,2,\ldots,\delta
\]
and weights%
\[
\omega_{i_{1},i_{2},\ldots,i_{\delta}}=d_{i_{1},i_{2},\ldots,i_{\delta}%
}^{\delta+\kappa+1},~i_{j}=0,1,\ldots,2n_{j},~j=1,2,\ldots,\delta
\]
needed for the rational trigonometric representation
(\ref{trigonometric_rational_surface}) of
(\ref{rational_trigonometric_surface_to_be_described});

\item if not all weights are positive, try to increase the components of the
order $\left(  n_{1},n_{2},\ldots,n_{\delta}\right)  $ of the trigonometric
surface used for the control point based exact description of the pre-image
$\mathbf{s}_{\wp}$ in $%
\mathbb{R}
^{\delta+\kappa+1}$ and repeat the previous projectional and weight
determination step.
\end{itemize}
\end{algorithm}

\begin{example}
[Application of Algorithm \ref{alg:cpbed_rational_trigonometric_surfaces} --
rational surfaces]Using surfaces of the type
(\ref{trigonometric_rational_surface}), Fig.
\ref{fig:rational_trigonometric_surface}\ shows the control point based exact
description of the rational trigonometric patch%
\begin{equation}
\mathbf{s}\left(  u_{1},u_{2}\right)  =\frac{1}{s^{4}\left(  u_{1}%
,u_{2}-\gamma\right)  }\left[
\begin{array}
[c]{c}%
s^{1}\left(  u_{1},u_{2}-\gamma\right) \\
s^{2}\left(  u_{1},u_{2}-\gamma\right) \\
s^{3}\left(  u_{1},u_{2}-\gamma\right)
\end{array}
\right]  ,~\left(  u_{1},u_{2}\right)  \in\left[  0,\frac{\pi}{4}\right]
\times\left[  0,\frac{\pi}{3}\right]  , \label{rational_trigonometric_8}%
\end{equation}
where%
\begin{align*}
s^{1}\left(  u_{1},u_{2}\right)  =  &  \cos\left(  u_{1}\right)  \left(
173\cos\left(  u_{2}\right)  -10\sin\left(  3u_{2}\right)  -10\sin\left(
5u_{2}\right)  +\cos\left(  7u_{2}\right)  +\cos\left(  9u_{2}\right)  \right)
\\
&  +\sin\left(  u_{1}\right)  \left(  -173\sin\left(  u_{2}\right)
+10\cos\left(  3u_{2}\right)  -10\cos\left(  5u_{2}\right)  +\sin\left(
7u_{2}\right)  -\sin\left(  9u_{2}\right)  \right)  ,\\
& \\
s^{2}\left(  u_{1},u_{2}\right)  =  &  \cos\left(  u_{1}\right)  \left(
173\sin\left(  u_{2}\right)  -10\cos\left(  3u_{2}\right)  +10\cos\left(
5u_{2}\right)  -\sin\left(  7u_{2}\right)  +\sin\left(  9u_{2}\right)  \right)
\\
&  +\sin\left(  u_{1}\right)  \left(  173\cos\left(  u_{2}\right)
-10\sin\left(  3u_{2}\right)  -10\sin\left(  5u_{2}\right)  +\cos\left(
7u_{2}\right)  +\cos\left(  9u_{2}\right)  \right)  ,\\
& \\
s^{3}\left(  u_{1},u_{2}\right)  =  &  20\cos\left(  u_{1}\right)  \left(
5\cos\left(  u_{2}\right)  +\sin\left(  3u_{2}\right)  +\sin\left(
5u_{2}\right)  \right) \\
&  +20\sin\left(  u_{1}\right)  \left(  5\sin\left(  u_{2}\right)
+\cos\left(  3u_{2}\right)  -\cos\left(  5u_{2}\right)  \right)  ,\\
& \\
s^{4}\left(  u_{1},u_{2}\right)  =  &  20\cos\left(  u_{1}\right)  \left(
5\sin\left(  u_{2}\right)  +\cos\left(  3u_{2}\right)  -\cos\left(
5u_{2}\right)  \right) \\
&  -20\sin\left(  u_{1}\right)  \left(  5\cos\left(  u_{2}\right)
+\sin\left(  3u_{2}\right)  +\sin\left(  5u_{2}\right)  \right)  +200
\end{align*}
and $\gamma=\frac{\pi}{3}$.
\end{example}

%

\begin{figure}
[!h]
\begin{center}
\includegraphics[
natheight=2.753600in,
natwidth=6.298400in,
height=2.7812in,
width=6.3261in
]%
{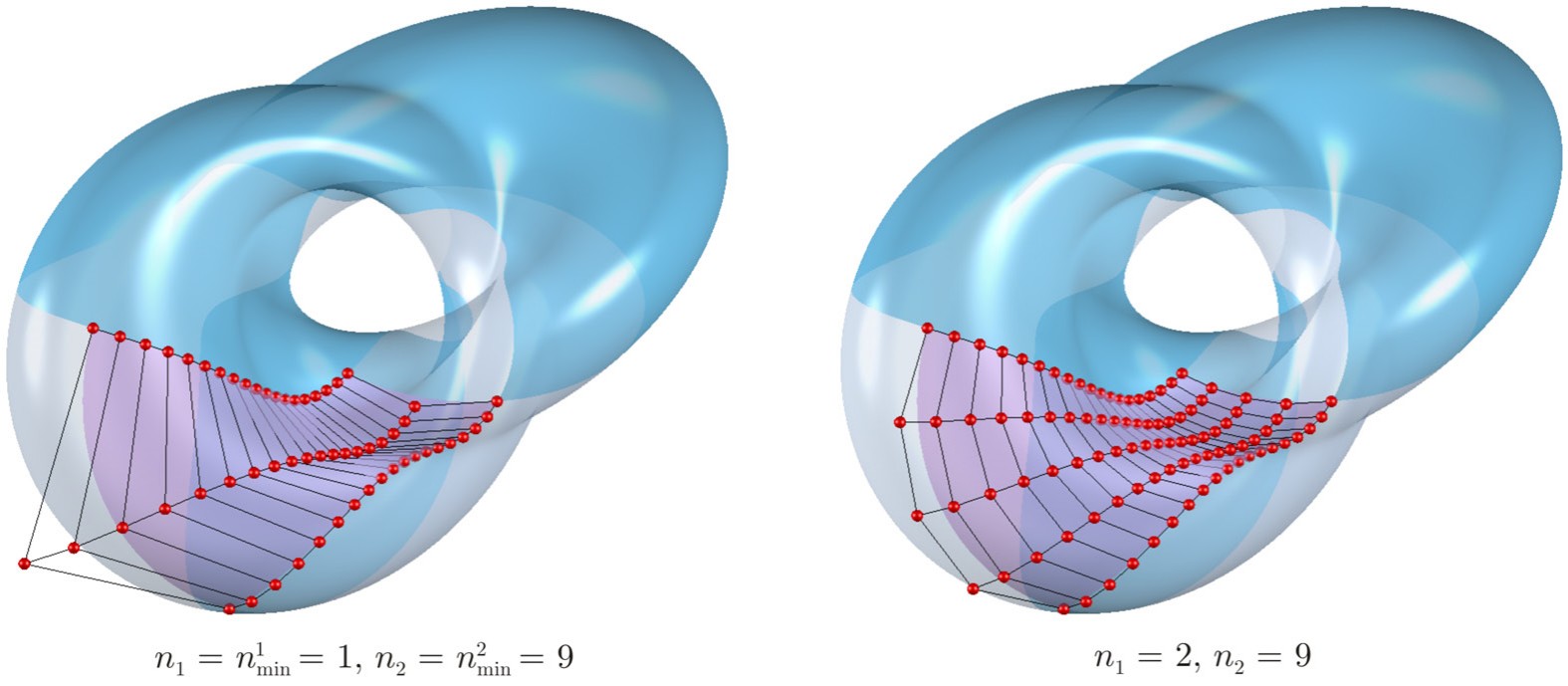}%
\caption{Control point based exact description of the patch
(\ref{rational_trigonometric_8}) by means of $3$-dimensional $2$-variate
rational trigonometric patches of different orders. Control nets were obtained
by following the steps of Algorithm
\ref{alg:cpbed_rational_trigonometric_surfaces} ($\delta=2$, $\kappa=1$;
$\alpha_{1}=\frac{\pi}{4}$, $\alpha_{2}=\frac{\pi}{3}$; $m_{1}=m_{2}=m_{3}=2$,
$m_{4}=3$).}%
\label{fig:rational_trigonometric_surface}%
\end{center}
\end{figure}

\subsection{Description of (rational) hyperbolic curves and
surfaces\label{sec:exact_description_hyperbolic}}

Assume that the smooth parametric curve%
\begin{equation}
\mathbf{g}\left(  u\right)  =\left[  g^{\ell}\left(  u\right)  \right]
_{\ell=1}^{\delta},~u\in\left[  0,\alpha\right]  ,~\alpha>0
\label{traditional_hyperbolic_curve}%
\end{equation}
has coordinate functions of the form%
\[
g^{\ell}\left(  u\right)  =\sum_{p\in P_{\ell}}c_{p}^{\ell}\cosh\left(
pu+\psi_{p}^{\ell}\right)  +\sum_{q\in Q_{\ell}}s_{q}^{\ell}\sinh\left(
qu+\varphi_{q}^{\ell}\right)  ,
\]
where $P_{\ell},Q_{\ell}\subset%
\mathbb{N}
$ and $c_{p}^{\ell},\psi_{p}^{\ell},s_{q}^{\ell},\varphi_{q}^{\ell}\in%
\mathbb{R}
$ and consider the vector space (\ref{hyperbolic_vector_space}) of order
\begin{equation}
n\geq n_{\min}=\max\left\{  z:z\in\cup_{\ell=1}^{\delta}\left(  P_{\ell}\cup
Q_{\ell}\right)  \right\}  . \label{minimal_hyperbolic_order}%
\end{equation}

Using Lemma \ref{lem:hp} and performing calculations similar to the proof of
Theorem \ref{thm:cpbed_trigonometric_curves} one obtains the next statement.

\begin{theorem}
[Control point based exact description of hyperbolic curves]%
\label{thm:cpbed_hyperbolic_curves}For any arbitrarily fixed order
(\ref{minimal_hyperbolic_order}) the curve (\ref{traditional_hyperbolic_curve}%
) given in traditional hyperbolic parametric form has a unique control point
based exact description, more precisely one has that
\[
\frac{\text{\emph{d}}^{r}}{\text{\emph{d}}u^{r}}g^{\ell}\left(  u\right)
=\sum_{i=0}^{2n}d_{i}^{\ell}\left(  r\right)  H_{2n,i}^{\alpha}\left(
u\right)  ,
\]
where%
\[
d_{i}^{\ell}\left(  r\right)  =\left\{
\begin{array}
[c]{ll}%
{\displaystyle\sum\limits_{p\in P_{\ell}}}
c_{p}^{\ell}p^{r}\left(  \rho_{p,i}^{n}\cosh\left(  \psi_{p}^{\ell}\right)
+\sigma_{p,i}^{n}\sinh\left(  \psi_{p}^{\ell}\right)  \right)  & \\
+%
{\displaystyle\sum\limits_{q\in Q_{\ell}}}
s_{q}^{\ell}q^{r}\left(  \sigma_{q,i}^{n}\cosh\left(  \varphi_{q}^{\ell
}\right)  +\rho_{q,i}^{n}\sinh\left(  \varphi_{q}^{\ell}\right)  \right)  , &
r=2z,\\
& \\%
{\displaystyle\sum\limits_{p\in P_{\ell}}}
c_{p}^{\ell}p^{r}\left(  \sigma_{p,i}^{n}\cosh\left(  \psi_{p}^{\ell}\right)
+\rho_{p,i}^{n}\sinh\left(  \psi_{p}^{\ell}\right)  \right)  & \\
+%
{\displaystyle\sum\limits_{q\in Q_{\ell}}}
s_{q}^{\ell}q^{r}\left(  \rho_{q,i}^{n}\cosh\left(  \varphi_{q}^{\ell}\right)
+\sigma_{q,i}^{n}\sinh\left(  \varphi_{q}^{\ell}\right)  \right)  , & r=2z+1
\end{array}
\right.
\]
denotes the $\ell$th coordinate of the $i$th control point $\mathbf{d}_{i}$
needed for the hyperbolic curve description (\ref{hyperbolic_curve}).
\end{theorem}

\begin{example}
[Application of Theorem \ref{thm:cpbed_hyperbolic_curves} -- curves]Fig.
\ref{fig:equilateral_hyperbola} shows the control point based description of
the arc%
\begin{equation}
\mathbf{g}\left(  u\right)  =\left[
\begin{array}
[c]{c}%
g^{1}\left(  u\right) \\
\\
g^{2}\left(  u\right)
\end{array}
\right]  =\left[
\begin{array}
[c]{r}%
\sinh\left(  u-\frac{3}{2}\right) \\
\\
\cosh\left(  u-\frac{3}{2}\right)
\end{array}
\right]  ,~u\in\left[  0,3\right]  \label{equilateral_hyperbola}%
\end{equation}
of an equilateral hyperbola.
\end{example}

%

\begin{figure}
[!h]
\begin{center}
\includegraphics[
height=2.7121in,
width=5.028in
]%
{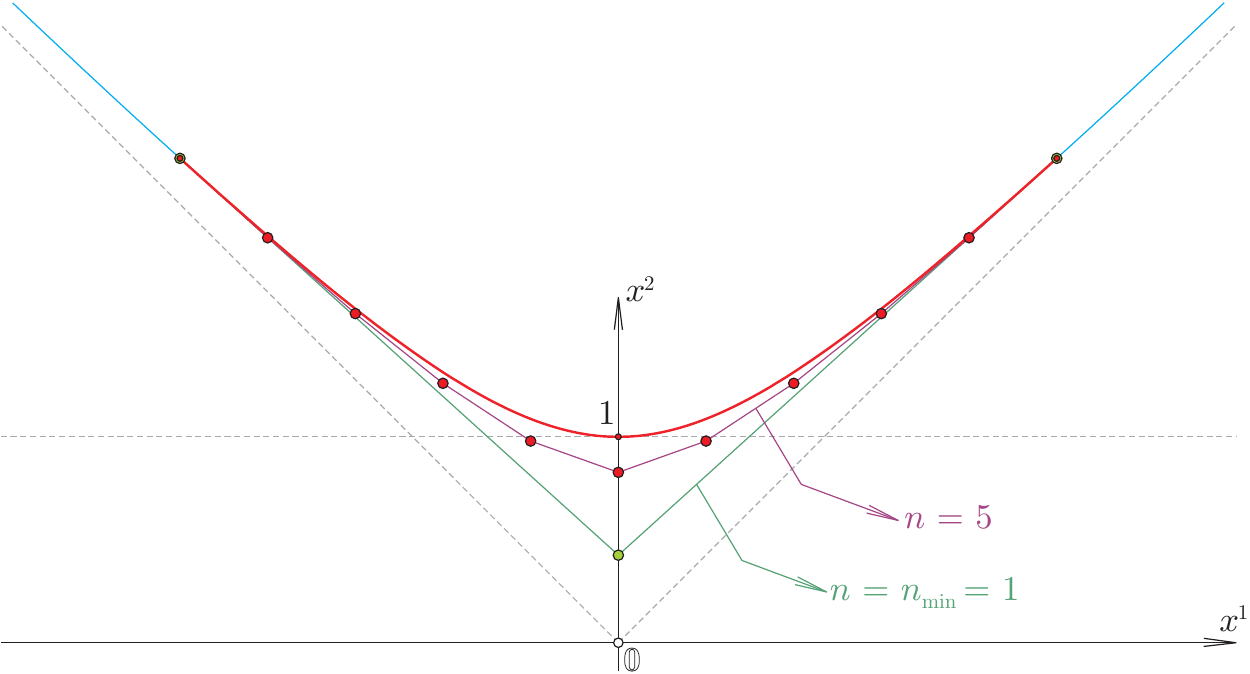}%
\caption{Using Theorem \ref{thm:cpbed_hyperbolic_curves}, the image shows the
control point based exact description of the hyperbolic arc
(\ref{equilateral_hyperbola}) by means of hyperbolic curves of the type
(\ref{hyperbolic_curve}) of varying order and fixed shape parameter $\alpha
=3$.}%
\label{fig:equilateral_hyperbola}%
\end{center}
\end{figure}

\begin{remark}
[Hyperbolic counterpart of Theorem \ref{thm:cpbed_trigonometric_surfaces}%
]Higher order (mixed) partial derivatives of a non-rational higher dimensional
multivariate hyperbolic surface can also be exactly described by means of
Theorem \ref{thm:cpbed_hyperbolic_curves}; one would simply obtain the
hyperbolic counterpart of Theorem \ref{thm:cpbed_trigonometric_surfaces}.
Moreover, one can combine Theorems \ref{thm:cpbed_trigonometric_curves} and
\ref{thm:cpbed_hyperbolic_curves} in order to exactly describe patches of
hybrid multivariate surfaces%
\[
\mathbf{s}\left(  \mathbf{u}\right)  =\left[
\begin{array}
[c]{cccc}%
s^{1}\left(  \mathbf{u}\right)  & s^{2}\left(  \mathbf{u}\right)  & \cdots &
s^{\delta+\kappa}\left(  \mathbf{u}\right)
\end{array}
\right]  \in%
\mathbb{R}
^{\delta+\kappa},~\mathbf{u}=\left[  u_{j}\right]  _{j=1}^{\delta}\in
\times_{j=1}^{\delta}\left[  0,\alpha_{j}\right]  ,~\alpha_{j}\in\left(
0,\beta_{j}\right)  ,~\kappa\geq0
\]
that are either trigonometric or hyperbolic in case of a fixed direction
$u_{j}$, where $\beta_{j}$ is either $\pi$ or $+\infty$ depending on the
trigonometric or hyperbolic type of the coordinate function $s^{j}$,
respectively. (Along a selected direction each coordinate function must be of
the same type).
\end{remark}

\begin{example}
[Combination of Theorems \ref{thm:cpbed_trigonometric_curves} and
\ref{thm:cpbed_hyperbolic_curves} \ -- hybrid surfaces]Fig. \ref{fig:catenoid}
illustrates the control point based exact description of the patch%
\begin{equation}
\mathbf{s}\left(  u_{1},u_{2}\right)  =\left[
\begin{array}
[c]{c}%
s^{1}\left(  u_{1},u_{2}\right) \\
\\
s^{2}\left(  u_{1},u_{2}\right) \\
\\
s^{3}\left(  u_{1},u_{2}\right)
\end{array}
\right]  =\left[
\begin{array}
[c]{c}%
\left(  1+\cosh\left(  u_{1}-\frac{3}{2}\right)  \right)  \sin\left(
u_{2}\right) \\
\\
\left(  1+\cosh\left(  u_{1}-\frac{3}{2}\right)  \right)  \cos\left(
u_{2}\right) \\
\\
\sinh\left(  u_{1}-\frac{3}{2}\right)
\end{array}
\right]  ,~\left(  u_{1},u_{2}\right)  \in\left[  0,3\right]  \times\left[
0,\frac{2\pi}{3}\right]  \label{catenoid}%
\end{equation}
that lies on a surface of revolution (also called hyperboloid) obtained
by the rotation of the equilateral hyperbolic arc%
\begin{equation}
\mathbf{g}\left(  u\right)  =\left[
\begin{array}
[c]{c}%
g^{1}\left(  u\right) \\
\\
g^{2}\left(  u\right)
\end{array}
\right]  =\left[
\begin{array}
[c]{r}%
\cosh\left(  u-\frac{3}{2}\right) \\
\\
\sinh\left(  u-\frac{3}{2}\right)
\end{array}
\right]  ,~u\in\left[  0,3\right]  \label{catenary}%
\end{equation}
along the axis $z$.
\end{example}

\begin{figure}
[!h]
\begin{center}
\includegraphics[
natheight=3.164400in,
natwidth=6.298400in,
height=3.192in,
width=6.3261in
]%
{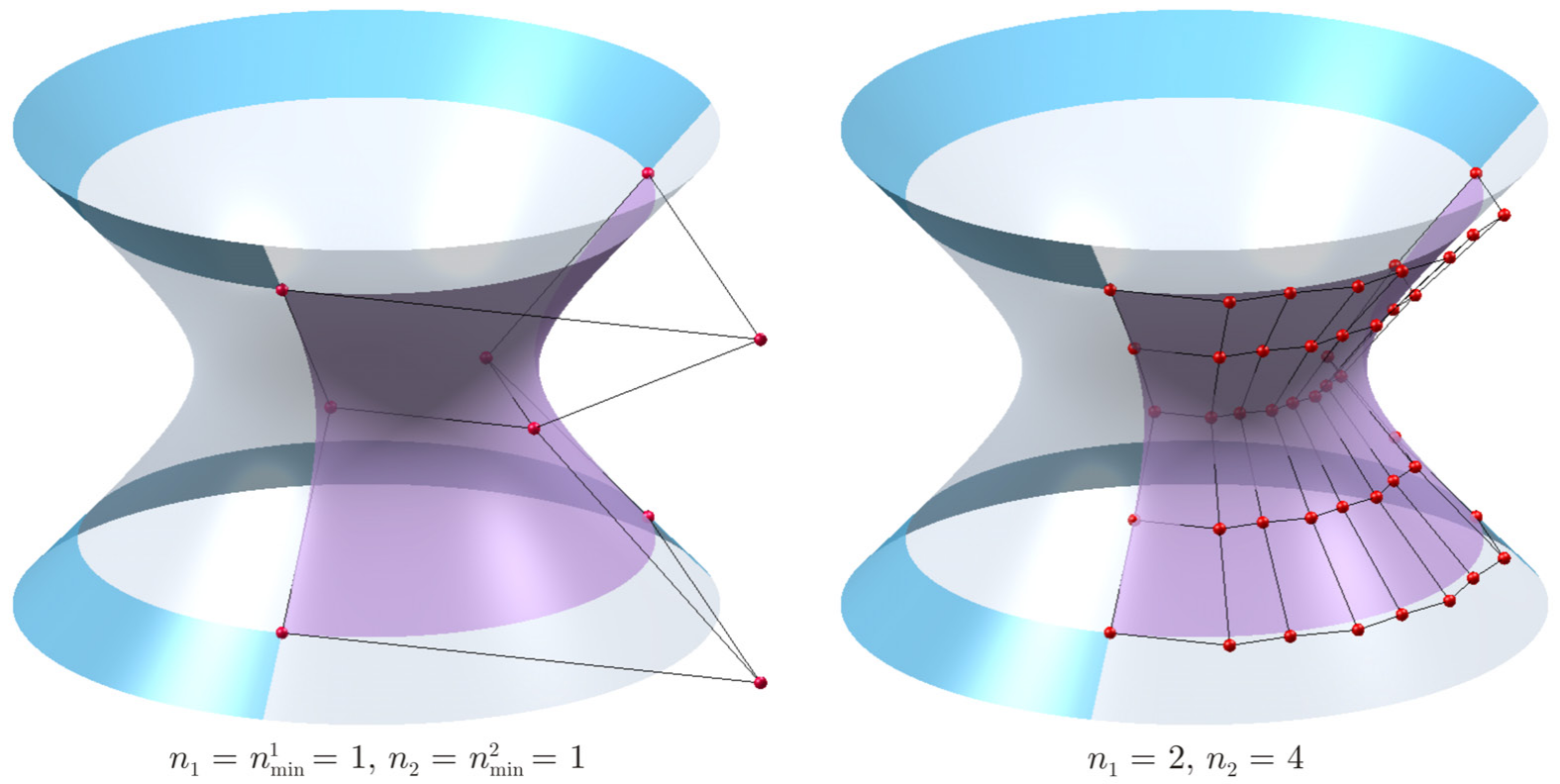}%
\caption{Control point based exact description of the hyperboloidal patch
(\ref{catenoid}) with hybrid surfaces of different orders. In order to
formulate the hybrid variant of Theorem \ref{thm:cpbed_trigonometric_surfaces}%
, in directions $u_{1}$ and $u_{2}$ the results of Theorems
\ref{thm:cpbed_hyperbolic_curves} ($\alpha_{1}=3$) and
\ref{thm:cpbed_trigonometric_curves} ($\alpha_{2}=\frac{2\pi}{3}$) were
applied ($m_1=m_2=m_3=1$), respectively.}%
\label{fig:catenoid}%
\end{center}
\end{figure}

\begin{remark}
[Hyperbolic counterpart of Algorithms
\ref{alg:cpbed_rational_trigonometric_curves} and
\ref{alg:cpbed_rational_trigonometric_surfaces}]Any smooth rational hyperbolic
curve/surface that is given in traditional parametric form (with a
non-vanishing function in its denominator) can also be exactly described by
means of rational hyperbolic curves/surfaces of the type
(\ref{rational_hyperbolic_curve})/(\ref{rational_hyperbolic_surface}); one
simply has to apply the hyperbolic counterpart of Algorithms
\ref{alg:cpbed_rational_trigonometric_curves} or
\ref{alg:cpbed_rational_trigonometric_surfaces}. Moreover, combining the
presented trigonometric algorithms and their hyperbolic counterparts, higher
dimensional hybrid multivariate rational surfaces can also exactly described
by means of multivariate hybrid tensor product surfaces.
\end{remark}

\begin{example}
[Applying the hyperbolic counterpart of Algorithm
\ref{alg:cpbed_rational_trigonometric_curves} ]Cases (a) and (b) of Fig.
\ref{fig:ed_rational_hpc}\ show the control point based exact description of
the rational hyperbolic arcs%
\begin{equation}
\mathbf{g}\left(  u\right)  =\frac{1}{g^{3}\left(  u\right)  }\left[
\begin{array}
[c]{c}%
g^{1}\left(  u\right)  \\
g^{2}\left(  u\right)
\end{array}
\right]  =\frac{1}{4+3\cosh\left(  u-1\right)  +\cosh\left(  3u-3\right)
}\left[
\begin{array}
[c]{r}%
4\cosh\left(  2u-2\right)  \\
8\sinh\left(  u-1\right)
\end{array}
\right]  ,~u\in\left[  0,3.1\right]  \label{rational_hyperbolic_curve_1}%
\end{equation}
and%
\begin{equation}
\mathbf{g}\left(  u\right)  =\frac{1}{g^{3}\left(  u\right)  }\left[
\begin{array}
[c]{c}%
g^{1}\left(  u\right)  \\
\\
g^{2}\left(  u\right)
\end{array}
\right]  =\frac{1}{11+4\cosh\left(  2u-\frac{3}{2}\right)  +\cosh\left(
4u-3\right)  }\left[
\begin{array}
[c]{c}%
16\cosh\left(  u-\frac{3}{4}\right)  \\
\\
4\sinh\left(  2u-\frac{3}{2}\right)
\end{array}
\right]  ,~u\in\left[  0,2.5\right]  \label{rational_hyperbolic_curve_2}%
\end{equation}
respectively. (Note that in both cases $\lim\limits_{u\rightarrow\pm\infty
}\mathbf{g}\left(  u\right)  =\mathbf{0}$.)
\end{example}

%

\begin{figure}
[!h]
\begin{center}
\includegraphics[
height=3.2162in,
width=6.1004in
]%
{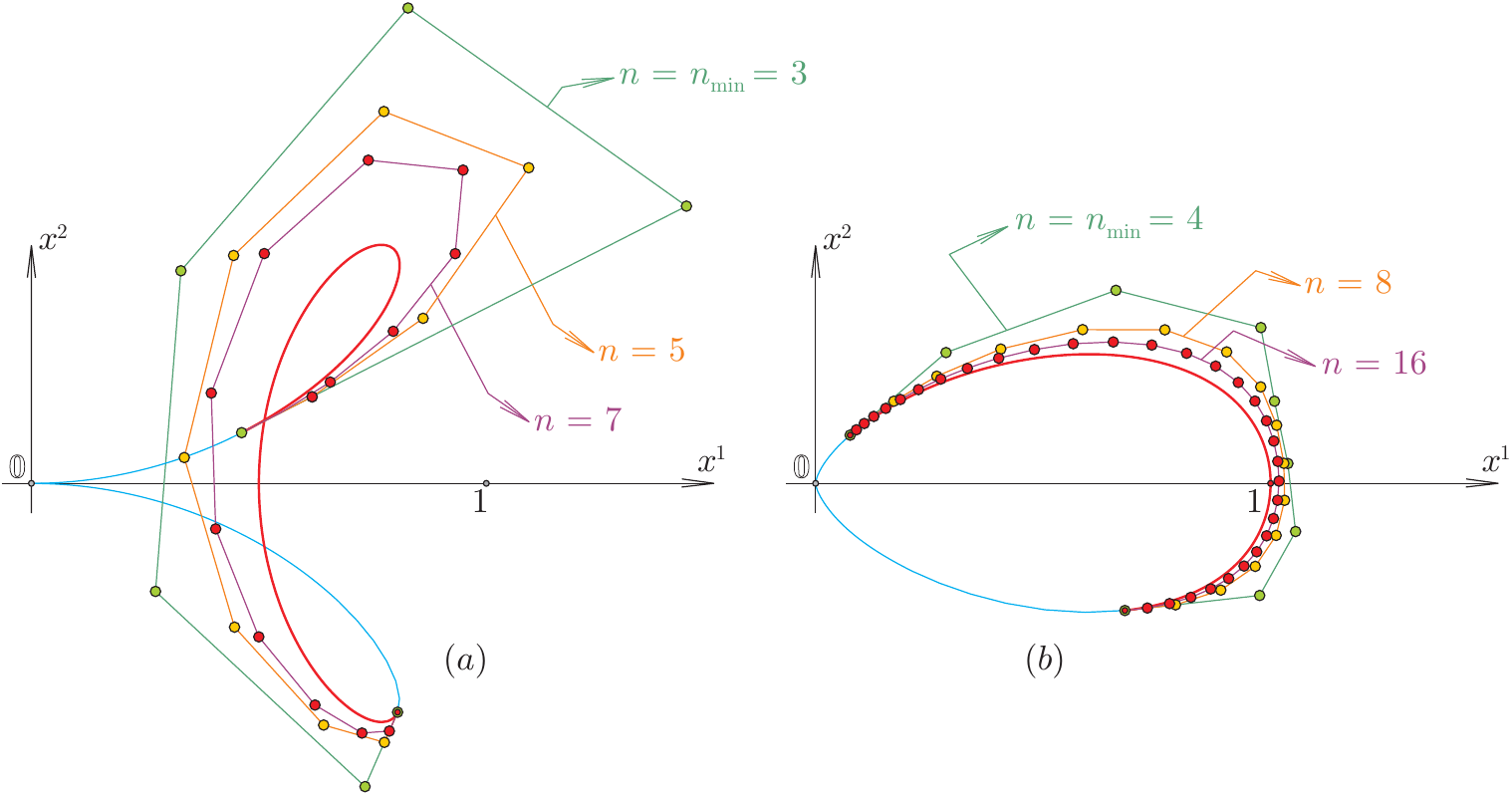}%
\caption{Using rational hyperbolic curves of the type
(\ref{rational_hyperbolic_curve}), cases (\emph{a}) and (\emph{b}) illustrate
the control point based exact description of arcs
(\ref{rational_hyperbolic_curve_1}) and (\ref{rational_hyperbolic_curve_2}),
respectively. Control polygons were determined by the hyperbolic counterpart
of Algorithm \ref{alg:cpbed_rational_trigonometric_curves}.}%
\label{fig:ed_rational_hpc}%
\end{center}
\end{figure}

\begin{example}
[Applying the hyperbolic counterpart of Algorithm
\ref{alg:cpbed_rational_trigonometric_surfaces}]Using surfaces of the type
(\ref{rational_hyperbolic_surface}), Fig.
\ref{fig:rational_hyperbolic_butterfly} illustrates several control point
configurations for the exact description of the rational hyperbolic surface
patch%
\begin{equation}
\mathbf{s}\left(  u_{1},u_{2}\right)  =\frac{1}{s^{4}\left(  u_{1}%
,u_{2}\right)  }\left[
\begin{array}
[c]{c}%
s^{1}\left(  u_{1},u_{2}\right)  \\
s^{2}\left(  u_{1},u_{2}\right)  \\
s^{3}\left(  u_{1},u_{2}\right)
\end{array}
\right]  ,~\left(  u_{1},u_{2}\right)  \in\left[  0,6\right]  \times\left[
0,10\right]  ,\label{rational_hyperbolic_butterfly}%
\end{equation}
where%
\begin{align*}
s^{1}\left(  u_{1},u_{2}\right)   &  =6\left(  \cosh\left(  2u_{1}-2\right)
+\sinh\left(  u_{2}-5\right)  \right)  ,\\
s^{2}\left(  u_{1},u_{2}\right)   &  =\frac{1}{10}\sinh\left(  u_{1}-1\right)
\cosh\left(  2u_{2}-10\right)  ,\\
s^{3}\left(  u_{1},u_{2}\right)   &  =2\left(  \sinh\left(  2u_{1}-2\right)
+\cosh\left(  2u_{1}-2\right)  \right)  \cosh\left(  u_{2}-5\right)  ,\\
s^{4}\left(  u_{1},u_{2}\right)   &  =275+100\cosh(2u_{1}-2)+25\cosh
(4u_{1}-4).
\end{align*}

\end{example}

\begin{figure}
[!h]
\begin{center}
\includegraphics[
natheight=5.475100in,
natwidth=6.298400in,
height=5.5019in,
width=6.3261in
]%
{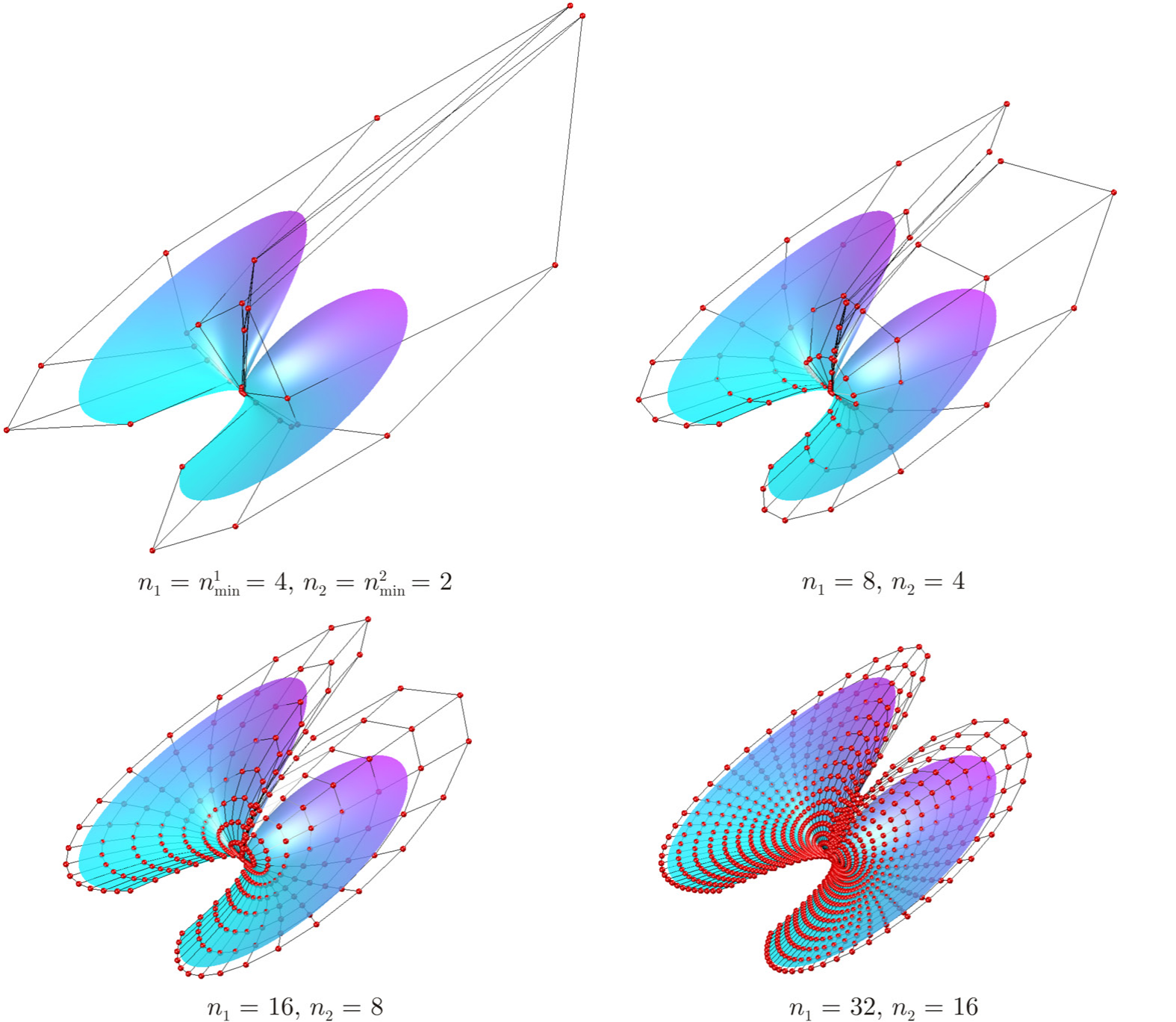}%
\caption{Control point based exact description of the patch
(\ref{rational_hyperbolic_butterfly}) by means of rational hyperbolic surfaces
of the type (\ref{rational_hyperbolic_surface}). Control nets were obtained by
the hyperbolic counterpart of Algorithm
\ref{alg:cpbed_rational_trigonometric_surfaces} ($\delta=2$, $\kappa=1$;
$\alpha_{1}=6$, $\alpha_{2}=10$; $m_1=2$, $m_2=m_3=m_4=1$).}%
\label{fig:rational_hyperbolic_butterfly}%
\end{center}
\end{figure}

\begin{example}
[Hybrid counterpart of Algorithm
\ref{alg:cpbed_rational_trigonometric_surfaces} -- hybrid rational
volumes]Using multivariate rational tensor product surfaces specified by
functions that are exclusively either trigonometric or hyperbolic in each of
their variables, Fig. \ref{fig:rational_hybrid_volume} shows the control point
based exact description of the $3$-dimensional $3$-variate rational surface
element (volume)%
\begin{equation}
\mathbf{s}\left(  \mathbf{u}\right)  =\frac{1}{s^{4}\left(  u_{1},u_{2}%
,u_{3}\right)  }\left[
\begin{array}
[c]{c}%
s^{1}\left(  u_{1},u_{2},u_{3}\right)  \\
s^{2}\left(  u_{1},u_{2},u_{3}\right)  \\
s^{3}\left(  u_{1},u_{2},u_{3}\right)
\end{array}
\right]  ,~\left(  u_{1},u_{2},u_{3}\right)  \in\left[  0,2\right]
\times\left[  0,\frac{3\pi}{4}\right]  \times\left[  0,\frac{\pi}{2}\right]
\label{rational_hybrid_volume}%
\end{equation}
where functions%
\begin{align*}
s^{1}\left(  u_{1},u_{2},u_{3}\right)  = &  \frac{5}{4}\cosh\left(
u_{1}-1\right)  \cos\left(  u_{2}\right)  \left(  \frac{3}{2}+\frac{3}{4}%
\sin\left(  u_{3}\right)  -\frac{1}{2}\cos\left(  u_{3}\right)  -\frac{1}%
{4}\sin\left(  3u_{3}\right)  \right)  ,\\
s^{2}\left(  u_{1},u_{2},u_{3}\right)  = &  \cosh\left(  u_{1}-1\right)
\sin\left(  u_{2}\right)  \left(  \frac{5}{2}+\frac{3}{4}\sin\left(
u_{3}\right)  -\frac{1}{4}\sin\left(  3u_{3}\right)  \right)  ,\\
s^{3}\left(  u_{1},u_{2},u_{3}\right)  = &  -\frac{5}{4}\sinh\left(
u_{1}-1\right)  \left(  \frac{7}{4}+\frac{1}{4}\cos\left(  2u_{3}\right)
\right)  ,\\
s^{4}\left(  u_{1},u_{2},u_{3}\right)  = &  1-\frac{3}{32}\sqrt{2-\sqrt{2}%
}\sin\left(  u_{3}\right)  -\frac{3}{32}\sqrt{2+\sqrt{2}}\cos\left(
u_{3}\right) 
-\frac{1}{32}\sqrt{2+\sqrt{2}}\sin\left(  3u_{3}\right)  -\frac{1}{32}%
\sqrt{2-\sqrt{2}}\cos\left(  3u_{3}\right)
\end{align*}
are hyperbolic or trigonometric in the first and the last two variables, respectively.
\end{example}

\begin{figure}
[!h]
\begin{center}
\includegraphics[
natheight=2.910100in,
natwidth=6.298400in,
height=2.9378in,
width=6.3261in
]%
{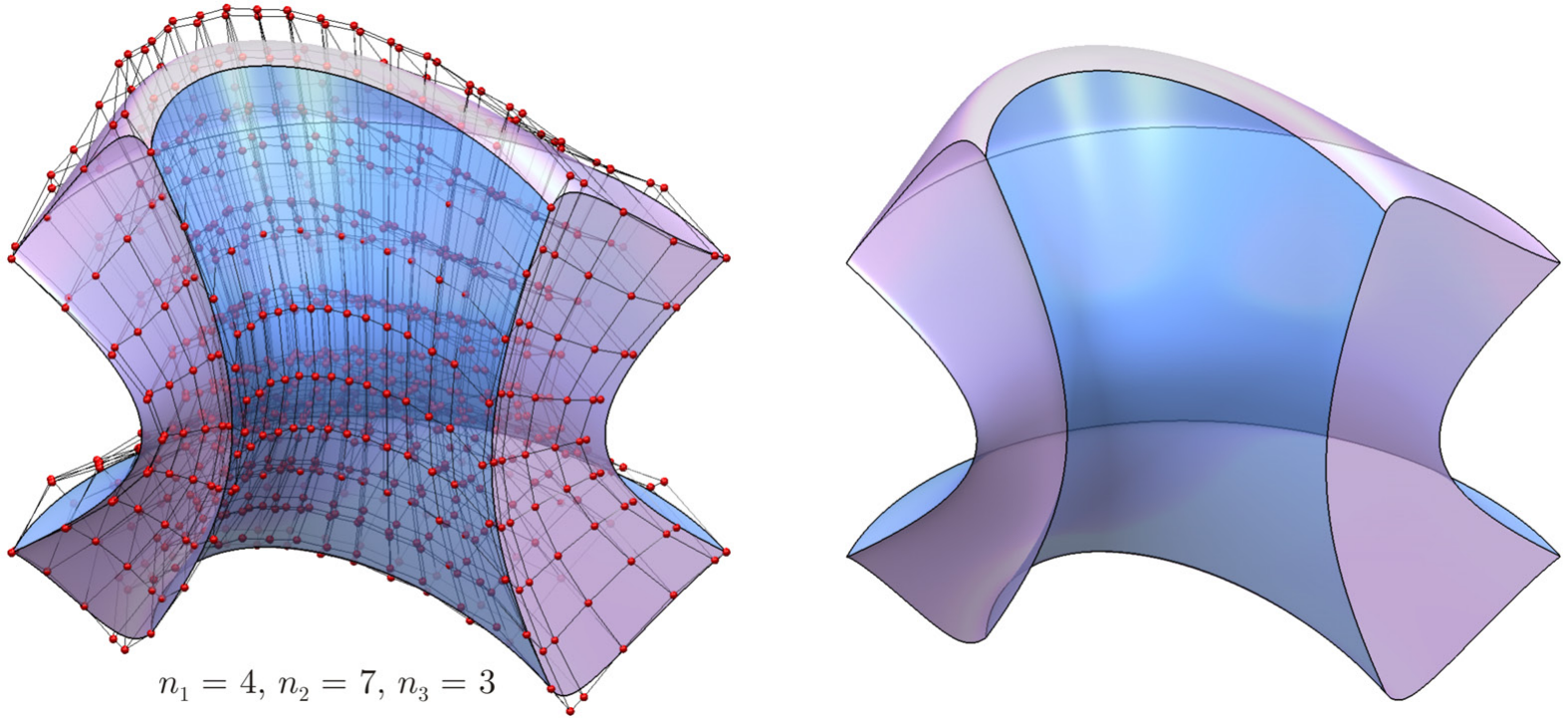}%
\caption{Control point based exact description of the hybrid $3$-dimensional
rational volume element (\ref{rational_hybrid_volume}). The control grid was
calculated by using the hybrid counterpart of Algorithm
\ref{alg:cpbed_rational_trigonometric_surfaces} ($\delta=3$, $\kappa=0$;
$\alpha_{1}=2$, $\alpha_{2}=\frac{3\pi}{4}$, $\alpha_{3}=\frac{\pi}{2}$; $m_1=m_2=m_3=m_4=1$).}%
\label{fig:rational_hybrid_volume}%
\end{center}
\end{figure}

\section{Final remarks\label{sec:final_remarks}}

Subdivision algorithms of trigonometric and hyperbolic curves detailed in
Section \ref{sec:special_parametrizations} can also be easily extended to
higher dimensional multivariate (rational) trigonometric or hyperbolic
surfaces, respectively. Therefore, similarly to standard rational B\'{e}zier
curves and surfaces that are present in the core of major CAD/CAM systems, all
subdivision based important curve and surface design algorithms (like
evaluation or intersection detection) can be both mathematically and
programmatically treated in a unified way by means of normalized B-bases
(\ref{Sanchez_basis}) and (\ref{Wang_basis}). Considering the large variety of
(rational) curves and multivariate surfaces that can be exactly described by
means of control points and the fact that classical (rational) B\'{e}zier
curves and multivariate surfaces are special limiting cases of the
corresponding curve and surface modeling tools defined in Sections
\ref{sec:special_parametrizations} and \ref{sec:multivariate_surfaces}, it is
worthwhile to incorporate all proposed techniques and algorithms presented in
Section \ref{sec:exact_description} into CAD systems of our days.

\begin{acknowledgements}
\'{A}goston R\'{o}th was supported by the European Union and the State of
Hungary, co-financed by the European Social Fund in the framework of
T\'{A}MOP-4.2.4.A/2-11/1-2012-0001 'National Excellence Program'. All assets,
infrastructure and personnel support of the research team led by the author
was contributed by the Romanian national grant CNCS-UEFISCDI/PN-II-RU-TE-2011-3-0047.
\end{acknowledgements}

\end{document}